\documentclass[11pt, reqno]{article}
\usepackage[utf8]{inputenc}
\usepackage{amsmath}
\newcommand{\norm}[1]{\left\lVert#1\right\rVert}
\usepackage{amsfonts, mathrsfs}
\usepackage{mathtools}
\usepackage{amsthm}
\usepackage{amssymb}
\usepackage{graphicx}
\usepackage[affil-it]{authblk}
\usepackage[sc]{mathpazo}
\usepackage{dsfont}
\usepackage{enumerate}
\usepackage{wrapfig}
\usepackage[ruled]{algorithm2e}

\usepackage[dvipsnames]{xcolor}
\newtheorem{theorem}{Theorem}
\newtheorem{lemma}[theorem]{Lemma}
\newtheorem{prop}[theorem]{Proposition}

\newtheorem{remark}[theorem]{Remark}

\newtheorem*{claim*}{Claim}
\newtheorem{claim}[theorem]{Claim}
\newtheorem{defn}[theorem]{Definition}
\newtheorem{definition}[theorem]{Definition}
\newtheorem{condition}[theorem]{Condition}
\numberwithin{equation}{section}
\numberwithin{theorem}{section}

\let\plainqed\qedsymbol
\newcommand{\claimqed}{$\lrcorner$}
\newenvironment{claimproof}{\begin{proof}\renewcommand{\qedsymbol}{\claimqed}}{\end{proof}\renewcommand{\qedsymbol}{\plainqed}}

\usepackage[letterpaper, margin=1in]{geometry}
\usepackage{hyperref}
\usepackage[sort&compress,numbers]{natbib}

\newtheorem{innercustomgeneric}{\customgenericname}
\providecommand{\customgenericname}{}
\newcommand{\newcustomtheorem}[2]{%
  \newenvironment{#1}[1]
  {%
   \renewcommand\customgenericname{#2}%
   \renewcommand\theinnercustomgeneric{##1}%
   \innercustomgeneric
  }
  {\endinnercustomgeneric}
}

\newcustomtheorem{customthm}{Theorem}
\newcustomtheorem{customlemma}{Lemma}
\newcustomtheorem{customproposition}{Proposition}

\newcommand{\pto}{\ensuremath{\xrightarrow{\ \mathbb{P}\ }}}  
\newcommand{\dto}{\ensuremath{\ \xrightarrow{d} \ }}  
\newcommand{\R}{\mathbb{R}}                 
\newcommand{\N}{\mathbb{N}}                 
\newcommand{\PP}{\mathbb{P}}  				
\newcommand{\E}{\mathbb{E}}                 
\newcommand{\irg}{\textsc{irg}}                 

\newcommand{\sel}{\texttt{select}}
\newcommand{\set}{\texttt{set}}
\newcommand{\sett}{\texttt{sett}}
\DeclareMathOperator*{\argmin}{arg\,min}

\newcommand\blfootnote[1]{%
  \begingroup
  \renewcommand\thefootnote{}\footnote{#1}%
  \addtocounter{footnote}{-1}%
  \endgroup
}  

\title{Exploiting Data Locality to Improve Performance of Heterogeneous Server Clusters}

\begin{document}
\author{
  Zhisheng Zhao$^1$, Debankur Mukherjee$^2$, Ruoyu Wu$^3$
}
\renewcommand\Authands{, }
\maketitle
\begin{abstract}
We consider load balancing in large-scale heterogeneous server systems in the presence of data locality that imposes constraints on which tasks can be assigned to which servers.
The constraints are naturally captured by a bipartite graph between the servers and the dispatchers handling assignments of various arrival flows.
When a task arrives, the corresponding dispatcher assigns it to a server with the shortest queue among $d\geq 2$ randomly selected servers obeying the above constraints. 
Server processing speeds are heterogeneous and they depend on the server-type.
For a broad class of bipartite graphs, we characterize the limit of the appropriately scaled occupancy process, both on the process-level and in steady state, as the system size becomes large. 
Using such a characterization, we show that data locality constraints can be used to significantly improve the performance of heterogeneous systems.
This is in stark contrast to either heterogeneous servers in a full flexible system or data locality constraints in systems with homogeneous servers, both of which have been observed to degrade the system performance. Extensive numerical experiments corroborate the theoretical results.

   \blfootnote{$^1$Georgia Institute of Technology, \emph{Email:} \href{mailto:zhisheng@gatech.edu}{zhisheng@gatech.edu}}
\blfootnote{$^2$Georgia Institute of Technology, \emph{Email:}  \href{mailto:debankur.mukherjee@isye.gatech.edu}{debankur.mukherjee@isye.gatech.edu}}
\blfootnote{$^3$Iowa State University, \emph{Email:} \href{mailto:ruoyu@iastate.edu}{ruoyu@iastate.edu}}
\blfootnote{\emph{Acknowledgements.} The work was partially supported by the NSF grant CIF-2113027.}
\end{abstract}

\section{Introduction}
Over the last two decades, large-scale load balancing has emerged as a fundamental research problem. In simple terms, the goal is to investigate \emph{how to efficiently allocate tasks in large-scale service systems, such as data centers and cloud networks?}
As modern data centers continue to process massive amounts of data with increasingly stringent processing time requirements, the need for more efficient and scalable, dynamic load balancing algorithms is greater than ever.
The study of scalable load balancing algorithms started with the seminal works of Mitzenmacher~\cite{Mitzenmacher96, ACMR95, M96focs} and Vvedenskaya et al.~\cite{VDK96}, where the popular `power-of-$d$ choices' or the JSQ($d$) algorithm was introduced. Here a canonical model was considered that consists of $N$ identical parallel servers, each serving a dedicated queue of tasks. 
Arriving tasks are routed to the shortest of $d\geq 2$ randomly selected queues by a centralized dispatcher, 
irrevocably and instantaneously, at the time of arrival.
Since then, this model has received significant attention from the research community and we have seen tremendous progress in our understanding of performance of various algorithms; see~\cite{BBLM21} for a recent survey.

Despite this phenomenal progress, when it comes to modern large-scale systems, many of the existing wisdoms can be observed to be false. This is primarily due to the fact that the above classical model fails to capture two of the most significant factors that impact the performance of these systems:
(a)~\textbf{Data locality constraints:} In simple terms, it means that tasks of a particular type can only be routed to a small subset of servers that are equipped with the appropriate resources to execute them~\cite{WZS20, RM22, TBDH20, TX17}.
For example, an image classification request must be routed to a server that is trained with appropriate machine learning models such as, deep convolutional neural network. 
Also, in online video services like Netflix and YouTube,  users' requests may only be routed to servers that are equipped with the required data (e.g., movies, music). 
The classical model ignores this effect and assumes \emph{full flexibility}, that is, that any task can be assigned to any server in the system.
In the presence of data locality constraints, the delay performance of the system may degrade drastically as compared to fully flexible systems. 
(b)~\textbf{Heterogeneity in service rates:}
Servers in any modern large-scale server clusters do not process tasks at equal speeds. 
This heterogeneity of the service rates is a major bottleneck in implementing the existing heuristics of the classical model. 
For example, if there are two groups of servers in the system, one faster and the other slower, then popular dynamic algorithms like JSQ($d$) that has a provably excellent delay performance when all server speeds are identical, can be observed to be unstable (i.e., their queue lengths blow up)~\cite{GAWD20, HLM21, MKM16, MM16}.
In other words, heterogeneity shrinks the stability region as formally established in~\cite{HLM21}. 
This happens simply because if all the servers are treated equally, then the slower server pool may receive a higher flow of arrivals than what it can process.\\

\noindent
\textbf{Takeaway.}
In summary, both data locality and heterogeneity of server speeds may significantly degrade the system performance. 
The main contribution of the current work is to establish that when these two aspects are considered together, then the performance can in fact be drastically improved. 
That is, if servers are heterogeneous, then efficiently designing the data locality constraints (by appropriately placing the resource files in the server network) can regain the full stability region, which was shrunk for fully flexible systems.
Moreover, we also establish that a carefully designed data locality constraints can ensure the celebrated double-exponential decay of tail probability of the steady-state queue length distribution even for the heterogeneous systems.

\subsection{Our Contributions}
Motivated by this, in the current paper, we consider a bipartite graph model for large-scale load balancing systems, which has recently gained popularity in the research community. 
In this model, a bipartite graph between the servers and task types describes the \emph{compatibility} between the two, where an edge represents the server's ability to process the corresponding task type.
This encompasses the classical \emph{full-flexibility} models as those having a complete bipartite compatibility graph. 
An immediate difficulty of the new model is that when the graph is non-trivial (i.e., not a collection of isolated pairs or a complete bipartite graph), the mean-field techniques break down. This is because, 
the queues no longer remain exchangeable, making the aggregate processes such as the vector of number of servers with queue length $i$ with $i = 0, 1,2,\ldots$ non-Markovian.
In addition, we also consider that each dispatcher handles the arrival flow of one of $K$ possible task types and that there are $M$ server types.
The rate of service at a server depends on its type.
Throughout the paper, the key quantity of interest will be the \emph{global occupancy process} $\mathbf{q}^N(t)=(q^N_{m,l}(t),m=1,\ldots,M,l\geq 1)$, where $q^N_{m,l}(t)$ represents the fraction of servers of type $m$ with queue-length at least $l$ at time~$t$ in the $N$-th system with $N$ servers, and we will look at the large-system asymptotic regime: $N\to\infty$.

Due to the compatibility constraints, the servers become non-exchangeable, even if they belong to the same type. 
This causes most of the existing frameworks~\cite{Dai95,Mitzenmacher96,AS15} to break down.
To characterize the process-level limit of the queue length process, we take resort to the theory of weakly interacting particle systems and asymptotically couple the evolution of the $N$-dimensional vector of queue lengths with an appropriately defined infinite system of independent \emph{McKean-Vlasov processes}~\cite{SA91,Meleard1996asymptotic}.
We also show the asymptotic independence of any finite number of queue length processes, also know as the \emph{propagation of chaos} property.
This convergence of the queue length processes (in $L_2$ sense) is then used to establish the transient convergence of the occupancy process.
One downside of the above convergence is that it depends on the assumption that the initial queue lengths within each set of servers of the same type are independent and identically distributed (i.i.d.) and are independent across the the set of servers of different types.
Due to this assumption, this convergence result cannot be used to establish the interchange of $t\to\infty$ and $N\to\infty$ limits, which is crucial in studying the limit of steady states.

To overcome this issue, we use the framework of~\cite{RM22}, recently introduced in the context of homogeneous systems. Here, a notion called \emph{proportional sparsity} for graph sequences was introduced, which ensures that the empirical queue length distribution within the set of compatible servers of any dispatcher is close to the empirical queue length distribution of the entire system. 
This was used in~\cite{RM22} to construct conditions on graphs that match the performance of a fully flexible system. 
In the current setup, however, this notion is inadequate, since our goal is not to match the performance of the fully flexible system (which is usually poor under heterogeneity).
That is why, we extend this notion to what we call the \emph{clustered proportional sparsity} for a sequence of graphs with increasing size, to accommodate the heterogeneous systems. 
The clustered proportional sparsity property allows us to construct a stochastic coupling between the system and another intermediate system whose task allocation is done by a carefully constructed algorithm called GWSQ($d$) (Algorithm~\ref{alg:GWQD-JSQ(d)}).
This coupling with the intermediate system, along with clustered proportional sparsity, helps us establish that if the initial occupancy of two systems are close, then 
the distance (in the $\ell_1$-norm) between their global occupancy remains small uniformly over any finite time interval. 
In turn, it implies that their limits of the global occupancy systems are the same.
As a consequence, we can remove the i.i.d.~assumption of the initial queue lengths, since the above guarantees that under clustered proportional sparsity, the convergence of the occupancy process depends only on the initial occupancy and \emph{not} on how the individual queues are distributed.

The above process-level limit result shows that the transient limit of the occupancy process can be described as a system of ODEs that depend on various graph parameters. Next, we also show that the interchange of limits holds and that the sequence of occupancy states in stationarity, converges weakly to the unique fixed point of the ODE.
One celebrated feature of the classical JSQ($d$) policy for homogeneous systems under full flexibility is that the steady-state queue length decays doubly exponentially as $\lambda^{(d^i-1)/(d-1)}$, where $\lambda\in (0,1)$ is the load per server~\cite{Mitzenmacher96, VDK96}.
We establish this double-exponential decay property for the heterogeneous system. 

It is worthwhile to note that the strength of the above results lie in that they hold for arbitrary deterministic sequence of graphs satisfying certain properties. 
However, we show that all these properties are satisfied almost surely by a sequence of inhomogeneous random graphs with parameters prescribed by the theorems. 
This makes it easy to design graphs with the desired favorable properties.

\subsection{Related Works}
The research on task allocation systems with limited flexibility can be traced back to the works of Turner~\cite{Turner98} and Foss and Chernova~\cite{FC98}.
Of particular importance to the current work, Foss and Chernova~\cite{FC98} considered stability properties of the system using the fluid model.
Later, Bramson~\cite{Bramson11a} generalized some parts of results in \cite{FC98} to a broad class of Join-Shortest-Queue (JSQ)-type systems, including the JSQ($d$) policy, via the Lyapunov function approach.
Stolyar~\cite{AS05} considered optimal routing in output-queued flexible server system, which is essentially the bipartite graph model for the load balancing system. 
Here the author considered a system with a fixed number of servers and dispatchers in the conventional heavy traffic regime and proposed a routing policy that is optimal in terms of server workload.
Recently, Cruise et al.~\cite{CJS20} considered load balancing problems on hypergraphs and proved its stability conditions. 
The above works, however, did not aim to precisely characterize the system performance in the large-scale scenario.

The analysis in the large-scale scenario became prominent in the last decade, with the emergence of its applications to load balancing in data centers and cloud networks.
In the full-flexibility setup, the analysis of heterogeneous-server systems gained some attention. 
In this case, Stolyar~\cite{AS15, Stolyar17} studied the zero-queueing property of the Join-Idle-Queue (JIQ) policy, Mukhopadhyay et al.~\cite{MKM16} and Mukhopadhyay and Mazumdar~\cite{MM16}
analyzed the JSQ($d$) policy in heterogeneous systems with processor-sharing service discipline, Hurtado-Lange and Maguluri~\cite{HLM21} studied the throughput and delay optimality properties of JSQ($d$), and Bhambay and Mukhopadhyay~\cite{BM22} studied a speed-aware JSQ policy.
The above works on the JSQ($d$) policy observe that the stability region shrinks if the dispatcher applies the JSQ($d$) policy blindly.
One way to mitigate this performance degradation is to take the server speeds into consideration while sampling servers or while assigning tasks to the sampled servers.
Such a `hybrid JSQ($d$)' scheme is able to recover the stability region.
The current work can be contrasted with this approach. 
First, in the presence of data locality, both the server speeds and the underlying compatibility constraints need to be taken into account during the sampling procedure, and the approach becomes significantly more complicated.
Second, we show how exploiting the data locality, the blind JSQ($d$) policy can recover the stability region and even achieve the double-exponential decay of tail probabilities of the stead-state queue length distribution. 
One advantage of the latter approach is that the dispatchers can be oblivious to the server speeds, which reduces the implementation complexity and also, makes it robust against changes to the servers (e.g., when servers are add/removed).

Recently, Allmeier and Gast~\cite{AG22} studied the application of (refined) mean-field approximations for heterogeneous systems. Their method is using an ODE to approximate the evolution of each server, and the error vanishes as the system scales. However, this method cannot be directly used in our case. Due to the bipartite compatibility graph structure, it is hard to capture the interactions between two servers, which means that we cannot write the transition rates of the underlying Markov chain as~\cite{AG22} does. Also, one important assumption in their work is the finite buffer, but we consider the infinite buffer case here.

The aspect of task-server compatibility constraints in large-scale load balancing and scheduling gained popularity only recently, as the data locality became prominent in data centers and cloud networks.
This led to many works in this area~\cite{NG15,MBL17,BMW17,RM22,WZS20,TX17, TX13}.
All these works consider homogeneous processing speeds at the servers.
The initial works~\cite{Turner98,NG15} focused on certain fixed-degree graphs and showed that the flexibility to forward tasks to even a few neighbors with possibly shorter queues may significantly improve the waiting time performance as compared to dedicated arrival streams or a collection of independent M/M/1 queues.
Tsitsiklis and Xu~\cite{TX13, TX17} considered asymptotic optimality properties of the bipartite graph topology in an input-queued, dynamic scheduling framework.
Later, in the (output-queued) load balancing setup, Mukherjee et al.~\cite{MBL17} considered the JSQ policy and Budhiraja et al.~\cite{BMW17} considered the transient analysis of the JSQ($d$) policy on non-bipartite graphs. 
The goal in these papers was to provide sufficient condition on the graph sequence to asymptotically match the performance of a complete graph.
Here we should mention that the non-bipartite graph model cannot be used to capture the data locality constraints.
In the presence of data locality constraints, the analysis of the JSQ($d$) policy for homogeneous systems, including both transient and interchange of limits, was performed by Rutten and Mukherjee~\cite{RM22}. 
Weng et al.~\cite{WZS20} is the first to consider the large-scale heterogeneous-server model under data locality. 
They showed that the Join-the-Fastest-Shortest-Queue (JFSQ) and Join-the-Fastest-Idle-Queue (JFIQ) policies achieve asymptotic optimality for minimizing mean steady-state waiting time when the bipartite graph is sufficiently well connected.
However, these results fall in the category of JSQ-type policies where the asymptotic behavior is degenerate in the sense that the queue lengths at servers can be either 0 or 1. 
Naturally, the results and their analysis are very different from the JSQ($d$)-type policies where queues of any length is possible.

\subsection{Notations}
Let $\N_0=\N\cup\{0\}$. For a set $S$, its cardinality is denoted as $|S|$. For a polish space $\mathcal{S}$, the space of right
continuous functions with left limits from $[0,\infty)$ to $\mathcal{S}$ is denoted as $\mathbb{D}([0,\infty),\mathcal{S})$, endowed with the Skorokhod topology.
The distribution of $\mathcal{S}$-valued random variable $X$ will be denoted as $\mathcal{L}(X)$.
For a function $f:\ [0,\infty)\rightarrow\mathbb{R}$, let $\norm{f}_{*,t}\coloneqq\sup_{0\leq s\leq t}|f(s)|$. The distribution of $\mathcal{S}$-valued random variable $X$ will be denoted as $\mathcal{L}(X)$. For $x\in\mathcal{S}$, the Dirac measure at the point $x$ is denoted as $\delta_x$. 
$\norm{\cdot}_p$ represents the $\ell_p$-norm. Define ${X\choose Y}=\frac{X(X-1)\cdots(X-Y+1)}{Y!}$ if $X\geq Y$ and is 0, otherwise.
RHS is the acronym of Right Hand side. 

\section{Model Description}
\label{sec:model-description}
The model below for large-scale systems with limited flexibility was considered by Tsitsiklis and Xu~\cite{TX13, TX17} in the context of scheduling algorithms for input-queued systems. 
Subsequently, it was considered in~\cite{MBL17, BMW17, RM22, WZS20} for output-queued load balancing systems.
Let $G^N=(W^N,V^N,E^N)$ be a system with $N$ single servers, each serving its own queue, and $W(N)$ dispatchers, where $W^N = \{1,...,W(N)\}$ and $V^N=\{1,...,N\}$ denote the sets of dispatchers and servers, respectively.
Similar to~\cite{TX13, TX17}, we assume that $\lim_{N\rightarrow\infty}W(N)/N=\xi$ where $\xi>0$ is a constant.
The set $E^N\subseteq W^N\times V^N$ of edges represents hard compatibility between the dispatchers and  servers in the $N$-th system. In other words, tasks of type $i$ can be assigned to a server $j$ if and only if $(i,j)\in E^N$.
Tasks arriving at a dispatcher must be assigned  instantaneously and irrevocably to one of the \emph{compatible} servers.
\begin{itemize}
    \item \textbf{Task types:} 
    A task can be of one of $K$ possible types labelled in $\mathcal{K}=\{1,...,K\}$ and each dispatcher handles arrivals of exactly one task type.
    Thus, we will interchangeably use the terms task-type and dispatcher-type throughout the article.
    Let $W^N_k$ denote the set of all dispatchers handling type-$k$ tasks.
    As $N\to\infty$, assume that $|W^N_k|/W(N) \to w_k\in (0,1)$ for $k\in \mathcal{K}$ with $\sum_{k=1}^K w_k = 1$. 
    Tasks arrive at each dispatcher as an independent Poisson process with rate~$\lambda$.
    \item \textbf{Server types:} Each server belongs to one of $M$ possible types labelled in $\mathcal{M}=\{1,...,M\}$. Let $V^N_m$ denote the set of type-$m$ servers, and
    as $N\to\infty$, $|V^N_m|/N\to v_m\in (0,1)$ for $m \in \mathcal{M}$ with $\sum_{m=1}^M v_m = 1$.
    \item \textbf{Service times:} The processing time at a type-$m$ server is exponentially distributed with mean $1/u_m$, where $u_m$ is a positive constant. 
    Throughout, we will assume that asymptotically, the system has sufficient service capacity in the sense that
    \begin{equation}\label{eq:capacity}
    \lambda\xi<\sum_{m\in\mathcal{M}}u_mv_m.    
    \end{equation}
    Note that the left and right hand side above represents the scaled total arrival rate and scaled maximum departure rate, respectively.
    \end{itemize}

For all the asymptotic results, we consider a general class of systems where the compatibility graph satisfies certain asymptotic criteria as specified in Condition~\ref{cond-1} below. 
Define 
\begin{align*}
    \deg^N_w(i,m) & =|\{j\in V^N_m:(i,j)\in E^N\}|,\quad i\in W^N,m\in \mathcal{M}, \\
    \deg^N_v(k,j) & =|\{i\in W^N_k:(i,j)\in E^N\}|,\quad j\in V^N,k\in\mathcal{K}.
\end{align*}
Namely, $\deg^N_w(i,m)$ is the number of the dispatcher $i$'s neighboring servers whose type is $m\in\mathcal{M}$. Similarly, $\deg^N_v(k,j)$ is the number of the server $j$'s neighboring dispatchers whose type is $k\in\mathcal{K}$.

\begin{condition}\label{cond-1}
    The sequence $\{G^N\}_{N\geq 1}$ satisfies the following:
    \begin{enumerate}[{\normalfont (a)}]
        \item For each $k\in\mathcal{K}$ and $m\in \mathcal{M}$, let $E^N(k,m)=\{(i,j)\in W^N_k\times V^N_m:(i,j)\in E^N\}$ 
        \begin{equation}
            \lim_{N\rightarrow\infty} \frac{|E^N(k,m)| }{|W^N_k|\times|V^N_m|}=p_{k,m}\in [0,1].
        \end{equation}
        We call the matrix $\mathbf{p}=(p_{k,m},k\in\mathcal{K},m\in\mathcal{M})$ as the compatibility matrix.
        \item For each $k\in \mathcal{K}$ and $m\in \mathcal{M}$, 
        \begin{equation*}
            \lim_{N\rightarrow \infty}\frac{\max_{i\in W^N_k}\deg^N_w(i,m)}{\min_{i\in W^N_k}\deg^N_w(i,m)}=1, \quad
            \lim_{N\rightarrow \infty}\frac{\max_{j\in V^N_m}\deg^N_v(k,j)}{\min_{j\in V^N_m}\deg^N_v(k,j)}=1.
        \end{equation*}
    \end{enumerate}
\end{condition}
Intuitively, the condition implies that the `asymptotic density' of edges between type-$k$ dispatchers and type-$m$ servers is given by $p_{k,m}$ and for each task-type-server-type pair, the servers have similar levels of flexibility. 
The classical, well-studied setup where any task can be processed by any server, corresponds to the complete bipartite graph with $p^N_{k,m}=1$, $\forall k\in\mathcal{K},m\in\mathcal{M}$. 
In Section~\ref{sec:rand}, we show that for any given $\mathbf{p}:=(p_{k,m},k\in\mathcal{K},m\in\mathcal{M})$, a sequence of graphs satisfying Condition~\ref{cond-1} can be obtained simply by putting edges suitably randomly. 
This is a certain class of \emph{inhomogeneous random graphs}, which we call \irg($\mathbf{p}$); see Definition~\ref{def:irg} for details.
In fact, the \irg($\mathbf{p}$) sequence of graphs will be proved to satisfy the required conditions for all the results of this article to hold.
\\

\noindent
\textbf{State Space.}
In the $N$-th system, let $X^N_j(t)$ be the number of tasks (including those in service) in the queue of server $j\in V^N$ at time $t$. 
Let $q^N_{m,l}(t)$ be the proportion of servers of type $m$ with queue length at least $l$ at time $t$, namely,
\begin{equation}\label{eq:defn-qN-mj}
    q^N_{m,l}(t)\coloneqq \frac{1}{|V^N_m|}\sum_{j\in V^N_m}\mathds{1}_{\big(X^N_j(t)\geq l\big)},\quad t\geq 0, m\in \mathcal{M},l\in\N_0.
\end{equation}
Let $\mathbf{q}^N(t)=\big(q^N_{m,l}(t),m\in\mathcal{M},l\in\N_0\big)$. Then $\mathbf{q}^N\coloneqq\big\{\mathbf{q}^N(t)\big\}_{0\leq t<\infty}$ is a process with sample paths in $\mathbb{D}([0,\infty),\mathcal{S})$ where
$$\mathcal{S}\coloneqq\Big\{\mathbf{q}\in[0,1]^{M\times \N_0}:q_{m,0}=1,q_{m,l}\ge q_{m,l+1}, \text{ and }\sum_{l\in\N_0}q_{m,l}<\infty,\forall m\in\mathcal{M}, l\in\N_0\Big\}$$
is equipped with the $\ell_1$-topology. 
Note that the space $\mathcal{S}$ is a complete metric space.\\

\noindent
\textbf{Local JSQ($d$) Policy.} 
For any fixed $d\geq 2$, each dispatcher uses the JSQ($d$) policy~\cite{Mitzenmacher96, VDK96} to assign the incoming tasks to servers.
To describe the policy, define the \emph{neighborhood of dispatcher $i\in W^N$}, $\mathcal{N}^N_w(i):=\{j\in V^N:(i,j)\in E^N\}$ with $\delta^N_i=|\mathcal{N}^N_w(i)|$. 
When a new task arrives at the dispatcher $i\in W^N$ with $\delta^N_i\geq d$, it is immediately assigned to the server with the shortest queue among $d$ servers selected uniformly at random from $\mathcal{N}^N_w(i)$. Ties are broken uniformly at random. 
If $\delta^N_i<d$, then the task is assigned to one server selected from $\mathcal{N}^N_w(i)$ uniformly at random.
This $\delta^N_i<d$ scenario is asymptotically not relevant for us since all the graphs that we will consider have diverging degrees as $N\to\infty$.

\section{Main Results}

\subsection{Mitigating the Stability Issue}\label{sec:stability-issue}
As discussed earlier, when the server speeds are heterogeneous, the fully flexible systems (with the complete bipartite compatibility graph) may not be stable under the JSQ($d$) policy, even if we assume that the sufficient service capacity in~\eqref{eq:capacity} is satisfied. 
The next lemma provides a necessary and sufficient condition for ergodicity of the queue length process. 
Recall $\delta^N_i=|\mathcal{N}^N_w(i)|$. For any fixed $N$, define 
    \begin{equation}\label{eq:stability-cond}
    \rho^N\coloneqq\max_{\substack{U\subseteq V^N\\ U\neq \emptyset}}\Big\{\Big(\sum_{j\in U}\sum_{m\in\mathcal{M}}\mathds{1}_{(j\in V^N_m)}u_m\Big)^{-1}\sum_{i\in W^N}\Big(\mathds{1}_{(\delta^N_i\geq d)}\sum_{\substack{S\subseteq(U\cap \mathcal{N}^N_w(i)):\\|S|=d}}\frac{\lambda}{{\delta^N_i\choose d}}+\mathds{1}_{(\delta^N_i< d)}\frac{|U\cap \mathcal{N}^N_w(i)|}{\delta^N_i}\Big)\Big\}.
    \end{equation}
\begin{lemma}\label{lem:stability}
The queue length process $\big(X_j^N(t)\big)_{j\in V^N}$ under the local JSQ($d$) policy is ergodic if and only if $\rho^N<1$.
\end{lemma}
The above lemma is an immediate consequence of \cite[Theorem~2.5]{FC98}; see also~\cite{Bramson11a}.
We omit its proof.
Intuitively, $\rho^N<1$ means that in the $N$-th system, for any subset $U$ of servers with possibly long
queues (compared to the rest servers), the total rate at which tasks are assigned to some server in this set must be less than the rate of departure from this set. 

Since we are interested in large-$N$ behavior, we will assume certain asymptotic version of the above stability criterion.
This is fairly standard in the large-system analysis, as one would want to avoid the `heavy-traffic' regime when $\rho^N\uparrow 1$ as $N\to\infty$.
The behavior in the latter scenario is typically qualitatively different from the so-called `subcritical' regime as defined below.

\begin{defn}[Subcritical Regime]
\label{def:subcrit}
The sequence $\{G^N\}_N$ of systems defined as above is said to be in the \textit{subcritical regime} with asymptotic load $\rho<1$ if
$\rho^N\to \rho<1$, as $N\to \infty$.
\end{defn}
Throughout this paper we will assume that the sequence of systems under consideration is in subcritical regime.
From Lemma~\ref{lem:stability}, it is immediate that if a sequence of systems is in subcritical regime, then its queue length process is ergodic for all large enough $N$.
The potential non-ergodicity of fully flexible, heterogeneous server clusters brings us to the question that when the sufficient service capacity in~\eqref{eq:capacity} is satisfied, whether we can design the underlying compatibility structure carefully so that the queue length process is ergodic.
In other words, \emph{can we regain the stability region?}
Proposition~\ref{prop:p-matrix-existence} below shows that this is indeed the case.
In some sense, this highlights the first-order improvements (i.e., in terms of stability properties) of a careful compatibility structure design in contrast to a fully flexible system.
\begin{prop}\label{prop:p-matrix-existence}
Let the parameters $\lambda,\xi,d$ and $w_k,v_m,u_m$, $k\in\mathcal{K}$, $m\in\mathcal{M}$, be such that~\eqref{eq:capacity} is satisfied.
Then there exists $(p_{k,m})_{k\in\mathcal{K}, m\in\mathcal{M}}\in [0,1]^{K\times M}$ such that, for any sequence of systems $\{G^N\}_{N\geq 1}$ satisfying Condition~\ref{cond-1}, the queue length process $\big(X_j^N(t)\big)_{j\in V^N}$ is ergodic for all $N$ large-enough. 
Moreover, such a $(p_{k,m})_{k\in\mathcal{K}, m\in\mathcal{M}}$ can be obtained explicitly by solving a set of inequalities.
\end{prop}
In the following sections, we will demonstrate, in addition to the first-order improvements, how asymptotic queue length distribution can be improved as well, for example, in terms of having a double-exponential decay of tail probabilities.

The proof of Proposition~\ref{prop:p-matrix-existence} is provided in Appendix~\ref{app:proof-lem-stability}. 
It relies on first building a simple criteria involving the system parameters, which, for sequence of systems satisfying Condition~\ref{cond-1}, ensures stability for all large-enough $N$ (Lemma~\ref{lem:relationship-v-u-pkm}). 
Then we show that given other parameters, a value of $(p_{k,m})_{k\in\mathcal{K}, m\in\mathcal{M}}$ satisfying this criteria can be found by checking the feasibility region defined by $M$ inequalities.
\vspace{2mm}

We end this subsection by presenting the above-mentioned simple asymptotic criteria for subcriticality. The proof is given in Appendix~\ref{app:proof-lem-stability}.
Denote $\delta_k\coloneqq\sum_{m\in\mathcal{M}}p_{k,m}v_m$ for each $k\in\mathcal{K}$.
\begin{lemma}
\label{lem:relationship-v-u-pkm}
Let $\{G^N\}_N$ be a sequence satisfying Condition~\ref{cond-1}. 
The sequence of systems is in subcritical regime if
\begin{equation}\label{eq:max-rho-set}
    \frac{\lambda\xi}{u_m}\sum_{k\in\mathcal{K}}\frac{w_k p_{k,m}}{\delta_k} <1,\quad \text{ for all } m\in\mathcal{M}.
\end{equation}
\end{lemma}

\subsection{Process-level Limit: IID Case}\label{sec:iid-transient}
Our first main result characterizes the process-level limit of the queue-length process $\big(X^N_j$, $j\in V\big)$, as $N\to \infty$, when the starting states $\big\{X^N_{j}(0):j\in V^N_m\big\}$ are i.i.d.~for all $m\in\mathcal{M}$ and independent across different $m$-values. 
When the sequence of graphs $\{G^N\}_N$ satisfies a stronger condition, called \emph{clustered proportional sparsity} (Definition~\ref{defn:prop-sparse}), the i.i.d.~condition can be removed. This is the content of Section~\ref{ssec:process-gen}.

Now, note that for a fixed $N\geq 1$, $\{X^N_j : j\in V_N\}$ is a system of $N$ stochastic processes with mean-field type interactions. 
Exploiting tools from the theory of weakly interacting particles, we show in Theorem~\ref{thm:weak-conv-xj} below that as the system size becomes large, queue-length processes converge weakly to those of an infinite system of independent \emph{McKean-Vlasov processes} $\{X_j : j\in \mathbb{N}\}$ (see e.g.\ \cite{SA91,Meleard1996asymptotic}). 
In fact, using a suitable coupling to be described in more details in Section~\ref{ssec:aux}, the convergence holds in $L_2$.
For ease of describing such processes and coupling, although we only assumed that certain fractions of servers are of certain task types in the model description, it will be convenient to fix the type of each server $j\in\N$ in this subsection, by defining a \emph{membership map} $\mathbf{M}: \N\rightarrow \mathcal{M}$, so that $V^N_m=\{j\in V^N:\mathbf{M}(j)=m\}$ with $\lim_{N\rightarrow\infty}\frac{|V^N_m|}{N}=v_m$ and $V_m=\lim_{N\rightarrow\infty}V^N_m$ for each $m\in\mathcal{M}$.
With such fixed server types and $X_j^N(0) \equiv X_j(0)$, let
\begin{align}
    X_j(t) & =X_j(0)-\int_0^t\mathds{1}_{\big(X_j(s-)>0\big)}D_j(ds)+\int_{[0,t]\times \R_+}\mathds{1}_{\big(0\leq y\leq C_j(s-)\big)}A_j(dsdy), \label{eq:thm-1-1} \\
    C_j(t) & =d\xi\sum_{k\in \mathcal{K}}\frac{p_{k,m}w_k}{\delta_k}\sum_{(M_2,...,M_d)\in \mathcal{M}^{d-1}} h_t(j,M_2,...,M_d), \label{eq:thm-1-2}
\end{align}
where $\mathbf{M}(j)=m$ and
\begin{align}
    h_t(j,M_2,...,M_d) & =\prod_{h=2}^d \frac{v_{M_h}p_{k,M_h}}{\delta_k}\int_{\N^{d-1}}b\big(X_j(t),x_{j_2},...,x_{j_d}\big)\mu^{M_2}_t(dx_{j_2})\cdots\mu^{M_d}_t(dx_{j_d}), \notag \\
    b(\mathbf{x}) & =b(x_1,...,x_d)\coloneqq\sum_{r=1}^d \frac{1}{r}\mathds{1}_{\big(x_1=\min_{j\in[d]}\mathbf{x},|\argmin \mathbf{x}|=r\big)}, \quad \mathbf{x}=(x_1,...,x_d)\in \N_0^d, \label{eq:b} \\
    \mu^{m}_t & =\mathcal{L}\big(X_{i}(t)\big), \quad \forall\, i\in V_{m}, m \in \mathcal{M}, t\geq 0. \notag
\end{align}
Here, $\{D_j : j \in V_m\}$ are i.i.d.\ Poisson processes with rate $u_m$ for each $m \in \mathcal{M}$, $\{A_j : j \in \mathbb{N}\}$ are i.i.d.\ Poisson random measures on $[0,\infty)\times \R_+$ with intensity $\lambda dsdy$, and all $D_j$'s and $A_j$'s are independent.
Loosely speaking, $A_j$ corresponds to the arrival processes and $D_j$ corresponds to the departure processes at servers. 
We note that the existence and uniqueness of solutions to \eqref{eq:thm-1-1} and \eqref{eq:thm-1-2} can be proved by standard arguments (see e.g., \cite{SA91,Meleard1996asymptotic}) using the boundedness and Lipschitz property of the functions $b$ and $x \mapsto \mathds{1}_{(x>0)}$ on $\mathbb{N}_0$.

\begin{theorem}[Convergence to McKean-Vlasov process and propagation of chaos]\label{thm:weak-conv-xj}
Consider any fixed $\mathbf{q}^{\infty}=\big(q^{\infty}_{m,l},m\in\mathcal{M},l\in\N_0\big)\in \mathcal{S}$. 
Assume that all $X^N_j(0)$'s are independent, and for each $m\in\mathcal{M}$, $\big\{X^N_{j}(0):j\in V^N_m\big\}$ is i.i.d.~with $\PP\big(X^N_j(0)\geq l\big)=q^{\infty}_{m,l}$, $l\in\N_0$. 
On any finite time interval $[0,T]$, $T>0$, for any $m\in \mathcal{M}$ and $j\in V_{m}$, the queue length process $X^N_j(\cdot)$ at server $j$ weakly converges to the process $X_j(\cdot)$ in \eqref{eq:thm-1-1}. 
In fact, one can suitably couple $X^N_j$ with $X_j$ such that
\begin{equation}\label{eq:thm1-coupled-cvg}
    \max_{j\in V^N}\E\norm{X^N_j-X_j}^2_{*,T}\xrightarrow{N\rightarrow\infty}0
\end{equation}
and hence the propagation of chaos property holds, that is, for any $n \in \mathbb{N}$ and distinct $j_h \in V_{M_h}$, $h=1,\dotsc,n$,
\begin{equation}\label{eq:thm1-POC}
    \mathcal{L}(X_{j_1}^N,\dotsc,X_{j_n}^N) \xrightarrow{N\rightarrow\infty} \mathcal{L}(X_{j_1},\dotsc,X_{j_n}) = \mu^{M_1}\otimes\dotsb\otimes\mu^{M_n}.
\end{equation}
\end{theorem}

Theorem~\ref{thm:weak-conv-xj} gives us the limit law of all individual queues.
Next, in Theorem~\ref{thm:conv-global-occup}, we will show how such a server-level convergence can be used to obtain a convergence result for the global occupancy process $\mathbf{q}^N(\cdot)$ to a deterministic dynamical system, which was our primary goal. 
The proofs of Theorem~\ref{thm:weak-conv-xj} and Theorem~\ref{thm:conv-global-occup} are provided in Section~\ref{sec:transient-limit}.

\begin{theorem}[Process-level convergence for i.i.d.~starting state]\label{thm:conv-global-occup}
Assume that all $X^N_j(0)$'s are independent, and for each $m\in\mathcal{M}$, $\{X^N_{j}(0):j\in V^N_m\}$ is i.i.d. with $\PP(X^N_j(0)\geq l)=q^{\infty}_{m,l}$, $l\in \N_0$ for some $\mathbf{q}^{\infty}=(q^{\infty}_{m,l},m\in\mathcal{M},l\in\N_0)\in \mathcal{S}$. 
Then on any finite time interval, the occupancy process $\mathbf{q}^N(\cdot)$ converges weakly with respect to Skorokhod $J_1$ topology to the deterministic limit $\mathbf{q}(\cdot)\coloneqq(q_{m,l}(\cdot),m\in\mathcal{M},l\in\N_0)$ given by the unique solution to the following system of ODEs: 
For all $m\in\mathcal{M}$, $q_{m,0}(t)=1$, $q_{m,l}(0)=q^{\infty}_{m,l}$, and
\begin{align}
    \frac{dq_{m,l}(t)}{dt}&=-u_m\big(q_{m,l}(t)-q_{m,l+1}(t)\big)\nonumber\\
    &\quad +\lambda\xi\big(q_{m,l-1}(t)-q_{m,l}(t)\big)\sum_{k\in\mathcal{K}}\frac{p_{k,m}w_k}{\delta_k}\frac{(\tilde{q}_{k,l-1}(t))^d-(\tilde{q}_{k,l}(t))^d}{\tilde{q}_{k,l-1}(t)-\tilde{q}_{k,l}(t)},\quad \forall l\in \N.\label{eq:ode-sys-2}
\end{align} 
where $\tilde{q}_{k,l}(t)=\sum_{m\in\mathcal{M}}\frac{v_mp_{k,m}}{\delta_k}q_{m,l}(t)$ for all $k\in\mathcal{K}$.
\end{theorem}

\begin{remark}\normalfont
    Using the propagation of chaos property \eqref{eq:thm1-POC} and the fact that $\{X_j(t) : j \in \mathbb{N}\}$ is independent and $\{X_j(t) : j \in V_m\}$ is i.i.d.\ for each $m \in \mathcal{M}$, it follows that the limit of the global occupancy process at any time instant $t$, in fact, corresponds to the laws of $X_j(t)$ for each type of servers $j$ in \eqref{eq:thm-1-1}, that is,
    \begin{equation*}
        \mu_t^m[l,\infty) = \mathbb{P}(X_j(t) \ge l) = q_{m,l}(t), \quad j \in V_m, m \in \mathcal{M}, l \in \mathbb{N}_0, t \ge 0.
    \end{equation*}
\end{remark}

\subsection{Process-level Limit: General Case}\label{sec:general-transient}
\label{ssec:process-gen}

Theorem~\ref{thm:conv-global-occup} requires the strong assumption that for each $m\in\mathcal{M}$, $X^N_j(0)$, $j\in V^N_m$, are i.i.d.. 
In order to argue the interchange of limits, we need to relax this assumption on initial states. 
This is because the arguments for the interchange of limits involves initiating the prelimit system at the steady state and then showing that as $N\to\infty$, the system must converge to the unique fixed point of the limiting ODE. 
The above requires us to characterize the (process-level) limiting trajectory of the system starting from arbitrary occupancy state. We achieve this in this section.

Intuitively, the assumption of i.i.d.~in Theorems~\ref{thm:weak-conv-xj} and~\ref{thm:conv-global-occup} ensures that the local occupancy observed by any dispatcher $i\in W^N_k$, $k\in\mathcal{K}$ is `close', in suitable sense, to the average occupancy at the entire system. 
This phenomenon can be ensured asymptotically, even without the i.i.d.~assumption if the graph sequence satisfies a property we call the \emph{clustered proportional sparsity}.
This notion was first introduced for the homogeneous systems in~\cite{RM22}. 
The definition below is a modified notion that is suitable for the current heterogeneous setting.

\begin{defn}[Clustered Proportional Sparsity]\label{defn:prop-sparse}
Recall $\mathcal{N}^N_w(i)=\{j\in V^N:(i,j)\in E^N\}$.
The sequence $\{G^N\}_N$ is called \textbf{clustered proportionally sparse} if for any $\varepsilon>0$, 
\begin{equation}
    \sup_{k\in\mathcal{K}}\sup_{U\subseteq V^N}\Big|\Big\{i\in W^N_k:\Big|\frac{|\mathcal{N}^N_w(i)\cap U|}{|\mathcal{N}^N_w(i)|}-\frac{|E^N_k(U)|}{|E^N_k(V^N)|}\Big|\geq \varepsilon\Big\}\Big|/|W^N_k|\xrightarrow{N\rightarrow\infty}0,
\end{equation}
where $E^N_k(U)\coloneqq\{(i,j)\in W^N_k\times U: (i,j)\in E^N\}$.
\end{defn}

\begin{remark}\normalfont
We can view the subset $U$ in the definition as a test set, say $U=\mathcal{Q}^N_{m,l}(t)$, where $\mathcal{Q}^N_{m,l}(t)$ is the set of type $m\in\mathcal{M}$ servers with queue length at least $l\in\N_0$ at time $t$. 
Hence, Definition~\ref{defn:prop-sparse} ensures that for all but $o(N)$ dispatchers, the observed empirical queue length distribution within its neighborhood, is close to the global weighted empirical queue length distribution (Definition~\ref{defn:global-queue-length}) of its corresponding type. Then, the global occupancy process evolves similarly to (and converges to the same limit as) the case when the initial states are i.i.d.. 
\end{remark}
\begin{theorem}[Process-level convergence]\label{thm:general-conv-global-occup}
Let $\{G^N\}_N$ be a clustered proportionally sparse sequence of graphs. Assume that $\mathbf{q}^N(0)$ weakly converges to $\mathbf{q}^{\infty}\in\mathcal{S}$. Then on any finite time interval, the occupancy process $\mathbf{q}^N(\cdot)$ converges weakly with respect to the Skorokhod $J_1$ topology to the deterministic limit $\mathbf{q}(\cdot)\coloneqq(q_{m,l}(\cdot),m\in\mathcal{M},l\in\N_0)$ given by the unique solution to the system of ODEs defined by \eqref{eq:ode-sys-2} with initial state $\mathbf{q}(0)=\big(q^{\infty}_{m,l},m\in\mathcal{M},l\in\N_0\big)$.
\end{theorem}

The proof of Theorem~\ref{thm:general-conv-global-occup} is given in Section~\ref{sec:general-conv}

\subsection{Convergence of Steady States}\label{sec:ode-ana-interchange-limits}\label{sec:interchange-limits}
In the last section, we showed the process-level convergence of global occupancy process $\mathbf{q}^N(\cdot)$ to a mean-field limit $\mathbf{q}(\cdot)$. 
In this section, we will establish the convergence of the sequence of stationary distributions to the unique fixed point of the mean-field limit by establishing the interchange of large-$N$ and large-$t$ limits: $\lim_{t\rightarrow\infty}\lim_{N\rightarrow\infty}\mathbf{q}^N(t)=\lim_{N\rightarrow\infty}\lim_{t\rightarrow\infty}\mathbf{q}^N(t)$.
Throughout this section, we will assume that the sequence of systems is in subcritical regime (recall Definition~\ref{def:subcrit}).
The first result below states that the limiting system of ODEs have a unique fixed point $\mathbf{q}^*$ and it satisfies the global stability property, i.e., for any initial point $\mathbf{q}(0)\in \mathcal{S}$, $\lim_{t\rightarrow\infty}\mathbf{q}(t)=\mathbf{q}^*$.
\begin{theorem}[Global stability]\label{thm: exists-unique}
Let $\bar{\mathbf{q}}(t,\mathbf{q}_0)$ be the solution to the system of ODEs in~\eqref{eq:ode-sys-2} with the initial point $\mathbf{q}(0)=\mathbf{q}_0\in \mathcal{S}$. 
Then there exists a unique fixed point $\mathbf{q}^*=\big(q^*_{m,l},m\in\mathcal{M},l\in\N_0\big)\in\mathcal{S}$ such that
$$\lim_{t\rightarrow\infty}\bar{\mathbf{q}}(t,\mathbf{q}_0)=\mathbf{q}^*.$$
\end{theorem}
The proof of Theorem~\ref{thm: exists-unique} is given in Section~\ref{sec:interchange-of-limits}. 
It relies on a monotonicity property of the system, which ensures that for two processes $\mathbf{q}^1(\cdot)$ and $\mathbf{q}^2(\cdot)$, if $\mathbf{q}^1(0)\leq \mathbf{q}^2(0)$, then $\mathbf{q}^1(t)\leq \mathbf{q}^2(t)$ for all $t\geq 0$ (see ref.~\cite{AS15,MS99}).

The last ingredient that we need in order to prove the interchange of limits is to establish tightness of the sequence of random variables $\{\mathbf{q}^N(\infty)\}_{N\geq 1}$ under suitable metric, where $\mathbf{q}^N(\infty) := \lim_{t \to \infty}\mathbf{q}^N(t)$. 
Here, as before, we should note that the process $(\mathbf{q}^N(t))_{t\geq 0}$ is not Markovian. 
That is why, the random variable $\mathbf{q}^N(\infty)$ should be interpreted as the functional applied to the steady-state system.
The tightness result is stated in the next theorem.
\begin{theorem}[Tightness]\label{thm:tightness}
For any $\varepsilon>0$, there exists a compact subset $\bar{K}(\varepsilon)\subseteq \mathcal{S}$, when $\mathcal{S}$ is equipped with the $\ell_1$-topology, such that 
    $$\PP(\mathbf{q}^N(\infty)\notin \bar{K}(\varepsilon))<\varepsilon,\quad \forall N\geq 1.$$
\end{theorem}
Theorem~\ref{thm:tightness} is proved in Section~\ref{sec:interchange-of-limits}.
The key idea is to use Lyapunov function approach to bound the expected sum of tails $q^N_{m,l}(\infty)$. 
Combining Theorems~\ref{thm:general-conv-global-occup},~\ref{thm: exists-unique}, and~\ref{thm:tightness} we can prove the following interchange of limits result.
\begin{theorem}[Convergence of steady states]\label{thm:interchange of limits}
Let $\{G^N\}_{N\geq 1}$ be a clustered proportionally sparse sequence of graphs satisfying Condition~\ref{cond-1}. Then the sequence of random variables $\{\mathbf{q}^N(\infty)\}_{N\geq 1}$ converges weakly to $\mathbf{q}^*$, the unique fixed point of the system of ODEs in~\eqref{eq:ode-sys-2}. 
\end{theorem}

One major discovery about the JSQ($d$) policy for the classical, homogeneous, fully flexible system is that the limit of the stationary distribution (which, in our case, is given by $\mathbf{q}^*$) has a double-exponential decay of tail~\cite{Mitzenmacher96,VDK96} for any $d\geq 2$.
This is in sharp contrast with the (single) exponential decay of the corresponding tail for random routing or $d=1$.
In fact, in this case, for any $d\geq 2$, $\mathbf{q}^*$ can be characterized explicitly as: $q^*_l=\lambda^{\frac{d^l-1}{d-1}}$, where 
$q_l^*$ is the (limiting) steady-state fraction of servers with queue length at least $l= 1, 2, \ldots$.
In the current case of heterogeneous systems, it is intractable to characterize the fixed point $\mathbf{q}^*$ explicitly. 
However, as stated in the next theorem, we can still prove that the doubly exponential decay of the tails $q^*_{m,l}$ for each $m\in\mathcal{M}$ holds.
\begin{theorem}[Double-exponential tail decay]\label{thm:double-expo}
Let $\mathbf{q}^*=\big(q^*_{m,l},m\in\mathcal{M},l\in\N_0\big)$ be the unique fixed point of the system of ODEs in~\eqref{eq:ode-sys-2}. 
Then, for all $m\in\mathcal{M}$, the sequence $\big\{q^*_{m,l},l\in\N_0\big\}$ decreases doubly exponentially, i.e., there exist positive constant $l_m\in\N_0$, $a_m\in(0,1)$ and $b_m>0$ such that for all $l\geq l_m$, 
    \begin{equation}
        q^*_{m,l}\leq b_m a_m^{d^l}.
    \end{equation}
\end{theorem}

\subsection{Simple Data Locality Design using Randomization}\label{sec:rand}
Sections~\ref{sec:stability-issue}--\ref{sec:interchange-limits} characterize the performance of the occupancy process for arbitrary deterministic sequence of systems where the underlying graph sequence satisfies certain properties. 
In particular, Condition~\ref{cond-1} and Definition~\ref{defn:prop-sparse} provide a sufficient criteria under which both the process-level convergence (Theorem~\ref{thm:general-conv-global-occup}) and interchange of limits (Theorem~\ref{thm:interchange of limits}) hold. 
In this section, we show that graphs satisfying the above required criteria can be obtained easily if the compatibility graph is designed suitably randomly.
Given the asymptotic edge-density parameters in Condition~\ref{cond-1}, we define a certain sequence of  \emph{inhomogeneous random graphs} or \irg\ as follows.

\begin{definition}[\irg($\mathbf{p}$)]\label{def:irg}
Given $\mathbf{p}\coloneqq (p_{k,m},k\in\mathcal{K},m\in\mathcal{M})$, the $N$-th system of \irg($\mathbf{p}$) is constructed as follows: For any $k\in\mathcal{K}$ and $m\in\mathcal{M}$, dispatcher $i$ and server $j$ shares an edge with probability $p_{k,m}$ for all $i\in W^N_k$ and $j\in V^N_m$, independently of each other.
\end{definition}

For any $\mathbf{p}$ for which the asymptotic stability criterion holds, we have the following result for the sequence of \irg($\mathbf{p}$).

\begin{theorem}\label{prop:random-graph}
Let $\mathbf{p} = (p_{k,m},k\in\mathcal{K},m\in\mathcal{M})$ be such that the stability criterion in \eqref{eq:max-rho-set} holds and $\{G_N\}_{N\geq 1}$ be a sequence of \irg($\mathbf{p}$) with increasing $N$.
Then the conclusions of Theorem~\ref{thm:general-conv-global-occup} and Theorem~\ref{thm:interchange of limits} hold for $\{G_N\}_{N\geq 1}$.
\end{theorem}

The proof of Theorem~\ref{prop:random-graph} is provided in Appendix~\ref{app:random-graph}. 
It relies on verifying that the sequence of \irg($\mathbf{p}$) graphs satisfies Condition~\ref{cond-1} and the property of clustered proportional sparsity almost surely.
The verification involves using concentration of measure arguments to establish structural properties of the compatibility graphs.

\section{Proof of Transient Limit Results}\label{sec:transient-limit}
In this section, we will prove the results of transient limit results, Theorems~\ref{thm:weak-conv-xj}, ~\ref{thm:conv-global-occup} and~\ref{thm:general-conv-global-occup} in Sections~\ref{ssec:thm-iid-1},~\ref{sec:thm3.5}, and~\ref{sec:general-conv}, respectively. 
We start by proving a few auxiliary results in Section~\ref{ssec:aux}.

\subsection{Auxiliary Results}\label{ssec:aux}

First, we will need a characterization of the evolution of the queue length process at each server.
To describe this evolution, let us introduce the following notations:
\begin{align}
     \set^N(j)&\coloneqq\Big\{(j_2,...,j_d)\in [N]^{d-1}:(j,j_2,\ldots,j_d)\text{ are distinct}\Big\},\label{defn:set-j}\\
     \sett^N(j)&\coloneqq\Big\{(j_2,\ldots,j_d,j'_2,\ldots,j'_d)\in[N]^{2d-2}:
        (j_2,...,j_d)\in \set^N{(j)},(j'_2,...,j'_d)\in\set^N{(j)},  \nonumber\\
        &\hspace{7cm}(j_2,...,j_d)\cap(j'_2,...,j'_d)\neq\emptyset\Big\}.\label{defn:sett-j}
\end{align}
To represent the graph, define the edge occupancy $\xi^N_{i,j}$ to be the binary variable: 
$$\xi^N_{i,j}=\begin{cases}
      1, & \text{if}\ (i,j)\in E^N, \\
      0, & \text{otherwise},
    \end{cases} \quad \mbox{for all } i\in W^N, j\in V^N.$$
Recall the function $b$, Poisson processes $\{D_j\}$ and Poisson random measures $\{A_j\}$ in and below \eqref{eq:b}.
By Condition~\ref{cond-1}, for all large enough $N$, all dispatchers in the $N$-th system have at least $d$ neighbors. Hence, WLOG, in the rest of this section, we will only consider the case $\delta^N_i\geq d$, $\forall i\in W^N$.
In that case, due to the Poisson thinning property, note that we can write $X^N_j(t)$ as follows:
\begin{equation}\label{eq:represent-X^N}
    X^N_j(t)=X^N_j(0)-\int_0^t \mathds{1}_{\big(X^N_j(s-)>0\big)}D_j(ds)+\int_{[0,\infty)\times \R_+}\mathds{1}_{\big(0\leq y\leq C^N_j(s-)\big)}A_j(dsdy),
\end{equation}
where 
\begin{align}\label{eq:represent-C^N}
    C^N_j(s)&=\sum_{i\in W^N}\xi^N_{i,j}\sum_{(j_2,...,j_d)\in\set^N(j)}\frac{\xi^N_{i,j_2}\times\cdots\times \xi^N_{i,j_d}}{{\delta^N_i\choose d}(d-1)!}b\big(X^N_j(s),X^N_{j_2}(s),...,X^N_{j_d}(s)\big)\\
        &=\sum_{k\in\mathcal{K}}\sum_{i\in W^N_k}\xi^N_{i,j}\sum_{(j_2,...,j_d)\in\set^N(j)}\frac{\xi^N_{i,j_2}\times\cdots\times \xi^N_{i,j_d}}{{\delta^N_i\choose d}(d-1)!}b\big(X^N_j(s),X^N_{j_2}(s),...,X^N_{j_d}(s)\big).\nonumber
\end{align}
The RHS of the first summation in \eqref{eq:represent-C^N} represents the probability that a job arriving at the dispatcher $i\in W^N$ will be assigned to the server $j\in V^N$ given the state $\big(X^N_j, j\in V^N\big)$.
Moreover, by Condition~\ref{cond-1}, the term $C^N_j$ for all $j\in V^N$ can be upper bounded, uniformly for all $t$, by a constant for all large enough $N$, which is stated in Lemma~\ref{lem:bound-C^N} below.

When we do some estimation, like bounding the term $C^N_j$, we need to uniformly bound the number of the neighbors of servers or dispatchers. Such uniformity is stated in Lemma~\ref{lem:uniformity-degree} and is a direct result of Condition~\ref{cond-1}.
Recall $\delta^N_i=|\mathcal{N}^N_w(i)|$ and $\delta_k=\sum_{m\in\mathcal{M}}p_{k,m}v_m$.

\begin{lemma}\label{lem:uniformity-degree}
For each $k\in\mathcal{K}$,
\begin{equation}\label{eq:rem-4}
    \lim_{N\rightarrow\infty}\max_{i\in W^N_k}\frac{\deg^N_w(i,m)}{|V^N_m|}=
    \lim_{N\rightarrow\infty}\min_{i\in W^N_k}\frac{\deg^N_w(i,m)}{|V^N_m|}=p_{k,m},\quad m\in\mathcal{M},
\end{equation}
and  \begin{equation}\label{eq:rem-4-1}
    \lim_{N\rightarrow\infty}\max_{i\in W^N_k}\frac{\delta^N_i}{N}=
    \lim_{N\rightarrow\infty}\min_{i\in W^N_k}\frac{\delta^N_i}{N}=\delta_k.
\end{equation}
Also, for each $m\in\mathcal{M}$,
\begin{equation}\label{eq:rem-5}
    \lim_{N\rightarrow\infty}\max_{j\in V^N_m}\frac{\deg^N_v(k,j)}{|W^N_k|}=\lim_{N\rightarrow\infty}\min_{j\in V^N_m}\frac{\deg^N_v(k,j)}{|W^N_k|}=p_{k,m},\quad k\in\mathcal{K}.
\end{equation}
\end{lemma}
\begin{lemma}\label{lem:bound-C^N}
For all large enough $N$, we have that for any $m\in\mathcal{M}$, $j\in V^N_m$, and $t\geq 0$, 
\begin{equation}\label{eq:cj-2d}
    C^N_j(t)\leq 2\xi d\sum_{k\in\mathcal{K}} \frac{p_{k,m}w_k}{\delta_k}.
\end{equation}
\end{lemma}
\begin{proof}
By the definition of $C^N_j(t)$, for any $t\geq 0$ and large enough $N$,
\begin{equation*}
        C^N_j(t) \leq  \sum_{k\in\mathcal{K}}\sum_{i\in W^N_k}\xi^N_{i,j}\sum_{(j_2,...,j_d)\in\set^N(j)}\frac{\xi^N_{i,j_2}\times\cdots\times \xi^N_{i,j_d}}{{\delta^N_i\choose d}(d-1)!}= \sum_{k\in\mathcal{K}} \sum_{i\in W^N_k}\xi^N_{i,j}\frac{{\delta^N_i-1\choose d-1}}{{\delta^N_i \choose d}}
        \leq  2\xi d\sum_{k\in\mathcal{K}} \frac{p_{k,m}w_k}{\delta_k},
\end{equation*}
where the first inequality is due to $b(\cdot)\leq 1$ and the last inequality comes from Lemma~\ref{lem:uniformity-degree}.
\end{proof}
By Lemma~\ref{lem:uniformity-degree}, we know that the neighborhoods of dispatchers of the same type are almost the same. With the scale of the system size, the local graph structure for each dispatcher of the same type will converge to the average one. 
The following two lemmas give necessary approximation of the graph structures for large-$N$ systems. Their proofs are combinatorial and are based on Condition~\ref{cond-1} and Lemma~\ref{lem:uniformity-degree}.
They are provided in Appendix~\ref{app:approx-graph}.

\begin{lemma} \label{lem: LLN-xi}
Consider a sequence $\{G^N\}_N$ satisfying Condition~\ref{cond-1}.
For each $m\in \mathcal{M}$,
\begin{equation}\label{eq:ave-deg}
   \begin{split}
        \max_{j\in V^N_m}\max_{k\in\mathcal{K}}\max_{(M_2,...,M_d)\in \mathcal{M}^{d-1}}\Big|&\sum_{i\in W^N_k}\xi^N_{i,j}\sum_{\substack{(j_2,...,j_d)\in \set^N (j)\\ s.t. \quad j_2\in V^N_{M_2},...,j_d\in V^N_{M_d}}}\frac{\xi^N_{i,j_2}\times \cdots \times \xi^N_{i,j_d}}{{\delta^N_i\choose d}(d-1)!}\\
        &\hspace{4cm} -\xi d\frac{p_{k,m}w_k}{\delta_k}\prod_{h=2}^d\frac{v_{M_h}p_{k,M_h}}{\delta_k}\Big|\xrightarrow{N\rightarrow\infty}0.
   \end{split}
\end{equation}
\end{lemma}

\begin{lemma}\label{lem:bound-sett}
Consider any $m\in \mathcal{M}$ and $j\in V_{m}$. For large enough $N$,
\begin{equation}\label{eq:ave-deg-sett-1}
    \begin{split}
        &\sum_{i\in W^N}\sum_{\sett^N{(j)}}\frac{\xi^N_{i,j}\times\xi^N_{i,j_2}\times\cdots\times \xi^N_{i,j_d}}{{\delta^N_{i}\choose d}(d-1)!} \frac{\xi^N_{i,j}\times\xi^N_{i,j'_2}\times\cdots\times \xi^N_{i,j'_d}}{{\delta^N_{i}\choose d}(d-1)!}
        \leq  \frac{C_1}{N^2}
    \end{split}
\end{equation}
where $C_1$ is a positive constant. Similarly,
\begin{equation}\label{eq:ave-deg-sett-2}
    \begin{split}
        &\sum_{\substack{i_1,i_2\in W^N,\\ i_1\neq i_2}}\sum_{\sett^N{(j)}} \frac{\xi^N_{i_1,j}\times\xi^N_{i_1,j_2}\times\cdots\times \xi^N_{i_1,j_d}}{{\delta^N_{i_1}\choose d}(d-1)!} \frac{\xi^N_{i_2,j}\times\xi^N_{i_2,j'_2}\times\cdots\times \xi^N_{i_2,j'_d}}{{\delta^N_{i_2}\choose d}(d-1)!}
    \leq  \frac{C_2}{N}
    \end{split}
\end{equation}
where $C_2$ is a positive constant.
\end{lemma}

\subsection{Convergence to McKean-Vlasov Process: IID Case}
\label{ssec:thm-iid-1}

\begin{proof}[Proof of Theorem~\ref{thm:weak-conv-xj}]
It suffices to prove \eqref{eq:thm1-coupled-cvg}.
Fix any $m\in \mathcal{M}$, $j\in V_{m}$ and $T>0$. We have that for any fixed $t\in[0,T]$ and any $N$ s.t. $j\in V^N$,
\allowdisplaybreaks
\begin{align}
        \E\norm{X^N_j-X_j}^2_{*,t}& \leq  c_0 \E \norm{X^N_j(t)-X_j(t)}^2 \nonumber\\
        &\leq  c_1 \E \int_0^t|\mathds{1}_{(X^N_j(s)>0)}-\mathds{1}_{(X_j(s)>0)}|^2ds+c_1 \E\Big(\int_0^t |\mathds{1}_{(X^N_j(s)>0)}-\mathds{1}_{(X_j(s)>0)}|ds\Big)^2 \nonumber\\
        &\quad +c_1 \E\int_{[0,t]\times \R_+}|\mathds{1}_{(0\leq y\leq C^N_j(s))}-\mathds{1}_{(0\leq y\leq C_j(s))}|^2dsdy\nonumber\\
        &\quad +c_1 \E\Big(\int_{[0,t]\times \R_+}|\mathds{1}_{(0\leq y\leq C^N_j(s))}-\mathds{1}_{(0\leq y\leq C_j(s))}|dsdy\Big)^2\nonumber\\
        & \leq  c_1 \E \int_0^t|X^N_j(s)-X_j(s)|^2ds+c_1 \E\Big(\int_0^t |X^N_j(s)-X_j(s)|ds\Big)^2\nonumber\\
        &\hspace{2cm}+c_1 \E \int_0^t|C^N_j(s)-C_j(s)|^2ds+c_1 \E\Big(\int_0^t |C^N_j(s)-C_j(s)|ds\Big)^2\nonumber\\
        & \leq  c_2\int_0^t \E |X^N_j(s)-X_j(s)|^2ds + c_2\int_0^t \E |C^N_j(s)-C_j(s)|ds,\label{eq:xNj-xN-sup}
\end{align}
where $c_0$, $c_1$ and $c_2$ are positive constants. The first two inequalities are by Doob's inequalities and Cauchy-Schwarz, respectively. 
The last inequality is due to Lemma~\ref{lem:uniformity-degree}.
By adding and subtracting terms, we have 
\begin{equation}\label{eq:cN-c}
    \begin{split}
        |C^N_j&(s)-C_j(s)|
        \leq |C^N_j(s)-C^{N,1}_j(s)|+|C^{N,1}_j(s)-C^{N,2}_j(s)|+|C^{N,2}_j(s)-C_j(s)|,
    \end{split}
\end{equation}
where \begin{align*}
    C^{N,1}_j & =\sum_{k\in \mathcal{K}}\sum_{i\in W^N_k}\Big[\xi^N_{i,j}\sum_{(j_2,...,j_d)\in \set^N (j)}\frac{\xi^N_{i,j_2}\times \cdots \times \xi^N_{i,j_d}}{{\delta^N_i\choose d}(d-1)!}b(X_j(s),X_{j_2}(s),...,X_{j_d}(s))\Big], \\
    C^{N,2}_j&=\sum_{k\in \mathcal{K}}\sum_{i\in W^N_k}\bigg[\xi^N_{i,j}\sum_{(j_2,...,j_d)\in \set^N (j)}\frac{\xi^N_{i,j_2}\times \cdots \times \xi^N_{i,j_d}}{{\delta^N_i\choose d}(d-1)!}\int_{\N^{d-1}}b(X_j(t),x_{j_2},...,x_{j_d})\\
    &\hspace{9cm}\mu^{\mathbf{M}(j_2)}_t(dx_{j_2})\cdots\mu^{\mathbf{M}(j_d)}_t(dx_{j_d})\bigg].
\end{align*}

\noindent
First, consider $|C^N_j(s)-C^{N,1}_j(s)|$. For large enough $N$,
\allowdisplaybreaks
\begin{align}\label{eq:cN-cN1}
    &\E|C^N_j(s)-C^{N,1}_j(s)|\nonumber\\
        &= \E\bigg|\sum_{k\in \mathcal{K}}\sum_{i\in W^N_k}\Big[\xi^N_{i,j}\sum_{(j_2,...,j_d)\in \set^N (j)}\frac{\xi^N_{i,j_2}\times \cdots \times \xi^N_{i,j_d}}{{\delta^N_i\choose d}(d-1)!}\big(b(X^N_j(s),X^N_{j_2}(s),...,X^N_{j_d}(s))\nonumber\\
        &\hspace{9cm}-b(X_j(s),X_{j_2}(s),...,X_{j_d}(s))\big)\Big]\bigg|\nonumber\\
        &\leq  \E\sum_{k\in \mathcal{K}}\sum_{i\in W^N_k}\Big[\xi^N_{i,j}\sum_{(j_2,...,j_d)\in \set^N (j)}\frac{\xi^N_{i,j_2}\times \cdots \times \xi^N_{i,j_d}}{{\delta^N_i\choose d}(d-1)!}\big(|X^N_j(s)-X_j(s)|+\cdots+|X^N_{j_d}(s)-X_{j_d}(s)|\big)\Big]\nonumber\\
        &\leq  d\times  \max_{j\in V^N}\E[|X^N_j(s)-X_j(s)|]\times \sum_{k\in \mathcal{K}}\sum_{i\in W^N_k}\xi^N_{i,j}\sum_{(j_2,...,j_d)\in \set^N (j)}\frac{\xi^N_{i,j_2}\times \cdots \times \xi^N_{i,j_d}}{{\delta^N_i\choose d}(d-1)!}\nonumber\\
        &\leq  c_3 \max_{j\in V^N}\E[|X^N_j(s)-X_j(s)|],
\end{align}
where $c_3$ is constant. The first inequality is from the that $b(\cdot)$ is Lipschitz continuous with Lipschitz constant 1 and the last inequality is from \eqref{eq:ave-deg}. 
\vspace{3mm}

\noindent
Second, consider $|C^{N,1}_j(s)-C^{N,2}_j(s)|$.
\allowdisplaybreaks
\begin{align}
        &\E\big[|C^{N,1}_j(s)-C^{N,2}_j(s)|^2\big]\nonumber\\
        &= \E\Big|\sum_{k\in \mathcal{K}}\sum_{i\in W^N_k}\Big[\xi^N_{i,j}\sum_{(j_2,...,j_d)\in \set^N (j)}\frac{\xi^N_{i,j_2}\times \cdots \times \xi^N_{i,j_d}}{{\delta^N_i\choose d}(d-1)!}b(X_j(s),X_{j_2}(s),...,X_{j_d}(s))\Big]\nonumber\\
        &\hspace{1.5cm}-\sum_{k\in \mathcal{K}}\sum_{i\in W^N_k}\Big[\xi^N_{i,j}\sum_{(j_2,...,j_d)\in \set^N (j)}\frac{\xi^N_{i,j_2}\times \cdots \times \xi^N_{i,j_d}}{{\delta^N_i\choose d}(d-1)!}\int_{\N^{d-1}_0}b(X_j(s),x_{j_2},...,x_{j_d})\\
        &\hspace{9cm}\mu^{\mathbf{M}(j_2)}_s(dx_{j_2})\cdots\mu^{\mathbf{M}(j_d)}_s(dx_{j_d})\Big]\Big|^2\nonumber\\
        & \leq  \E\Big[\sum_{i_1,i_2\in W^N}\sum_{\sett^N(j)}\frac{\xi^N_{i_1,j}\times\xi^N_{i_1,j_2}\times\cdots\times \xi^N_{i_1,j_d}}{{\delta^N_{i_1}\choose d}(d-1)!} \frac{\xi^N_{i_2,j}\times\xi^N_{i_2,j'_2}\times\cdots\times \xi^N_{i_2,j'_d}}{{\delta^N_{i_2}\choose d}(d-1)!}\Big]\nonumber\\
       &\overset{(a)}{\leq}  \E\Big[\sum_{i\in W^N}\sum_{\sett^N{(j)}}\frac{\xi^N_{i,j}\times\xi^N_{i,j_2}\times\cdots\times \xi^N_{i,j_d}}{{\delta^N_{i}\choose d}(d-1)!} \frac{\xi^N_{i,j}\times\xi^N_{i,j'_2}\times\cdots\times \xi^N_{i,j'_d}}{{\delta^N_{i}\choose d}(d-1)!}\nonumber\\
       &\hspace{1.5cm}+\sum_{i_1,i_2\in W^N,i_1\neq i_2}\sum_{\sett^N{(j)}} \frac{\xi^N_{i_1,j}\times\xi^N_{i_1,j_2}\times\cdots\times \xi^N_{i_1,j_d}}{{\delta^N_{i_1}\choose d}(d-1)!} \frac{\xi^N_{i_2,j}\times\xi^N_{i_2,j'_2}\times\cdots\times \xi^N_{i_2,j'_d}}{{\delta^N_{i_2}\choose d}(d-1)!}\Big],\nonumber\\
       & \leq  c_4 N^{-2}+ c_5 N^{-1}\label{eq:cN1-cN2}
\end{align}
where the first inequality is due to the fact that $X_j(0)$ is i.i.d.~for $j\in V_m$ and independent for different $m$, so for each $m\in\mathcal{M}$, $\{X_j(s),j\in V_m\}$ are also i.i.d., and the independence across the server pools holds for any fixed $s>0$. 
Hence, if $(j,j_2,...,j_d,j'_2,...,j'_d)$ are distinct, then
\begin{equation*}
    \begin{split}
        \E\Big[&\big(b(X_j(t),X_{j_2}(t),...,X_{j_d}(t))-\int_{\N^{d-1}}b(X_j(t),x_{j_2},...,x_{j_d})\mu^{\mathbf{M}(j_2)}_t(dx_{j_2})\cdots\mu^{\mathbf{M}(j_d)}_t(dx_{j_d})\big)\\
        &\big(b(X_j(t),X_{j'_2}(t),...,X_{j'_d}(t))-\int_{\N^{d-1}}b(X_j(t),x_{j'_2},...,x_{j'_d})\mu^{\mathbf{M}(j'_2)}_t(dx_{j'_2})\cdots\mu^{\mathbf{M}(j'_d)}_t(dx_{j'_d})\big)\Big]=0,
    \end{split}
\end{equation*}
and $b(\cdot)$ and $\int b(\cdot)\mu(d\cdot)$ are both in $[0,1]$. The last inequality of \eqref{eq:cN1-cN2} is by \eqref{eq:ave-deg-sett-1} and \eqref{eq:ave-deg-sett-2}.

\noindent
Third, consider $|C^{N,2}_j(s)-C_j(s)|$.
\begin{align}
        &\E \big[|C^{N,2}_j(s)-C_j(s)|\big] \notag \\
        &=\E\Big[\Big|\sum_{k\in \mathcal{K}}\sum_{i\in W^N_k}\Big[\xi^N_{i,j}\sum_{(j_2,...,j_d)\in \set^N (j)}\frac{\xi^N_{i,j_2}\times \cdots \times \xi^N_{i,j_d}}{{\delta^N_i\choose d}(d-1)!}\int_{\N^{d-1}}b(X_j(t),x_{j_2},...,x_{j_d}) \notag \\
        &\hspace{9cm}\mu^{\mathbf{M}(j_2)}_t(dx_{j_2})\cdots\mu^{\mathbf{M}(j_d)}_t(dx_{j_d})\Big] \notag \\
        &-d\xi\sum_{k\in \mathcal{K}}\frac{p_{k,m}w_k}{\delta_k}\sum_{(M_2,...,M_d)\in \mathcal{M}^{d-1}} \prod_{h=2}^d \frac{v_{M_h}p_{k,M_h}}{\delta_k}\int_{\N^{d-1}}b(X_j(t),x_{j_2},...,x_{j_d}) \notag \\
        &\hspace{9cm}\mu^{\mathbf{M}(j_2)}_t(dx_{j_2})\cdots\mu^{\mathbf{M}(j_d)}_t(dx_{j_d})\Big|\Big] \notag \\
        &\leq c_6(N), \label{eq:cN2-c}
\end{align}
where $c_6(N)$ only depends on $N$ and goes to $0$ as $N\rightarrow\infty$ and the inequality comes from \eqref{eq:ave-deg} and the fact that $\int b(\cdot)\mu(d\cdot)\in[0,1]$.
Now, by \eqref{eq:xNj-xN-sup}, \eqref{eq:cN-c}, \eqref{eq:cN-cN1}, \eqref{eq:cN1-cN2} and \eqref{eq:cN2-c}, we have that for large enough $N$,
$$\max_{j\in V^N}\E\norm{X^N_j-X_j}^2_{*,t}\leq c_{10}\int_0^t\max_{j\in V^N}\E\norm{X^N_j-X_j}^2_{*,t}ds+f(N),$$
where $c_{10}$ is a constant and $f(N)$ is a function which goes to 0 as $N\rightarrow\infty$. Last by Gronwall's inequality, we have 
\eqref{eq:thm1-coupled-cvg} and this completes the proof.
\end{proof}

\subsection{Convergence of the Occupancy Process: IID Case}
\label{sec:thm3.5}
In this section, we want to show the convergence of the occupancy process $\mathbf{q}^N(\cdot)$ to the limit process $\mathbf{q}$ represented by the ODE~\eqref{eq:ode-sys-2}. 
The first step is to investigate the existence and uniqueness of the solution of the ODE~\eqref{eq:ode-sys-2}.
Define $$\bar{\mathcal{S}}\coloneqq\Big\{\mathbf{q}\in[0,1]^{M\times \N_0}:q_{m,0}=1,q_{m,l}\ge q_{m,l+1}, \forall m\in\mathcal{M}, l\in\N_0\Big\}$$
and clearly, $\mathcal{S}\subseteq\bar{\mathcal{S}}$.

\begin{lemma}\label{lem:unique-sol-ode}
If $\mathbf{q}(0)=\mathbf{q}_0 \in \bar{\mathcal{S}}$, then the ODE system \eqref{eq:ode-sys-2} has a unique solution denoted as $\bar{\mathbf{q}}(t,\mathbf{q}_0)$, $t\geq 0$ in $\bar{\mathcal{S}}$.
\end{lemma}
The proof of Lemma~\ref{lem:unique-sol-ode} is based on the Picard successive approximation method (\cite[Theorem 1(i)]{MS99}) and is provided in Appendix~\ref{app:unique-sol-ode}.

\begin{proof}[Proof of Theorem~\ref{thm:conv-global-occup}]
Fix any $T\in (0,\infty)$. For each $m\in\mathcal{M}$, consider random measures $\mu^N_m=\frac{1}{|V^N_m|}\sum_{j\in V^N_m}\delta_{X^N_j(\cdot)}$ and $\Bar{\mu}^N_m=\frac{1}{|V^N_m|}\sum_{j\in V^N_m}\delta_{X_j(\cdot)}$ on $\mathbb{S}\coloneqq\mathbb{D}([0,T],\N_0)$, where $X_j(\cdot)$ is defined in \eqref{eq:thm-1-1}. Denote the joint measures $\mu^N=(\mu^N_1,...,\mu^N_M)$ and $\Bar{\mu}^N=(\Bar{\mu}^N_1,...,\Bar{\mu}^N_M)$. Denote by $d_{BL}(\cdot,\cdot)$ the bounded-Lipschitz metric for probability measures on $\mathbb{S}$:
\begin{equation*}
        d_{BL}(\mu_1,\mu_2)\coloneqq\sup_{\norm{f}_{BL}\leq 1}\Big|\int_{\mathbb{S}}fd\mu_1-\int_{\mathbb{S}}fd\mu_2\Big|, \quad \norm{f}_{BL}\coloneqq\max\Big\{\norm{f}_{\infty},\sup_{x\neq y}\frac{f(x)-f(y)}{d(x,y)}\Big\}.
\end{equation*}
From \eqref{eq:thm1-coupled-cvg}
we have
\begin{equation*}
    \begin{split}
        \E d_{BL}(\mu^N_m,\Bar{\mu}^N_m)&\leq \E\sup_{\norm{f}_{BL}\leq 1}\frac{1}{|V^N_m|} \sum_{j\in V^N_m}|f(X^N_j)-f(X_j)|
        \leq \frac{1}{|V_m^N|}\sum_{j\in V^N}\E \norm{X^N_j-X_j}_{*,T}\xrightarrow{N\rightarrow\infty}0
    \end{split}
\end{equation*}
which implies that $d_{BL}(\mu^N_m,\Bar{\mu}^N_m)\pto 0$ for each $m \in \mathcal{M}$.
Since $\Bar{\mu}^N_m \pto \mu_m$ by LLN, we have $\mu^N=(\mu^N_1,...,\mu^N_M) \pto (\mu_1,...,\mu_M)$ by Slutsky's theorem.
Also, it is easy to check that $$\sup_{N}\E\Big[ \sup_{0\leq t\leq T}\norm{\mathbf{q}^N(t)}^2_{\ell_1}\Big]<\infty.$$ Thus, we have $\mathbf{q}^N\pto \mathbf{q}$.
Next, we need to show that $\mathbf{q}$ satisfies \eqref{eq:ode-sys-2}. Define $f_l(x)=\mathds{1}_{\{x\geq l\}}$, $l\in\N_0$. By \eqref{eq:thm-1-1}, we have that for any $m\in\mathcal{M}$ and $j\in V_m$,
\allowdisplaybreaks
\begin{align*}
    \E f_l(X_j(t))&=\E f_l(X_j(0))+\int_0^t u_m \E \mathds{1}_{\{X_j(s)>0\}}\big(f_l(X_j(s)-1)-f_l(X_j(s))\big)ds\\
    &\quad+\int_0^t\int_{\N^{d-1}}   \lambda \xi d \sum_{k\in \mathcal{K}}\frac{p_{k,m}w_k}{\delta_k}\sum_{(M_2,...,M_d)\in \mathcal{M}^{d-1}}\prod_{h=2}^d \frac{v_{M_h}p_{k,M_h}}{\delta_k}\\
    & \qquad \times \E \big[b(X_j(s),x_{j_2},...,x_{j_d})\big(f_l(X_j(s)+1)-f_l(X_j(s))\big)\big]\mu^{M_2}_s(dx_{j_2})\cdots\mu^{M_d}_s(dx_{j_d})ds\\
    &=\E f_l(X_j(0))+\int_0^t u_m \E \mathds{1}_{\{X_j(s)>0\}}(f_{l+1}(X_j(s))-f_l(X_j(s)))ds\\
    &\quad+\int_0^t\int_{\N^{d-1}}   \lambda \xi d \sum_{k\in \mathcal{K}}\frac{p_{k,m}w_k}{\delta_k}\sum_{(M_2,...,M_d)\in \mathcal{M}^{d-1}}\prod_{h=2}^d \frac{v_{M_h}p_{k,M_h}}{\delta_k}\\
    &\qquad\times\E [b(l-1,x_{j_2},...,x_{j_d})(f_{l-1}(X_j(s)) -f_l(X_j(s)))]\mu^{M_2}_s(dx_{j_2})\cdots\mu^{M_d}_s(dx_{j_d})ds.
\end{align*}
For any $m\in\mathcal{M}$, if $j\in V_m$, then $\E f_l(X_j(t))=q_{m,l}(t)=\mu^m_t[l,\infty)$ for $l=1,2,...$. Hence, 
\begin{align}\label{eq:ode-1}
    q_{m,l}(t)&=q_{m,l}(0)-\int_0^t u_m(q_{m,l}(s)-q_{m,l+1}(s))ds + \int_0^t \lambda \xi d \sum_{k\in \mathcal{K}}\frac{p_{k,m}w_k}{\delta_k}(q_{m,l-1}(s)-q_{m,l}(s))\nonumber\\
    &\times\sum_{(M_2,...,M_d)\in \mathcal{M}^{d-1}}\prod_{h=2}^d \frac{v_{M_h}p_{k,M_h}}{\delta_k} \int_{\N^{d-1}}b(l-1,x_{j_2},...,x_{j_d})\mu^{M_2}_s(dx_{j_2})\cdots\mu^{M_d}_s(dx_{j_d})ds
\end{align}
Also, 
\begin{align}\label{eq:ave-k}
    &\sum_{(M_2,...,M_d)\in \mathcal{M}^{d-1}}\prod_{h=2}^d \frac{v_{M_h}p_{k,M_h}}{\delta_k} \int_{\N^{d-1}}b(l-1,x_{j_2},...,x_{j_d})\mu^{M_2}_s(dx_{j_2})\cdots\mu^{M_d}_s(dx_{j_d})\nonumber\\
    &=\sum_{\bar{r}\in\bar{\mathcal{R}}}\sum_{\bar{r}'\in\bar{\mathcal{R}}'(\bar{r})}\frac{1}{1+|\bar{r}'|}\prod_{m\in\mathcal{M}}{r_m\choose r'_m}\Big(\frac{v_mp_{k,m}}{\delta_k}\Big)^{r_m}(q_{m,l-1}(s)-q_{m,l}(s))^{r'_m}(q_{m,l}(s))^{r_m-r'_m}\nonumber\\
    &=\sum_{r=0}^{d-1}\frac{1}{1+r}{d-1\choose r}\Big(\sum_{m\in\mathcal{M}}\frac{v_mp_{k,m}}{\delta_k}q_{m,l-1}(s)-\sum_{m\in\mathcal{M}}\frac{v_mp_{k,m}}{\delta_k}q_{m,l}(s)\Big)^r\Big(\sum_{m\in\mathcal{M}}\frac{v_mp_{k,m}}{\delta_k}q_{m,l}(s)\Big)^{d-1-r}\nonumber\\
    &=\sum_{r=1}^{d}\frac{1}{r}{d-1\choose r-1}\big(\tilde{q}_{k,l-1}(s)-\tilde{q}_{k,l}(s)\big)^{r-1}(\tilde{q}_{k,l}(s))^{d-r}\quad (\text{Let } \tilde{q}_{k,l}(s)=\sum_{m\in\mathcal{M}}\frac{v_mp_{k,m}}{\delta_k}q_{m,l}(s))\nonumber\\
    &=\frac{(\tilde{q}_{k,l-1}(s))^d-(\tilde{q}_{k,l}(s))^d}{d(\tilde{q}_{k,l-1}(s)-\tilde{q}_{k,l}(s))}
\end{align}
where $\bar{\mathcal{R}}=\{\bar{r}=(r_1,...,r_M)\in \N_0^M:\sum_{m\in\mathcal{M}}r_m=d-1\}$ and $\bar{\mathcal{R}}'(\bar{r})=\{\bar{r}'=(r'_1,...,r'_M)\in\N_0^M: r'_m\leq r_m,\forall m\in\mathcal{M}\}$ given $\bar{r}\in\bar{\mathcal{R}}$.
Plugging \eqref{eq:ave-k} into \eqref{eq:ode-1}, we get the desired result.
\end{proof}

\subsection{Convergence of the Occupancy Process: General Case}
\label{sec:general-conv}

In this section, we will discuss the case that the sequence $\{G^N\}_N$ is clustered proportionally sparse, which helps us remove the i.i.d.~assumption in Theorem~\ref{thm:general-conv-global-occup}. Intuitively, if $\{G^N\}_N$ is clustered proportionally sparse, then for each $k\in\mathcal{K}$ and each dispatcher $i\in W^N_k$, the queue-length distribution of its neighborhood will always be close (in appropriate sense) to the corresponding global weighted queue-length distribution. 
Clustered proportional sparsity ensures that this statement holds uniformly for all occupancy states.
Loosely speaking, this statement enables us to make sure that the evolution of the occupancy process happens in the same way for any initial state as in the case of i.i.d.~initial state.
For the case of homogeneous systems, the notion of proportional sparsity was introduced in~\cite{RM22}. 
Here, proportional sparsity was defined in a way that for most dispatcher $i$, the fraction of its neighbors within any subset $U$ of servers is proportional to the size of the subset $U$. 
However, due to the heterogeneous compatibility between dispatchers and servers, such fraction, in the current setup, depends on the corresponding type of the dispatcher as well (see the term $\frac{E^N_k(U)}{E^N_k(V^N)}$ in Definition~\ref{defn:prop-sparse}). 
Thus, unlike the homogeneous case where the local queue-length distribution is directly compared to the global queue-length distribution of the system, for the heterogeneous case, we need to define $K$ types of global weighted queue-length distribution (see Definition~\ref{defn:global-queue-length}), where the weights are determined by the asymptotic properties of the graph structure: $(v_m,m\in\mathcal{M})$ and $(p_{k,m},k\in\mathcal{K},m\in\mathcal{M})$. 
Then, we compare the local queue-length distribution of dispatcher $i$ to the global \emph{weighted} queue-length distribution of the corresponding type as defined below.

\begin{defn}\label{defn:global-queue-length}
Consider any fixed $N\in\N$ and $k\in\mathcal{K}$. 
Given the global occupancy $\mathbf{q}^N=(q^N_{m,l},m\in\mathcal{M},l\in\N_0)$ of the $N$-th system, the \textbf{global weighted queue-length distribution} (GWQD) of type k is defined as $\big(x^N_{k,m,l},m\in\mathcal{M},l\in \N_0\big)$, where $$x^N_{k,m,l}=\frac{v_mp_{k,m}}{\delta_k}(q^N_{m,l+1}-q^N_{m,l}).$$
\end{defn}
Also, the local queue-length distribution is defined as follows.
\begin{defn}
Consider any fixed $N\in\N$ and $k\in\mathcal{K}$. Given the state $(X^N_j,j\in V^N)$ of the $N$-th system, the \textbf{local queue-length distribution} (LQD) of dispatcher $i\in W^N_k$ is defined as $(\hat{x}^N_{i,m,l},m\in\mathcal{M},l\in\N_0)$, where 
$$\hat{x}^N_{i,m,l}=\frac{|\{j\in V^N_m: \xi^N_{i,j}=1\text{ and }X^N_j=l\}|}{|\mathcal{N}^N_w(i)|}.$$
\end{defn}
Although the dispatcher following the JSQ($d$) policy selects a target server based on its LQD, if its LQD is close (in suitable sense) to its corresponding GWQD, then the selection can be viewed as if the decision was based on the GWQD. 
The latter case is easier to analyze. 
Hence, if a dispatcher's LQD is close to its corresponding GWQD, we call it a \emph{good dispatcher}:

\begin{defn}[$\varepsilon$-Good Dispatcher]
Consider any fixed $N\in\N$ and an $\varepsilon>0$. Given the state $(X^N_j,j\in V^N)$ of the $N$-th system. A dispatcher $i\in W^N_k$, $k\in\mathcal{K}$, is $\varepsilon$-good if 
\begin{equation}
    \sum_{m\in\mathcal{M}}\sum_{l\in\N_0}|\hat{x}^N_{i,m,l}-x^N_{k,m,l}|\leq \varepsilon.
\end{equation}
Also, a dispatcher is $\varepsilon$-bad if it is not $\varepsilon$-good.
\end{defn}

\subsubsection{Consequences of Clustered Proportional Sparsity}
The proof of Theorem~\ref{thm:general-conv-global-occup} relies on the idea that if the local occupancy of each dispatcher within a particular type evolves similar to the global occupancy of that type, then the process-level limiting behavior should not depend on any specific initial state. 
That is, it will enable us to go beyond the i.i.d.~assumption.
First step, for this approach to work, is to show that almost all dispatchers are $\varepsilon$-good for any $\varepsilon>0$. Here is where we need the property of clustered proportional sparsity.
This is stated in the next proposition.
\begin{prop}\label{prop: num-bad-disp}
Let $\{G^N\}_N$ be a sequence of clustered proportionally sparse graphs. For any $T\geq 0$ and $\varepsilon_1, \varepsilon_2>0$, 
\begin{equation}
    \PP\Big(\sup_{t\in [0,T]}\mathscr{B}^{\varepsilon_1}_N(t)\geq \varepsilon_2 |W^N|\Big)\xrightarrow{N\rightarrow\infty}0,
\end{equation}
where $\mathscr{B}^{\varepsilon_1}_N(t)$ is the number of $\varepsilon_1$-bad dispatchers at time $t$. 
\end{prop}
The intuition behind Proposition~\ref{prop: num-bad-disp} is that the servers of type $m\in\mathcal{M}$ with queue length $l\in\N_0$ forms a subset $U^N_{m,l}$ of the server set $V^N$. 
If this set is large, then by the clustered proportional sparsity, for any fixed $k\in\mathcal{K}$ and almost all $i\in W^N_k$, the fraction of the dispatcher $i$'s neighbors within $U^N_{m,l}$ is close to $\frac{|E^N_k(U^N_{m,l})|}{|E^N_k(V^N)|}$, which is close to $x^N_{k,m,l}$ for large enough $N$ by Condition~\ref{cond-1}. 
Also, in order to deal with the sum over $l\in \N_0$, we will need to establish uniform bounds of the tail of the occupancy process on any finite time interval. The complete proof is given in Appendix~\ref{app:proof-of-bad-dispatcher}.

\subsubsection{Coupling with an intermediate system}
The main methodology for the proof of Theorem~\ref{thm:general-conv-global-occup} is a stochastic coupling with a sequence $\{G'^N\}_{N\geq 1}$ of carefully constructed systems where the evolution of each system $G'^N$ can be coupled with that of the system $G^N$. 
For each $N$, the system $G'^N$ has the same sets of dispatchers and servers as $G^N$, i.e., $W'^N=W^N$ and $V'^N=V^N$. 
However, the task assignment in $G'^N$ happens differently. 
To describe the task assignment policy, let us introduce the following notations:
Let $X'^N_j(t)$ be the number of tasks (including those in service) in the queue of server $j\in V'^N$ at time $t$. 
Let $\mathbf{q}'^N(t)=\big(q'^N_{m,l}(t),m\in\mathcal{M},l\in\N_0\big)$ be the corresponding global occupancy at time $t$, which is defined in the same way as $\mathbf{q}^N$ for the system $G^N$. 
Then, the system $G'^N$ assigns tasks under the Global Weighted Shortest Queue (GWSQ($d$)) policy as described in Algorithm~\ref{alg:GWQD-JSQ(d)}.
The GWSQ($d$) policy is essentially a variant of the JSQ(d) policy since for each new task, the dispatcher selects a target set of servers of size $d$ according to the global weighted queue-length distribution.

\begin{algorithm}[h]\small
\caption{GWSQ($d$)}\label{alg:GWQD-JSQ(d)}
\While{A new task arrives at dispatcher $i\in W^N_k$, $k\in\mathcal{K}$}{
Get the current global occupancy $\mathbf{q}^N=(q^N_{m,l},m\in\mathcal{M},l\in\N_0)$;\\
Calculate the \textit{global weighted queue-length distribution} $\mathbf{x}^N_k=(x^N_{k,m,l},m\in\mathcal{M},l\in\N_0)$ of type $k$,
$$x^N_{k,m,l}=\frac{v_mp_{k,m}}{\delta_k}(q^N_{m,l+1}-q^N_{m,l});$$\\
Randomly select a set $\sel^N$ with size $d$ as the following:
\begin{itemize}
    \item  Let $
Y^N_{k,m,l}(t)\in \N_0$ be the number of servers of type $m\in\mathcal{M}$ with queue length \\
$l\in \N_0$ in the set $\sel^N$;
\item $(Y^N_{k,m,l}(t),m\in\mathcal{M},l\in N_0)$ satisfies 
$$\sum_{m\in\mathcal{M},l\in\N_0}Y^N_{k,m,l}(t)=d;$$
\item The probability of selecting $(Y^N_{k,m,l}(t),m\in\mathcal{M},l\in N_0)$ is 
$$\PP(Y^N_{k,m,l}(t),m\in\mathcal{M},l\in N_0)=\prod_{m\in\mathcal{M},l\in\N_0}{X^N_{k,m,l}(t)\choose Y^N_{k,m,l}(t)}/{N\choose d};$$
where $X^N_{k,m,l}= N\times x^N_{k,m,l}$.
\end{itemize}
Get $l^*=\min(l\in\N_0: \exists k\in\mathcal{K},m\in\mathcal{M}\text{ such that }Y^N_{k,m,l}>0)$;\\
Assign the task a type $m\in\mathcal{M}$ server with queue length $l^*$ with probability 
$$\frac{Y^N_{k,m,l^*}}{\sum_{m\in\mathcal{M}}Y^N_{k,m,l^*}}.$$

}
\end{algorithm}

Next, we couple the evolution of the system $G'^N$ with that of the system $G^N$ by the \emph{optimal coupling} method. 
The optimal coupling for two stochastic processes is similar to the maximal coupling for two discrete random variables (say, $X$ and $Y$), maximizing the probability $\PP(X=Y)$.

\paragraph{Optimal Coupling.} Fix any $N$. 
In both systems, within the pool of servers of each type, arrange the servers in the non-decreasing order of their queue lengths (ties are broken arbitrarily).
Now, couple the evolution of the system $G^N$ with the system $G'^N$ in the following way:
\begin{itemize}
    \item \textbf{Departure.} For any $m\in\mathcal{M}$ and $n = 1, \ldots, |V^N_m|$, synchronize the departure epochs of the $n^{th}$ ordered servers of type $m$ in the two systems.
    \item \textbf{Arrival.} The coupling of arrivals is the tricky part. 
    For this, first synchronize the arrival epochs at each dispatcher $i$ in both systems $G'^N$ and $G^N$. 
    At an arrival epoch of dispatcher $i\in W^N_k$, let $(\hat{x}^N_{i,m,l},m\in\mathcal{M},l\in\N_0)$ be the \emph{local} empirical distribution of dispatcher $i$ in the system $G^N$ and $(x'^N_{k,m,l},m\in\mathcal{M},l\in\N_0)$ be the weighted \emph{global} empirical distribution of type $k$ dispatchers in the system $G'^N$. 
    Then, in the system $G^N$, probability that the task will be assigned to a server of type $m\in\mathcal{M}$ with queue length $l\in\N_0$ is given by 
    \begin{equation}\label{eq:p-ml}
        \begin{split}
            p^N_{m,l}(i)&:= \frac{\sum_{r=1}^d\sum_{r_1=1}^{r}\frac{r_1}{r}{|\mathcal{N}^N_w(i)|\hat{x}^N_{i,m,l}\choose r_1}{|\mathcal{N}^N_w(i)|\sum_{\mathcal{M}\setminus \{m\}}\hat{x}^N_{i,m,l}\choose r-r_1} {|\mathcal{N}^N_w(i)|\sum_{\mathcal{M}}\sum_{l'\geq l+1}\hat{x}^N_{i,m,l'}\choose d-r}}{{|\mathcal{N}^N_w(i)|\choose d}}.
        \end{split}
    \end{equation}
  In the system $G'^N$, the probability that the task will be assigned to a server of type $m\in\mathcal{M}$ with queue length $l\in\N_0$ is given by
    \begin{equation}\label{eq:p-'-ml}
        \begin{split}
            p'^N_{m,l}(k)&:= \frac{\sum_{r=1}^d\sum_{r_1=1}^{r}\frac{r_1}{r}{ X'^N_{k,m,l}\choose r_1}{ \sum_{\mathcal{M}\setminus\{m\}}X'^N_{k,m,l} \choose r-r_1}{\sum_{\mathcal{M}}\sum_{l'\geq l+1} X'^N_{k,m,l'} \choose d-r}}{{N \choose d}}.
        \end{split}
    \end{equation}
    For convenience, we denote $p^N_{m,l}(i)$ and $p'^N_{m,l}(k)$ as $p^N_{m,l}$ and $p'^N_{m,l}$, respectively.
    Denote $\bar{p}^N_{m,l}=\min(p^N_{m,l},p'^N_{m,l})$ for $m\in\mathcal{M}$ and $l\in \N_0$. 
    
    Now, to couple the task assignment, let us draw a $\mathrm{Uniform}[0,1]$ random variable $U$, independently of any other processes and across various arrival epochs. 
    $U$ is used to generate the random variables $(M^N,L^N)\in\mathcal{M}\times \N_0$ and $(M'^N,L'^N)\in\mathcal{M}\times \N_0$ for the system $G^N$ and the system $G'^N$, respectively. 
    In the system $G^N$, set $(M^N,L^N)=(m,l)\in \mathcal{M}\times \N_0$, if 
    \begin{equation}\label{eq:ML-GN}
        \begin{split}
            U\in&\Big[\sum_{m'=1}^{m-1}\sum_{l'=0}^{\infty}\bar{p}^N_{m',l'}+\sum_{l'=0}^{l-1}\bar{p}^N_{m,l'},\sum_{m'=1}^{m-1}\sum_{l'=0}^{\infty}\bar{p}^N_{m',l'}+\sum_{l'=0}^{l}\bar{p}^N_{m,l'}\big)\\
            &\bigcup \big[\bar{p}^N+\sum_{m'=1}^{m-1}\sum_{l'=0}^{\infty}(p^N_{m',l'}-\bar{p}^N_{m',l'})+\sum_{l'=0}^{l-1}(p^N_{m,l'}-\bar{p}^N_{m,l'}),\\
            &\qquad \bar{p}^N+\sum_{m'=1}^{m-1}\sum_{l'=0}^{\infty}(p^N_{m',l'}-\bar{p}^N_{m',l'})+\sum_{l'=0}^{l}(p^N_{m,l'}-\bar{p}^N_{m,l'})\Big),
        \end{split}
    \end{equation}
    where $\bar{p}^N=\sum_{m'=1}^{M}\sum_{l'=0}^{\infty}\bar{p}^N_{m',l'}$, and assign the task to a server of type $m$ with queue length~$l$. 
    Similarly, in the system $G'^N$, set $(M'^N,L'^N)=(m,l)\in \mathcal{M}\times \N_0$, if 
    \begin{equation}\label{eq:M'L'-GN}
        \begin{split}
            U\in&\Big[\sum_{m'=1}^{m-1}\sum_{l'=0}^{\infty}\bar{p}^N_{m',l'}+\sum_{l'=0}^{l-1}\bar{p}^N_{m,l'},\sum_{m'=1}^{m-1}\sum_{l'=0}^{\infty}\bar{p}^N_{m',l'}+\sum_{l'=0}^{l}\bar{p}^N_{m,l'}\big)\\
            &\bigcup \big[\bar{p}^N+\sum_{m'=1}^{m-1}\sum_{l'=0}^{\infty}(p'^N_{m',l'}-\bar{p}^N_{m',l'})+\sum_{l'=0}^{l-1}(p'^N_{m,l'}-\bar{p}^N_{m,l'}),\\
            &\qquad \bar{p}^N+\sum_{m'=1}^{m-1}\sum_{l'=0}^{\infty}(p'^N_{m',l'}-\bar{p}^N_{m',l'})+\sum_{l'=0}^{l}(p'^N_{m,l'}-\bar{p}^N_{m,l'})\Big),
        \end{split}
    \end{equation}
    and assign the task to a server of type $m$ with queue length $l$. 
\end{itemize}
As alluded to before, the above coupling is constructed in a way that maximizes the probability of the two systems to assign an arriving task to some server with the same queue length.
Next, the difference in the occupancy processes of the two systems, on any finite time interval, can be upper-bounded by the number of times the two systems assign to two different queue lengths. 
This is formalized by the notion of \emph{mismatch} below, which was originally introduced in~\cite{MBLW16-3}.

\begin{defn}[Mismatch]
At an arrival epoch, the system $G^N$ and the system $G'^N$ are said to mismatch if $(M^N,L^N)\neq (M'^N,L'^N)$, that is, the arriving task is not assigned to servers of the same type with the same queue length in the two systems. Denote by $\Delta^N(t)$ the cumulative number of times the systems mismatch in queue length up to time $t$.
\end{defn} 
The next proposition provides a deterministic bound on the difference between the occupancy processes of the two systems in terms of the number of mismatches.
\begin{prop}\label{prop: mismatch}
For any $N\geq 1$, consider the system $G^N$ and the system $G'^N$ coupled as above. Then the following holds almost surely on the coupled probability space: for $t\geq 0$, 
\begin{equation}\label{eq:mismatch-bound}
    \sum_{m\in\mathcal{M}}\sum_{l\in\N_0}|Q^N_{m,l}(t)-Q'^N_{m,l}(t)|\leq 2\Delta^N(t),
\end{equation}
provided the inequality holds at $t=0$. $Q^N_{m,l}(t)$ and $Q'^N_{m,l}(t)$ represent the number of servers of type $m\in\mathcal{M}$ with queue length at least $l\in\N_0$ in the system $G^N$ and the system $G'^N$ at time $t$, respectively.
\end{prop}
Bounds of the form as given in~\eqref{eq:mismatch-bound} was originally established in~\cite[ Proposition~4]{MBLW16-3}, and is later used in various contexts~\cite{MBL17,RM22}.
The proof does not depend on any specific assignment policy and relies on showing inductively that if the inequality in~\eqref{eq:mismatch-bound} holds before an event time epoch, then it is preserved after the event time epoch as well.
The proof of Proposition~\ref{prop: mismatch} can be obtained following the similar arguments. 
We omit the details. 
 
\begin{lemma}\label{cor:Q-x}
Given $\sum_{m\in\mathcal{M}}\sum_{l\in\N_0}|Q^N_{m,l}-Q'^N_{m,l}|\leq 2\Delta^N$. Then, there exist $N_0\in\N_0$ and a positive constant $L$ such that for any $k\in\mathcal{K}$, 
\begin{equation}\label{eq:x-delta}
    \sum_{m\in\mathcal{M}}\sum_{l\in\N_0}|x^N_{k,m,l}-x'^N_{k,m,l}|\leq L \Delta^N/N, \quad \forall N\geq N_0.
\end{equation}
\end{lemma}
\begin{proof}
By the model assumption, there exists $N_0\in\N_0$ such that for all $N\geq N_0$, $|V^N_m|\geq \frac{1}{2}Nv_m$, $\forall m\in\mathcal{M}$, which gives us that 
\begin{equation}
    \begin{split}
        \sum_{m\in\mathcal{M}}\sum_{l\in\N_0}|x^N_{k,m,l}-x'^N_{k,m,l}|
        & = \sum_{m\in\mathcal{M}}\sum_{l\in\N_0}\frac{v_mp_{k,m}}{\delta_k}|Q^N_{m,l}-Q'^N_{m,l}|/|V^N_m|\\
        & \leq \sum_{m\in\mathcal{M}}\sum_{l\in\N_0}\frac{2p_{k,m}}{\delta_k}|Q^N_{m,l}-Q'^N_{m,l}|/N
         \leq L \Delta^N/N,
    \end{split}
\end{equation}
where $L=4\max_{k\in\mathcal{K},m\in\mathcal{M}}\frac{p_{k,m}}{\delta_k}$.
\end{proof}

The final ingredient that we need is the probability of mismatch in a particular epoch, under the optimal coupling method.
The next lemma bounds this probability in terms of the $\ell_1$-distance between the LQD of the $G_N$ system and the GWQD of the $G'^N$ system.
\begin{lemma}\label{lem:jsq-Lipschitz}
Consider an arrival epoch at dispatcher $i$, and assume that in this epoch, 
the LQD in the system $G^N$ is given by $(\hat{x}^N_{i,m,l},m\in\mathcal{M},l\in\N_0)$  and the GWQD of type-$k$ servers in the system $G'^N$ is given by $(x'^N_{k,m,l},m\in\mathcal{M},l\in\N_0)$. 
Then there exists exists a finite positive constant $L_1$, such that for all large enough $N$,
\begin{equation}
    \PP(\text{Mismatch})\leq L_1 \sum_{m\in\mathcal{M}}\sum_{l\in\N_0}|\hat{x}^N_{i,m,l}-x'^N_{k,m,l}|.
\end{equation}
\end{lemma}
The key step in the proof of Lemma~\ref{lem:jsq-Lipschitz} is that, given the queue-length distribution $\mathbf{x}=(x_{m,l},m\in\mathcal{M},l\in\N_0)$, the probability $p_{m,l}$ that a task will be assigned to a server of type $m\in\mathcal{M}$ with queue length $l\in\N_0$ can be approximated by 
\begin{equation*}
    p_{m,l}\approx\sum_{r=1}^d\sum_{r_1=1}^{r}\frac{r_1}{r}\frac{d!}{r_1!(r-r_1)!(d-r)!}\big(x_{m,l}\big)^{r_1}\big(\sum_{\mathcal{M}\setminus\{m\}}x_{m,l}\big)^{r-r_1}\big(\sum_{\mathcal{M}}\sum_{l'\geq l+1}x_{m,l'}\big)^{d-r}
\end{equation*} 
and that the function $x^k$ is Lipschitz for $x\in[0,1]$. 
The complete proof is given in Appendix~\ref{app:lipschitz-jsq-d}.

\subsubsection{Proof of Theorem~\ref{thm:general-conv-global-occup}}\label{sec:proof-thm-3.10}
Now we have all the ingredients to prove Theorem~\ref{thm:general-conv-global-occup}.
Let us explain the high-level proof scheme first.

\noindent
\textbf{Step~1.} Using the optimal coupling, we will show that the global occupancy processes $\{\mathbf{q}^N(\cdot)\}_N$ and $\{\mathbf{q}'^N(\cdot)\}_N$ must converge to the same limit process as $N\to\infty$, if their initial states are the same, $X^N_j(0)=X'^N_j(0)$ for all $j$. In other words, with the same initial states,
\begin{equation*}
    \lim_{N\rightarrow\infty}\mathbf{q}^N(\cdot)=\lim_{N\rightarrow\infty}\mathbf{q}'^N(\cdot).
\end{equation*}

\noindent
\textbf{Step~2.} Since there is no graph structure in the system $G'^N$, all servers of the same type in the system $G'^N$ are exchangeable. Hence, $\mathbf{q}'^N(\cdot)$ is Markovian, which implies that given $\mathbf{q}'^N(0)$, its evolution does not depend on how individual $X'^N_j(0)$'s are distributed. 
Denote the system $G'^N$ with i.i.d. assumption as $G'^N_1$, where the i.i.d. assumption refers to that for any $m\in\mathcal{M}$, $X^N_j(0)$, $j\in V^N_m$, are i.i.d.. 
Also, denote the system $G'^N$ without the i.i.d.~assumption as $G'^N_2$. 
Their occupancy processes are $\mathbf{q}'^N_1(\cdot)$ and $\mathbf{q}'^N_2(\cdot)$, respectively. 
Since task assignment policy in $G'^N$ does not distinguish between two servers having same type and queue lengths, 
by a natural coupling, $\mathbf{q}'^N_1(t)=\mathbf{q}'^N_2(t)$ holds for all $t\geq 0$, implying that 
\begin{equation*}
    \lim_{N\rightarrow\infty}\mathbf{q}'^N_1(\cdot)=\lim_{N\rightarrow\infty}\mathbf{q}'^N_2(\cdot).
\end{equation*}

\noindent
\textbf{Step~3.} Denote the system $G^N$ with i.i.d.~assumption as $G^N_1$ and the system $G'^N$ without the i.i.d.~assumption as $G^N_2$ and their occupancy processes by $\mathbf{q}^N_1(\cdot)$ and $\mathbf{q}^N_2(\cdot)$, respectively.
Combining Step~1 and Step~2, the following equation holds. With the same initial global occupancy state, 
\begin{equation*}
    \lim_{N\rightarrow\infty}\mathbf{q}^N_1(\cdot)=\lim_{N\rightarrow\infty}\mathbf{q}'^N_1(\cdot)=\lim_{N\rightarrow\infty}\mathbf{q}'^N_2(\cdot)=\lim_{N\rightarrow\infty}\mathbf{q}^N_2(\cdot),
\end{equation*}
where the first and last equalities are due to Step 1 and the second equality is due to Step~2.

\noindent
\textbf{Step~4.} Use Theorem~\ref{thm:conv-global-occup} to note that when the sequence $\{G^N\}_N$ satisfies the assumption that for each $m\in\mathcal{M}$, $X^N_j(0)$, $j\in V^N_m$, are i.i.d., the scaled global occupancy process $\mathbf{q}^N$ converge weakly to $\mathbf{q}$ described by the system of ODEs in~\eqref{eq:ode-sys-2}.

\noindent
\textbf{Step~5.} By Steps 3 and 4, Theorem~\ref{thm:general-conv-global-occup} holds. \\

In the above proof scheme, observe that all that remains is to show Step~1, which is given below.

\begin{proof}[Proof of Theorem~\ref{thm:general-conv-global-occup}]
For Step~1 described in the proof scheme above, by Proposition~\ref{prop: mismatch}, it is sufficient to show that
for any $\varepsilon^*>0$ and $\delta^*>0$, there exists an $N_0\geq 1$ such that 
\begin{equation}\label{eq:delta-small}
    \PP\big(\sup_{t\in[0,T]}\Delta^N(t)/N\geq \varepsilon^*\big)\leq \delta^*,\quad \forall N\geq N_0.
\end{equation}
Fix an $\varepsilon>0$, which will be chosen later. 
Let $\mathscr{G}^{\varepsilon}_N(t)$ and $\mathscr{B}^{\varepsilon}_N(t)$ be the number of $\varepsilon$-good and $\varepsilon$-bad dispatchers in the system $G^N$ at time $t$, respectively. 
We couple the evolution of the system $G^N$ with that of the system $G'^N$ by the optimal coupling method. 
In system $G^N$, let $(x^N_{k,m,l}(t),m\in\mathcal{M},l\in\N_0)$ be the global weighted queue-length distribution of type $k\in\mathcal{K}$ and $(\hat{x}^N_{i,m,l}(t),m\in\mathcal{M},l\in\N_0)$ be the local queue-length distribution of the dispatcher $i\in W^N_k$, $k\in\mathcal{K}$. 
Also, let $(x'^N_{k,m,l}(t),m\in\mathcal{M},l\in\N_0)$ be the global weighted queue-length distribution of type $k\in\mathcal{K}$ in system $G'^N$. 
Denote $\rho^N_k(t)=\sum_{m\in\mathcal{M}}\sum_{l\in\N_0}|x^N_{k,m,l}(t)-x'^N_{k,m,l}(t)|$. At an arrival epoch $t\geq 0$, if a task arrives at an $\varepsilon$-good dispatcher $i\in W^N_k$, then 
\begin{align}
    &\sum_{m\in\mathcal{M}}\sum_{l\in\N_0}|\hat{x}^N_{i,m,l}(t-)-x'^N_{k,m,l}(t-)|\nonumber\\
    &\leq \sum_{m\in\mathcal{M}}\sum_{l\in\N_0}|\hat{x}^N_{i,m,l}(t-)-x^N_{k,m,l}(t-)|+\sum_{m\in\mathcal{M}}\sum_{l\in\N_0}|x^N_{k,m,l}(t-)-x'^N_{k,m,l}(t-)|
    =\varepsilon+\rho^N_k(t).
\end{align}
Recall the uniform random variable $U$ and $\bar{p}^N_{m,l}$ defined in the description of the optimal coupling method. The probability that the systems have a mismatch at such arrival epoch is bounded by 
\begin{equation}
    \begin{split}
        \PP\Big(U\notin [0,\sum_{m\in\mathcal{M}}\sum_{l\in\N_0}\bar{p}^N_{m,l}]\Big)&=1-\sum_{m\in\mathcal{M}}\sum_{l\in\N_0}\bar{p}^N_{m,l}=\sum_{m\in\mathcal{M}}\sum_{l\in\N_0}p'^N_{m,l}-\sum_{m\in\mathcal{M}}\sum_{l\in\N_0}\bar{p}^N_{m,l}\\
        & \leq  \sum_{m\in\mathcal{M}}\sum_{l\in\N_0}|p'^N_{m,l}-p^N_{m,l}|\leq L_1 \sum_{m\in\mathcal{M}}\sum_{l\in\N_0}|\hat{x}^N_{i,m,l}(t-)-x'^N_{k,m,l}|\\
        & \leq   L_1(\rho^N_k(t)+\varepsilon),
    \end{split}
\end{equation}
where the second inequality is from Lemma~\ref{lem:jsq-Lipschitz}.
At an arrival epoch $t\geq 0$, if a task arrives at an $\varepsilon$-bad dispatcher $i\in W^N_k$, then with probability at most one, the systems have a mismatch. 
Due to the Poisson thinning property, we can construct an independent unit-rate Poisson process $(Z(t))_{t\geq 0}$, so that $\Delta^N(t)$ can be upper bounded by a random time change of $Z$ as the following: for all $t\in[0,T]$,
\begin{equation}\label{eq:delta-bound}
    \begin{split}
        \Delta^N(t)& \leq  Z\Big(\sum_{k\in\mathcal{K}}\sum_{i\in W^N_k}\lambda\int_0^t[\mathds{1}_{(i\in \mathscr{G}^{\varepsilon}_N(s-))}L_1(\rho^N_k(s-)+\varepsilon)+\mathds{1}_{(i\in\mathscr{B}^{\varepsilon}_N(s-))}\cdot 1]ds\Big)\\
        & \leq  Z\Big(\sum_{k\in\mathcal{K}}\sum_{i\in W^N_k}\lambda\int_0^t[\mathds{1}_{(i\in \mathscr{G}^{\varepsilon}_N(s-))}L_1(L\Delta^N(s-)/N+\varepsilon)+\mathds{1}_{(i\in\mathscr{B}^{\varepsilon}_N(s-))}\cdot 1]ds\Big)\\
        &= Z\Big(\lambda\int_0^t[\mathscr{G}^{\varepsilon}_N(s-) L_1(L\Delta^N(s-)/N+\varepsilon)+\mathscr{B}^{\varepsilon}_N(s-)\cdot 1]ds\Big),
    \end{split}
\end{equation}
where the second inequality is due to Lemma~\ref{cor:Q-x}.
By Proposition~\ref{prop: num-bad-disp}, we have that for any $\varepsilon'>0$, there exists an $N(\varepsilon')$ such that for all $N\geq N(\varepsilon')$, 
\begin{equation}\label{eq: bad-set-bound}
    \PP(\sup_{t\in[0,T]}\mathscr{B}^{\varepsilon}_N(t)\geq \varepsilon'|W^N|)\leq \frac{\varepsilon'}{2}.
\end{equation}
Hence, by \eqref{eq:delta-bound}, \eqref{eq: bad-set-bound}, and Tonelli's theorem, we have that for all $N\geq N(\varepsilon')$ and $t\in[0,T]$, \begin{equation}
    \E\Big(\frac{\Delta^N(t)}{N}\Big)\leq \lambda\int_0^t\Big[L_1 \Big(L \frac{W(N)}{N} \frac{\E(\Delta^N(s-))}{N}+\varepsilon\Big)+\frac{W(N)}{N}\frac{3\varepsilon'}{2}\Big]ds.
\end{equation}
Also, by the assumption that $\lim_{N\rightarrow\infty}\frac{W(N)}{N}=\xi$, there exists $N_0$ such that $\frac{W(N)}{N}\leq 2\xi$. Hence, we have that for all $N\geq \max(N(\varepsilon'),N_0)$ and $t\in[0,T]$, \begin{equation}\label{eq:exp-delta-N}
    \E\Big(\frac{\Delta^N(t)}{N}\Big)\leq \lambda\int_0^t\Big[L_1 \Big( 2L\xi \frac{\E(\Delta^N(s-))}{N}+\varepsilon\Big)+3\xi\varepsilon'\Big]ds.
\end{equation}
By applying Grönwall's inequality to \eqref{eq:exp-delta-N}, we have 
\begin{equation}\label{eq:exp-delta-bound}
    \E\Big(\frac{\Delta^N(t)}{N}\Big)\leq \lambda (L_1\varepsilon+3\xi\varepsilon')t\exp(2LL_1\xi\lambda t).
\end{equation}
Since $\Delta^N(t)$ is nonnegative, by Markov's inequality and \eqref{eq:exp-delta-bound}, we have 
\begin{equation}
    \PP(\sup_{t\in[0,T]}\Delta^N(t)/N\geq \varepsilon^*)\leq \frac{1}{\varepsilon^*}\lambda (L_1\varepsilon+3\xi\varepsilon')t\exp(2LL_1\xi\lambda t)
\end{equation}
and we can choose small enough $\varepsilon$ and $\varepsilon'$ such that \eqref{eq:delta-small} holds.
\end{proof}

\section{Proof of Interchange of Limits}\label{sec:interchange-of-limits}
\subsection{Properties of the Limiting System of ODEs}\label{ssec:prop-ode}
First, we define the fixed point of the ODE~\eqref{eq:ode-sys-2}.
Recall $\delta_k=\sum_{m\in\mathcal{M}}p_{k,m}v_m$ and $\tilde{q}_{k,l}(t)=\sum_{m\in\mathcal{M}}\frac{v_mp_{k,m}}{\delta_k}q_{m,l}(t)$. Let $\mathbf{q}^*=(q^*_{m,l}\in \R_+,m\in\mathcal{M},l\in\N_0)$ be a fixed point of the ODE~\eqref{eq:ode-sys-2} if for all $ m\in\mathcal{M},l\in \N$,
\begin{equation}\label{eq:fixed-point}
   u_m(q^*_{m,l}-q^*_{m,l+1})=\lambda\xi(q^*_{m,l-1}-q^*_{m,l})\sum_{k\in\mathcal{K}}\frac{p_{k,m}w_k}{\delta_k}\frac{(\tilde{q}^*_{k,l-1})^d-(\tilde{q}^*_{k,l})^d}{\tilde{q}^*_{k,l-1}-\tilde{q}^*_{k,l}},
\end{equation}
with $q^*_{m,0}=1$, $m\in\mathcal{M}$. The next proposition shows some important properties of the fixed point $\mathbf{q}$ of the ODE \eqref{eq:ode-sys-2}.
\begin{prop}\label{prop:dbly-decay}
If there exists a fixed point $\mathbf{q}^*$ of the ODE \eqref{eq:ode-sys-2} such that for each $m\in\mathcal{M}$, $q_{m,0}=1$ and $q_{m,l}\xrightarrow{l\rightarrow\infty}0$, then 
for each $m\in\mathcal{M}$, the sequence $\{q_{m,l},l\in\N_0\}$ decreases doubly exponentially. 
\end{prop}
The proof of Proposition~\ref{prop:dbly-decay} is provided in Appendix~\ref{app:proof-doubly-decay}. The key observation used in the proof is that by \eqref{eq:fixed-point}, $q^*_{m,l}$ can be expressed in terms of $q^*_{m,l-1}$ and $q^*_{m,l-2}$. 
Thus, we can recursively characterize the values of $q^*_{m,l}$, $l\geq 2$, if we know $q^*_{m,0}$ and $q^*_{m,1}$, $m\in\mathcal{M}$.

By Proposition~\ref{prop:dbly-decay}, we know that if $\mathbf{q}^*$ is a fixed point of the ODE~\eqref{eq:ode-sys-2} and, for all $m\in\mathcal{M}$, $q^*_{m,l}\xrightarrow{l\rightarrow\infty}0$, then such $\mathbf{q}^*$ must be in $\mathcal{S}$ so we only need to show that such $\mathbf{q}^*$ exists.
For the proof of the existence of such $\mathbf{q}^*$, we need a technical lemma, which will be used in \eqref{eq:limits-x*}.
\begin{lemma}\label{lem:alpha-m}
Consider a sequence $\{G^N\}_N$ satisfying Condition~\ref{cond-1}. If $\{G^N\}_N$ is proportionally sparse and in the subcritical regime, then for any $(\alpha_1,...,\alpha_M)\in[0,1]^M$ with $\sum_{m\in\mathcal{M}}\alpha_m>0$, the following holds:
\begin{equation}
    \Big(\sum_{m\in\mathcal{M}}\alpha_m v_m u_m\Big)^{-1}\lambda\xi\sum_{k\in\mathcal{K}}w_k\Big(\frac{\sum_{m\in\mathcal{M}} \alpha_m p_{k,m} v_m}{\delta_k}\Big)^d\leq \rho<1.
\end{equation}
\end{lemma}
The proof of Lemma~\ref{lem:alpha-m} is provided in Appendix~\ref{app:proof-lem-alpha-m}.

\begin{proof}[Proof of Theorem~\ref{thm: exists-unique}]
We prove the \textit{existence} of the fixed point first. From \eqref{eq:fixed-point}, we know that if $(q^*_{m,1},m\in\mathcal{M})$ are fixed, then all $(q^*_{m,l},m\in\mathcal{M},l\geq 2)$ are determined as well. Hence, $\mathbf{q}^*$ can be the viewed as the function of $(q^*_{m,1},m\in\mathcal{M})$. Moreover, in the steady state, $\sum_{m\in\mathcal{M}}q^*_{m,1}=\lambda\xi$, which implies that $q^*_{M,1}$ can be decided by the values of $q^*_{m,1}$, $m\in\mathcal{M}\setminus \{M\}$.
Hence, we construct the sequence $\mathbf{q}(\bar{\alpha})=(q_{m,l}(\bar{\alpha}),m\in\mathcal{M},l\in\N_0)$ as functions of the vector $\bar{\alpha}=(\alpha_1,...,\alpha_{M-1})\in (0,1)^{M-1} $ as follows:
\begin{align}
    & q_{m,0}(\bar{\alpha})=1, \forall m\in\mathcal{M}, \notag \\
    & q_{m,1}(\bar{\alpha})=\alpha_m,\quad m\in \mathcal{M}\setminus\{M\}, \quad \mbox{and} \quad q_{M,1}=\frac{\lambda\xi-\sum_{m\in \mathcal{M}\setminus \{
    M\}}\alpha_m v_m u_m}{v_M u_M}, \notag \\
    & u_m(q_{m,l}(\bar{\alpha})-q_{m,l+1}(\bar{\alpha}))=\lambda\xi(q_{m,l-1}(\bar{\alpha})-q_{m,l}(\bar{\alpha}))\sum_{k\in\mathcal{K}}\frac{p_{k,m}w_k}{\delta_k}\frac{(\tilde{q}_{k,l-1}(\bar{\alpha}))^d-(\tilde{q}_{k,l}(\bar{\alpha}))^d}{\tilde{q}_{k,l-1}(\bar{\alpha})-\tilde{q}_{k,l}(\bar{\alpha})}, l \ge 1. \label{eq:recursion-q-alpha}
\end{align}
Since, for all $m\in\mathcal{M}$, $q_{m,1}(\bar{\alpha})$ should be in $(0,1)$, then $\bar{\alpha}=(\alpha_1,...,\alpha_{M-1})$ must lie in the polyhedron $\mathbf{P}_1$ defined as follows:
\begin{equation*}
    \begin{split}
        &\mathbf{P}_1\coloneqq\Big\{\alpha_m\in\big(\max(0,\frac{\lambda\xi-\sum_{m'\in\mathcal{M}\setminus \{m\}}v_{m'}u_{m'}}{v_mu_m}),\min(\frac{\lambda\xi}{v_m u_m},1)\big), \forall m\leq M-1,\\
        &\hspace{6cm}\text{ and }\lambda\xi-v_Mu_M<\sum_{m\leq M-1}\alpha_m v_m u_m<\lambda\xi \Big\}.
    \end{split}
\end{equation*}
For all $\bar{\alpha}\in \mathbf{P}_1$, we have $1=q_{m,0}(\bar{\alpha})>q_{m,1}(\bar{\alpha})>0$, $\forall m\in\mathcal{M}$. Consider $l=2$. By \eqref{eq:recursion-q-alpha}, we have that when $\alpha_m=0$, $m\in\mathcal{M}\setminus \{M\}$,
    $$u_m(0-q_{m,2}(\bar{\alpha}))=\lambda\xi(1-0)\sum_{k\in\mathcal{K}}\frac{p_{k,m}w_k}{\delta_k}\frac{1-(\tilde{q}_{k,1}(\bar{\alpha}))^d}{1-\tilde{q}_{k,1}(\bar{\alpha})}$$
    implying that $q_{m,2}(\bar{\alpha})<0$;
     when $\alpha_m=1$, $m\in\mathcal{M}\setminus \{M\}$,
    $$u_m(1-q_{m,2}(\bar{\alpha}))=0$$
    implying that $q_{m,2}(\bar{\alpha})=1>0$;
    when $\alpha_m=\frac{\lambda\xi}{v_m u_m}$, $m\in\mathcal{M}\setminus \{M\}$,
    $$u_m(\frac{\lambda\xi}{v_m u_m}-q_{m,2}(\bar{\alpha}))=\lambda\xi (1-\frac{\lambda\xi}{v_m u_m})\sum_{k\in\mathcal{K}}\frac{p_{k,m}w_k}{\delta_k}\frac{1-(\tilde{q}_{k,1}(\bar{\alpha}))^d}{1-\tilde{q}_{k,1}(\bar{\alpha})}$$
    implying that 
    \begin{equation*}
        \begin{split}
            q_{m,2}(\bar{\alpha})&=\frac{\lambda\xi}{v_m u_m}-\frac{\lambda\xi}{u_m}(1-\frac{\lambda\xi}{v_m u_m})\sum_{k\in\mathcal{K}}\frac{p_{k,m}w_k}{\delta_k}\frac{1-(\tilde{q}_{k,1}(\bar{\alpha}))^d}{1-\tilde{q}_{k,1}(\bar{\alpha})}\\
            &> \frac{\lambda\xi}{v_m u_m}-\frac{\lambda\xi}{u_m}(1-\frac{\lambda\xi}{v_m u_m})\sum_{k\in\mathcal{K}}\frac{p_{k,m}w_k}{\delta_k}\\
            &> \frac{\lambda\xi}{v_m u_m}-\frac{\lambda\xi}{u_m}(1-\frac{\lambda\xi}{v_m u_m})\sum_{k\in\mathcal{K}}\frac{p_{k,m}w_k}{p_{k,m}v_m}\\
            &= \frac{\lambda\xi}{v_m u_m}-\frac{\lambda\xi}{v_m u_m}(1-\frac{\lambda\xi}{v_m u_m})>0.
        \end{split}
    \end{equation*}
Let $r_{m,1}$, $m\leq M-1$ be the maximum number which satisfies the following:
\begin{enumerate}[(1)]
    \item $r_{m,1}<\min(\frac{\lambda\xi}{v_m u_m},1)$,
    \item $\exists \bar{\alpha}\in \mathbf{P}_1$ with $\alpha_m=r_{m,1}$ such that $q_{m,2}(\bar{\alpha})=0$.
\end{enumerate}
Define $\mathbf{P}'_1\subseteq\mathbf{P}_1$ as the following:
\begin{equation*}
    \begin{split}
        &\mathbf{P}'_1\coloneqq\Big\{\alpha_m\in\big(\max(r_{m,1},\frac{\lambda\xi-\sum_{m\leq M-1}v_{m'}u_{m'}}{v_mu_m}),\min(\frac{\lambda\xi}{v_m u_m},1)\big), \forall m\leq M-1,\\
        &\hspace{6.5cm}\text{ and }\lambda\xi-v_Mu_M<\sum_{m\leq M-1}\alpha_m v_m u_m<\lambda\xi \Big\}.
    \end{split}
\end{equation*}
Again, by using \eqref{eq:recursion-q-alpha}, we get that when $\sum_{m\leq M-1}\alpha_mu_m=\lambda\xi-v_Mu_M$ (i.e., $q_{M,1}(\bar{\alpha})=1$), 
    $$u_M(1-q_{M,2}(\bar{\alpha}))=0$$
    implying that $q_{M,2}(\bar{\alpha})=1>0$; when $\sum_{m\leq M-1}\alpha_mu_m=\lambda\xi$ (i.e., $q_{M,1}(\bar{\alpha})=0$), 
    $$u_M(0-q_{M,2}(\bar{\alpha}))=\lambda\xi(1-0)\sum_{k\in\mathcal{K}}\frac{p_{k,M}w_k}{\delta_k}\frac{1-(\tilde{q}_{k,1}(\bar{\alpha}))^d}{1-\tilde{q}_{k,1}(\bar{\alpha})}$$
    implying that $q_{M,2}(\bar{\alpha})<0$.

Let $r_1$ be the minimum number which satisfies the following:
\begin{enumerate}
    \item $r_1<\lambda\xi$,
    \item There exists $\bar{\alpha}\in\mathbf{P}'_1$ such that $\sum_{m\leq M-1}\alpha_m v_m u_m=r_1$ and $q_{M,2}(\bar{\alpha})=0$.
\end{enumerate}
Define $\mathbf{P}_2\subseteq\mathbf{P}'_1\subseteq\mathbf{P}_1$ as the following:
\begin{equation*}
    \begin{split}
        &\mathbf{P}_2\coloneqq\Big\{\alpha_m\in\big(\max(r_{m,1},\frac{\lambda\xi-\sum_{m'\leq M-1}v_{m'}u_{m'}}{v_mu_m}),\min (\frac{r_1}{v_m u_m},1)\big),\forall m\leq M-1\\
        &\hspace{6cm}\text{ and } \lambda\xi-v_M u_M\leq \sum_{m\leq M-1}\alpha_m v_m u_m\leq r_1\Big\}.
    \end{split}
\end{equation*}
Hence, for all $\bar{\alpha}\in \mathbf{P}_2$, we have $1=q_{m,0}(\bar{\alpha})>q_{m,1}(\bar{\alpha})>q_{m,2}(\bar{\alpha})>0$, $\forall m\in\mathcal{M}$. Continuing this process, we can define a sequence $\{\mathbf{P}_1\supseteq \mathbf{P}_2\supseteq\cdots\}$ of polyhedra such that  for all $\bar{\alpha}\in \mathbf{P}_n$, we have $1=q_{m,0}(\bar{\alpha})>q_{m,1}(\bar{\alpha})>\cdots>q_{m,n}(\bar{\alpha})>0$, $\forall m\in\mathcal{M}$. Thus, we can get decreasing sequences $\{q_{m,l}(\bar{\alpha})\}_{l\in\N_0}$, $m\in\mathcal{M}$ for some $\bar{\alpha}$. Since $q_{m,l}\geq 0$, $\forall m\in\mathcal{M}, l\in\N_0$, then $\forall m\in\mathcal{M}$, $\exists x^*_m$ such that $\lim_{l\rightarrow\infty}q_{m,l}(\bar{\alpha})=x^*_m$. Next, we need to show that $x^*_m=0$, $\forall m\in\mathcal{M}$.
By \eqref{eq:sum-qml}, we have 
\begin{equation}\label{eq:limits-x*}
    \sum_{m\in\mathcal{M}}v_m u_m x^*_m=\lambda\xi\sum_{k\in\mathcal{K}}w_k\big(\sum_{m\in\mathcal{M}}\frac{p_{k,m}v_m}{\delta_k}x^*_m\big)^d.
\end{equation}
Clearly, $x^*_m=0$, $\forall m\in\mathcal{M}$ is a solution of \eqref{eq:limits-x*}. It must be the unique solution, since by Lemma~\ref{lem:alpha-m}, for all $(x^*_m,m\in\mathcal{M})\in[0,1]^M$ with $\sum_{m\in\mathcal{M}}x^*_m>0$, 
$$(\sum_{m\in\mathcal{M}}v_m u_m x^*_m)^{-1}\lambda\xi\sum_{k\in\mathcal{K}}w_k\big(\sum_{m\in\mathcal{M}}\frac{p_{k,m}v_m}{\delta_k}x^*_m\big)^d<1$$ implything that \eqref{eq:limits-x*} does not hold. Now, let $q^*_{m,l}=q_{m,l}(\bar{\alpha})$, $\forall m\in\mathcal{M}, l\in\N_0$. 
\vspace{3mm}

\noindent
Now, we are going to show the \textit{uniqueness}. The proof of the uniqueness is based on a monotonicity property of the system, which is stated in the following claim.
\begin{claim}\label{claim:mono}
If $\mathbf{q}\leq \hat{\mathbf{q}}$ for $\mathbf{q}, \hat{\mathbf{q}}\in\mathcal{S}$, then $\bar{\mathbf{q}}(t,\mathbf{q})\leq \bar{\mathbf{q}}(t,\hat{\mathbf{q}})$ for all $t$.
\end{claim}
\begin{claimproof}
Consider any $\mathbf{q}\leq \hat{\mathbf{q}}\in\mathcal{S}$. It is easy to construct two copies of the $N$-th systems with initial states $\{X^N_j(0),j\in V^N\}$ and $\{\hat{X}^N_j(0),j\in V^N\}$ satisfying:
\begin{enumerate}
    \item For all $j\in V^N$, $X^N_j(0)\leq \hat{X}^N_j(0)$;
    \item $\{X^N_j(0),j\in V^N\}$ has the corresponding global occupancy $\mathbf{q}^N(0)=\mathbf{q}\in\mathcal{S}$; similarly, $\{\hat{X}^N_j(0),j\in V^N\}$ has $\hat{\mathbf{q}}^N(0)=\hat{\mathbf{q}}\in\mathcal{S}$.
\end{enumerate}
By a natural coupling, we have that for all $j\in V^N$ and $t\geq 0$, $X^N_j(t)\leq \hat{X}^N_j(t)$ implying that $\mathbf{q}^N(t) \leq \hat{\mathbf{q}}^N(t)$. Since systems are stable, then $\mathbf{q}^N(t),\hat{\mathbf{q}}^N(t)\in\mathcal{S}$ for all $t\geq 0$. Moreover, by Theorem~\ref{thm:conv-global-occup}, the claim follows.
\end{claimproof}
\vspace{1mm}
\noindent
We continue the proof of the uniqueness. Now, it is sufficient to show that $\lim_{t\rightarrow\infty}\bar{\mathbf{q}}(t,\mathbf{q}_0)=\mathbf{q}^*$ which either $\mathbf{q}_0\leq \mathbf{q}^*$ or $\mathbf{q}_0\geq \mathbf{q}^*$ in component-wise, since Claim~\ref{claim:mono} implies that 
$$\bar{\mathbf{q}}(t,\min(\mathbf{q}_0,\mathbf{q}^*))\leq \bar{\mathbf{q}}(t,\mathbf{q}_0)\leq \bar{\mathbf{q}}(t,\max(\mathbf{q}_0,\mathbf{q}^*)),\quad \forall \mathbf{q}_0\in\bar{\mathbf{S}}, t\geq 0.$$
We will prove the case that if $\mathbf{q}_0\leq \mathbf{q}^*$, then 
$$\lim_{t\rightarrow\infty}\bar{\mathbf{q}}(t,\mathbf{q}_0)=\mathbf{q}^*.$$
Since the case that $\mathbf{q}_0\geq \mathbf{q}^*$ is similar.
Also, Note that $q_{m,l}(\infty)$, $\forall m\in\mathcal{M}, l\geq 2$ can be solved recursively by \eqref{eq:fixed-point} when $q_{m,1}(\infty)$, $\forall m\in\mathcal{M}$ are determined so it is sufficient to show that $q_{m,1}(\infty)=q^*_{m,1}$, $\forall m\in\mathcal{M}$. By ODE~\eqref{eq:ode-sys-2}, we have 
$$\frac{d\sum_{m\in\mathcal{M}}v_m q_{m,1}(t)}{dt}=-\sum_{m\in\mathcal{M}}v_m u_m q_{m,1}(t)+\lambda\xi.$$
Since $\mathbf{q}_0\leq \mathbf{q}^*$, then $\bar{\mathbf{q}}(t,\mathbf{q}_0)\leq \bar{\mathbf{q}}(t,\mathbf{q}^*)=\mathbf{q}^*$. Observe that $\sum_{m\in\mathcal{M}}v_m u_m q^*_{m,1}=\lambda\xi$. Hence, if for some $m\in\mathcal{M}$, $q_{m,1}(t)<q^*_{m,1}$, then $\frac{d\sum_{m\in\mathcal{M}}v_m q_{m,1}(t)}{dt}>0$, which implies that $$\lim_{t\rightarrow\infty}\sum_{m\in\mathcal{M}}v_m q_{m,1}(t)=\sum_{m\in\mathcal{M}}v_m q_{m,1}(\infty)=\lambda\xi.$$
Since for all $m\in\mathcal{M}$ and $t\geq 0$, $q_{m,1}(t)\leq q^*_{m,1}$, then $\lim_{t\rightarrow\infty}q_{m,1}(t)=q^*_{m,1}$ must hold for all $m\in\mathcal{M}$.

\end{proof}

\begin{proof}[Proof of Theorem~\ref{thm:double-expo}]
The result holds immediately from Proposition~\ref{prop:dbly-decay} and Theorem~\ref{thm: exists-unique}.
\end{proof}

\subsection{Proof of Tightness and Interchange of Limits}
Next, we are going to prove the tightness of the steady state occupancy processes $\{\mathbf{q}^N(\infty)\}_N$. Let $\bar{q}^N_l(\infty)=\sum_{m\in\mathcal{M}}q^N_{m,l}(\infty)$ and $\bar{\mathbf{q}}^N(\infty)=(\bar{q}^N_l(\infty),l\in\N_0)$. In order to show the tightness of $\{\mathbf{q}^N(\infty)\}_N$, it is sufficient to show that the sequence is $\{\bar{\mathbf{q}}^N(\infty)\}_N$, which is stated in the next proposition. 
For showing the above tightness, we will bound the tail of the expected global occupancy of the stationary state first.
\begin{lemma}\label{lem:tail-bound-expectation-steady}
Let $\{G^N\}_N$ be a sequence of proportionally sparse graphs satisfying Condition~\ref{cond-1}. There exists an $N_0$ such that for all $N\geq N_0$ and $\ell\geq 1$, 
\begin{equation}
    \sum_{l=\ell}^{\infty}\E(\bar{q}^N_{l}(\infty))\leq \frac{(1+\rho)/2}{1-(1+\rho)/2}\E(\bar{q}^N_{\ell-1}(\infty)).
\end{equation}
Furthermore, 
\begin{equation}
    \E(\bar{q}^N_{\ell}(\infty))\leq \Big(\frac{1+\rho}{2}\Big)^{\ell}, \quad \forall \ell\in\N_0.
\end{equation}
\end{lemma}
\noindent
The proof of Lemma~\ref{lem:tail-bound-expectation-steady} is similar to \cite[Lemma~3]{RM22}. We define a sequence $\{L^N_{m,\ell}\}_{m\in\mathcal{M},\ell\in\N_0}$ of Lyapunov functions and bound the drift of $L^N_{m,\ell}$ which enables us to bound the tail sum of $\bar{q}^N_l(\infty)$ starting from $\ell$. 
Given the $N$-th system state $X^N=(X^N_j,j\in V^N)$. Let $Q^N_{m,l}(X)$ be the set of servers of type $m\in\mathcal{M}$ with queue length at least $l\in\N_0$. For each $m\in\mathcal{M}$, we define a sequence of Lyapunov functions $L^N_{m,\ell}(X)=\sum_{i=\ell}^{\infty}\sum_{l=i}|Q^N_{m,l}(X)|$, $\ell\in\N_0$. The complete proof is provided in Appendix~\ref{app:proof-tail-bound-expectation}.

The next lemma from \cite{MBL17} gives us the criterion for $\ell_1$-tightness. 

\begin{lemma}[{\cite[Lemma~2]{MBLW16-3}}]\label{lem:criterion-tight}
Let $\{\mathbf{X}^N\}$ be a sequence of random variables in $\zeta$. Then, the following are equivalent:
\begin{enumerate}[{\normalfont (i)}]
    \item $\{\mathbf{X}^N\}$ is tight with respect to the product topology, and for all $\varepsilon>0$,
    \begin{equation}
        \lim_{k\rightarrow\infty}\varlimsup_{N\rightarrow\infty}\PP\big(\sum_{i\geq k}x^N_i> \varepsilon\big)=0.
    \end{equation}
    \item $\{\mathbf{X}^N\}$ is tight with respect to the $\ell_1$-topology.
\end{enumerate}
\end{lemma}

\begin{proof}[Proof of Theorem~\ref{thm:tightness}]
Since, for all $l\in\N_0$, $\bar{q}^N_l\in[0,1]$, then it is easy to check that $\{\bar{\mathbf{q}}^N(\infty)\}$ is tight with respect to the product topology. Hence, it is sufficient to show that for any $\varepsilon>0$,
\begin{equation}\label{eq:l-N-q-bar}
    \lim_{\ell\rightarrow\infty}\varlimsup_{N\rightarrow\infty} \PP\Big(\sum_{l\geq \ell}\bar{q}^N_l(\infty)>\varepsilon\Big)=0. 
\end{equation}
By Markov's inequality and Lemma~\ref{lem:tail-bound-expectation-steady}, we have that for all $N\geq N_0$, 
\begin{equation}
    \begin{split}
        &\PP\Big(\sum_{l\geq \ell}\bar{q}^N_l(\infty)>\varepsilon\Big)\leq \frac{1}{\varepsilon}\E\Big(\sum_{l\geq \ell}\bar{q}^N_l(\infty)\Big)
         \leq  \frac{1}{\varepsilon}\frac{(1+\rho)/2}{1-(1+\rho)/2}\E\Big(\bar{q}^N_{\ell-1}(\infty)\Big)\leq \frac{1}{\varepsilon}\frac{\big((1+\rho)/2\big)^{\ell}}{1-(1+\rho)/2},
    \end{split}
\end{equation}
which implies that \eqref{eq:l-N-q-bar} holds. By Lemma~\ref{lem:criterion-tight}, the desired result holds.
\end{proof}

\begin{proof}[Proof of Theorem~\ref{thm:interchange of limits}]
By Theorem~\ref{thm:tightness}, $\{\mathbf{q}^N(\infty)\}_N$ is tight with respect to the $\ell_1$-topology. Then, any subsequence has a convergent further subsequence. Let $\{\mathbf{q}^{N_n}(\infty)\}_n$ be such convergent subsequence and assume $\mathbf{q}^N(\infty)\dto \mathbf{q}^*$. 
Clearly, $\mathbf{q}^*$ must be in the space $\mathcal{S}$. 
Now, initiate the $N$-th system at its stationarity. Then the system is in steady state at any fixed finite time $t\geq 0$.
That is, we have $\mathbf{q}^{N_n}(t)\sim \mathbf{q}^{N_n}(\infty)$ for all $t\in [0,T]$.
Also by Theorem~\ref{thm:general-conv-global-occup}, 
\begin{equation}
    \mathbf{q}^{N_n}(t)\dto \mathbf{q}(t).
\end{equation}
Thus, for all $t\in[0,T]$,
\begin{equation}
    \mathbf{q}(t)\sim \mathbf{q}^*,
\end{equation}
which implies that $\mathbf{q}^*$ is a stationary point of the limiting system. By Theorem~\ref{thm: exists-unique}, we know that $\mathbf{q}^*$ is unique. Therefore, the desired result holds.
\end{proof}

\section{Numerical Results}\label{sec:numerical-results}
In this section, we will present the simulation to validate the theoretical results. 
Using the insights from the theoretical results, we will also show that systems with carefully designed compatibility structure perform much better than the classical, fully flexible systems.
Throughout this section, we set the system parameters as follows:
\begin{itemize}
    \item $K=2$: two types of dispatchers;
    \item $M=3$: three types of servers;
    \item $d=2$: the system follows the JSQ($2$) policy;
    \item $\boldsymbol\mu=(1,5,10)$, where each $\mu_m$, $m=1,2,3$, is the service rate of type $m$ servers;
    \item $\lambda=3$, which is the arrival rate at each dispatcher is $\lambda$;
    \item $Q=\begin{bmatrix} 0.2 & 0.5 & 0.3\\ 0.5 & 0 & 0.5 \\0.9 & 0.1 & 0  \end{bmatrix}$, where each $q_{m,l}$ is the probability that type $m$, $m=1,2,3$ server's initial queue length is $l$, $l=1,2,3$;
    \item Fraction of types of dispatchers: $\begin{bmatrix}w_1 & w_2\end{bmatrix}=\begin{bmatrix}0.2 & 0.8 \end{bmatrix}$;
    \item Fraction of types of servers: $\begin{bmatrix}v_1 & v_2 & v_3\end{bmatrix}=\begin{bmatrix}0.5 & 0.3 & 0.2 \end{bmatrix}$;
    \item $\xi=1$: the relationship between the number of dispatchers and that of servers in the system.
\end{itemize}
By the above setting, the capacity sufficiency is satisfied, $\lambda\xi=3<\sum_{m\in\mathcal{M}}v_mu_m=4$. 
The first experiment is to compare the performance of the classical, fully flexible system with that of the system with carefully designed compatibility structure.

\paragraph{Complete Bipartite vs.~Designed Compatibility Structure.} The complete bipartite is the case that the compatibility matrix $\mathbf{p}^0= (p^0_{m,k},m\in\mathcal{M},k\in\mathcal{K})$ is a matrix with all elements equal to 1. 
From Lemma~\ref{lem:stability}, we have that an $N$-th system under JSQ(d) is stable if and only if it should satisfies the following: 
\begin{equation*}
    \rho^N=\max_{\substack{U\subseteq V^N\\ U\neq \emptyset}}\Big\{\Big(\sum_{j\in U}\sum_{m\in\mathcal{M}}\mathds{1}_{(j\in V^N_m)}u_m\Big)^{-1}\sum_{i\in W^N}\sum_{\substack{S\subseteq(U\cap \mathcal{N}^N_w(i)):\\|S|=d}}\frac{\lambda}{{\delta^N_i\choose d}}\Big\}<1.
\end{equation*}
By Lemma~\ref{lem:alpha-m}, for the complete bipartite case, we have that 
\begin{equation*}
    \lim_{N\rightarrow\infty}\rho^N \geq \max_{\mathcal{M}'\subseteq\mathcal{M}}\big(\sum_{m\in\mathcal{M}'} v_m u_m\big)^{-1}\lambda\xi\sum_{k\in\mathcal{K}}w_k\Big(\sum_{m\in\mathcal{M}'} v_m\Big)^d\geq (0.5\times 1)^{-1}\times 3\times (0.5)^2>1,  
\end{equation*}
which implies that for large enough $N$, the system under JSQ($2$) is unstable. The bottleneck here is that the type 1 servers with poor performance receive heavy workload. 
By Proposition~\ref{prop:p-matrix-existence}, if the capacity sufficiency is satisfied, then there always exists a compatibility matrix $\mathbf{p}^1\in [0,1]^{K\times M}$ making all large enough systems stable under JSQ($2$). Checking the feasible region defined in Lemma~\ref{lem:relationship-v-u-pkm}, we get one of appropriate matrices $\mathbf{p}^1$ defined as
\begin{equation*}
    \mathbf{p}^1=\begin{bmatrix} 0.05 & 0.6 & 1\\ 0.1 & 0.7 & 1 \end{bmatrix}.
\end{equation*}
The intuition for designing the compatibility matrix, like $\mathbf{p}^1$, is to lower the traffic intensity for type 1 servers by decreasing the fraction of the type 1 servers in the neighborhood of each dispatcher.
For the experiment, we set the number of servers $N=1000$, and consider two systems $S1$ and $S2$. $S1$ is a system with complete bipartite graph structure; $S2$, generated by \irg($\mathbf{p}^1$) (Definition~\ref{def:irg}), is a system with compatibility matrix $\mathbf{p}^1$.
We simulate the evolution of each system for 100 times and plot the mean sample path in Figure~\ref{fig:Complete-Designed}.

 \begin{figure}
    \centering
    \includegraphics[width=0.5\linewidth]{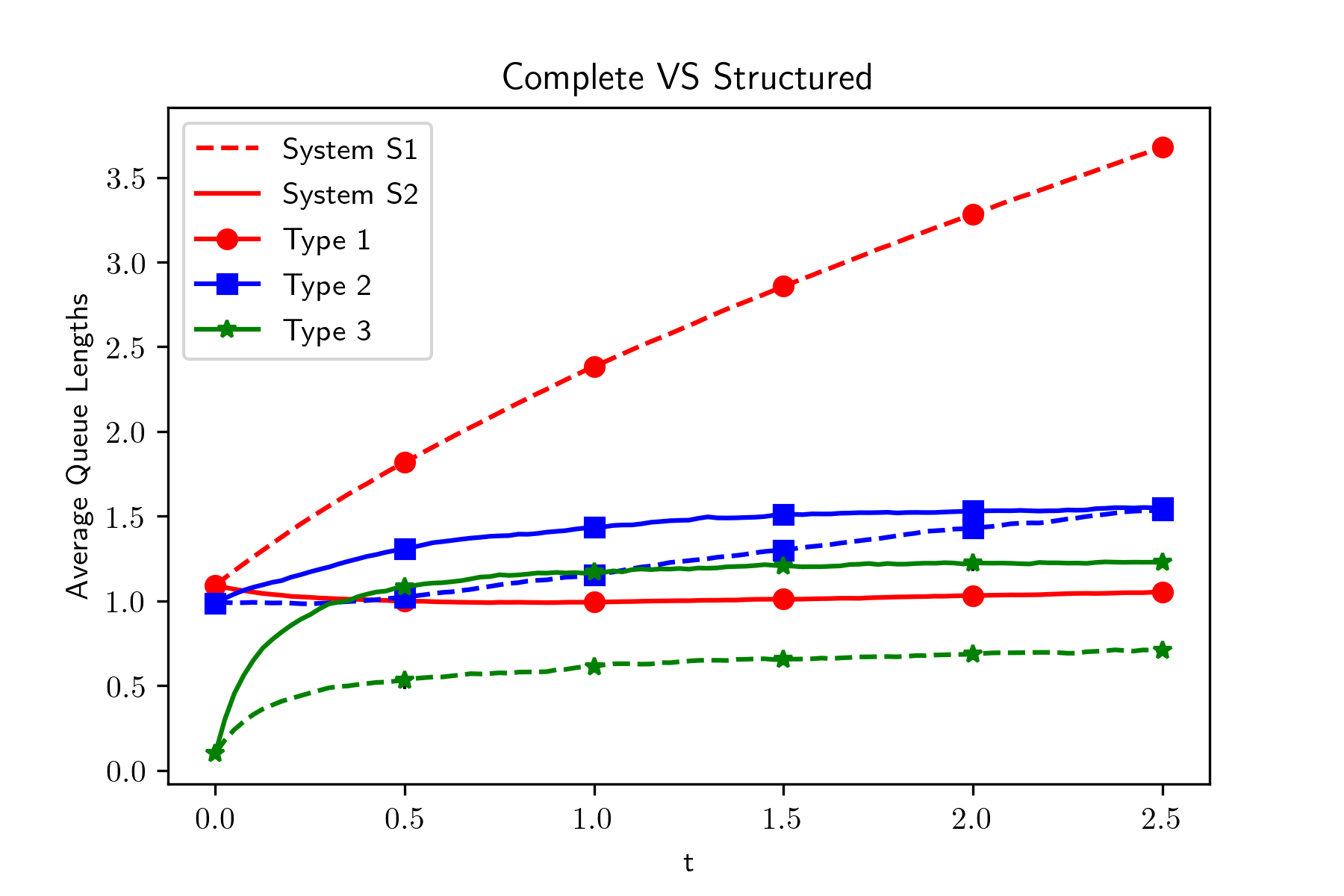}
    \caption{Complete Bipartite vs.\ Appropriate Designed Structure}
    \label{fig:Complete-Designed}
 \end{figure}
Figure~\ref{fig:Complete-Designed} shows that the average queue length of type 1 servers in $S1$ almost monotonically increases as $t$ increases, which implies that the average queue length of type 1 servers in $S1$ is unbounded. 
However, in the system $S2$, the average queue length of each type of servers is bounded. From this numerical result, we observer that with appropriately designed graph structure, the performance of the system can be improved. Although we tried to plot the $95\%$ confidence interval (CI) for each point $t=0.5,1.0,1.5,2.0,2.5$, the CI is narrow and its size is smaller than that of markers in the plot.

\paragraph{Convergence of global occupancy states.} In this experiment, we generate systems by \irg($\mathbf{p}^1$) and simulate the evolution of systems with size $N=100,500,1000$. For each system, we also simulate for 100 times and plot the mean trajectories of  $q^N_{m,1}$ and $q^N_{m,2}$, $m\in\{1,2,3\}$ in Figure~\ref{fig:convergence}. Also, we plot the evolution of $q_{m,1}$ and $q_{m,2}$, $m=1,2,3$,  of the limit system. The simulation results show that the evolution of the global occupancy of the $N$-th system converges to that of the limit system as $N$ goes to infinity. From the simulation result, we find that $q^N_{1,1}$, and especially $q^N_{1,2}$, decrease very fast when their initial value are large. In other words, when the average queue length of type 1 servers is large, it will decrease very fast. The reason is due to our designed compatibility matrix such that compared with other type servers, type 1 servers are sampled much less often.

\begin{figure}
\minipage{0.32\textwidth}
  \includegraphics[width=\linewidth]{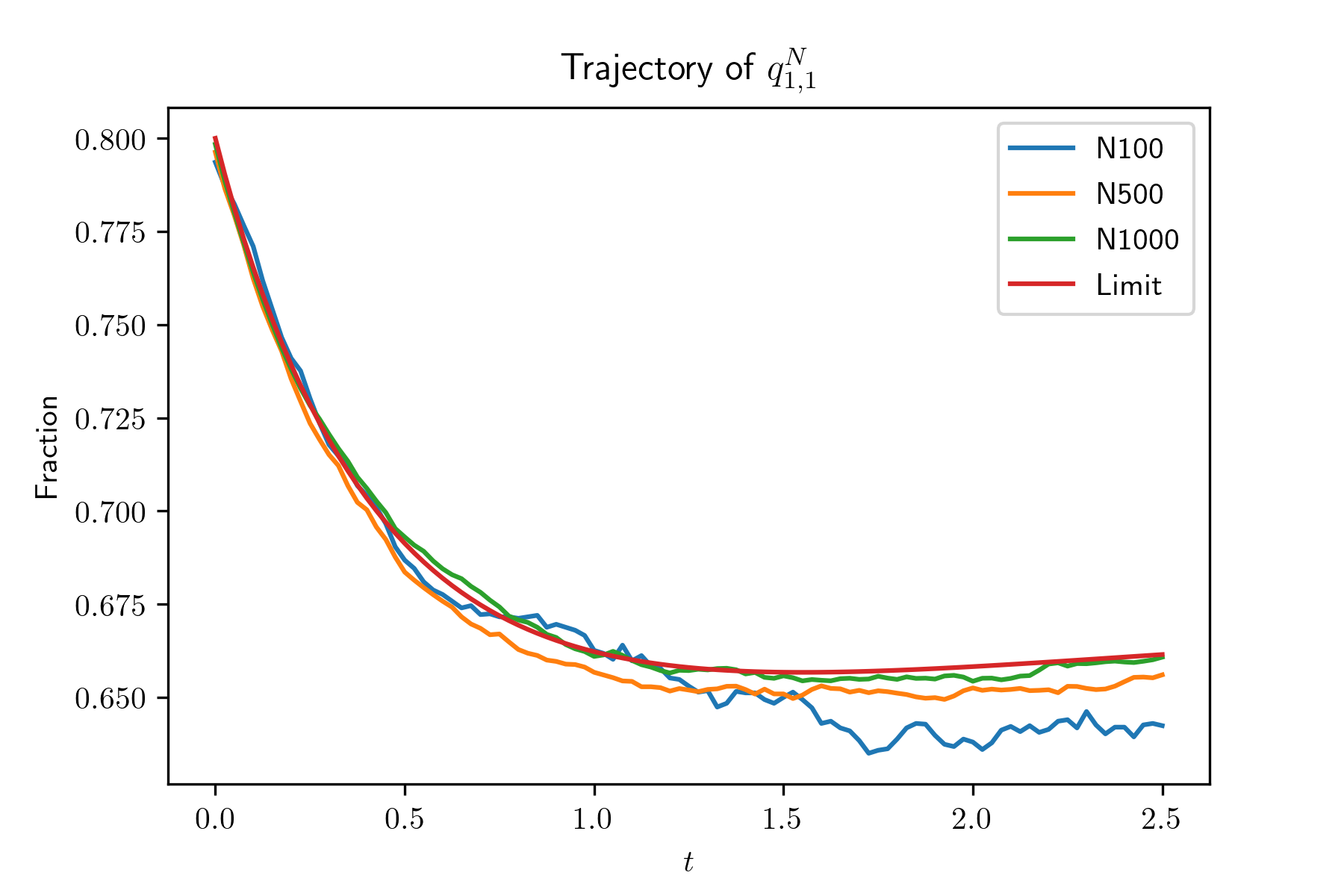}
\endminipage\hfill
\minipage{0.32\textwidth}
  \includegraphics[width=\linewidth]{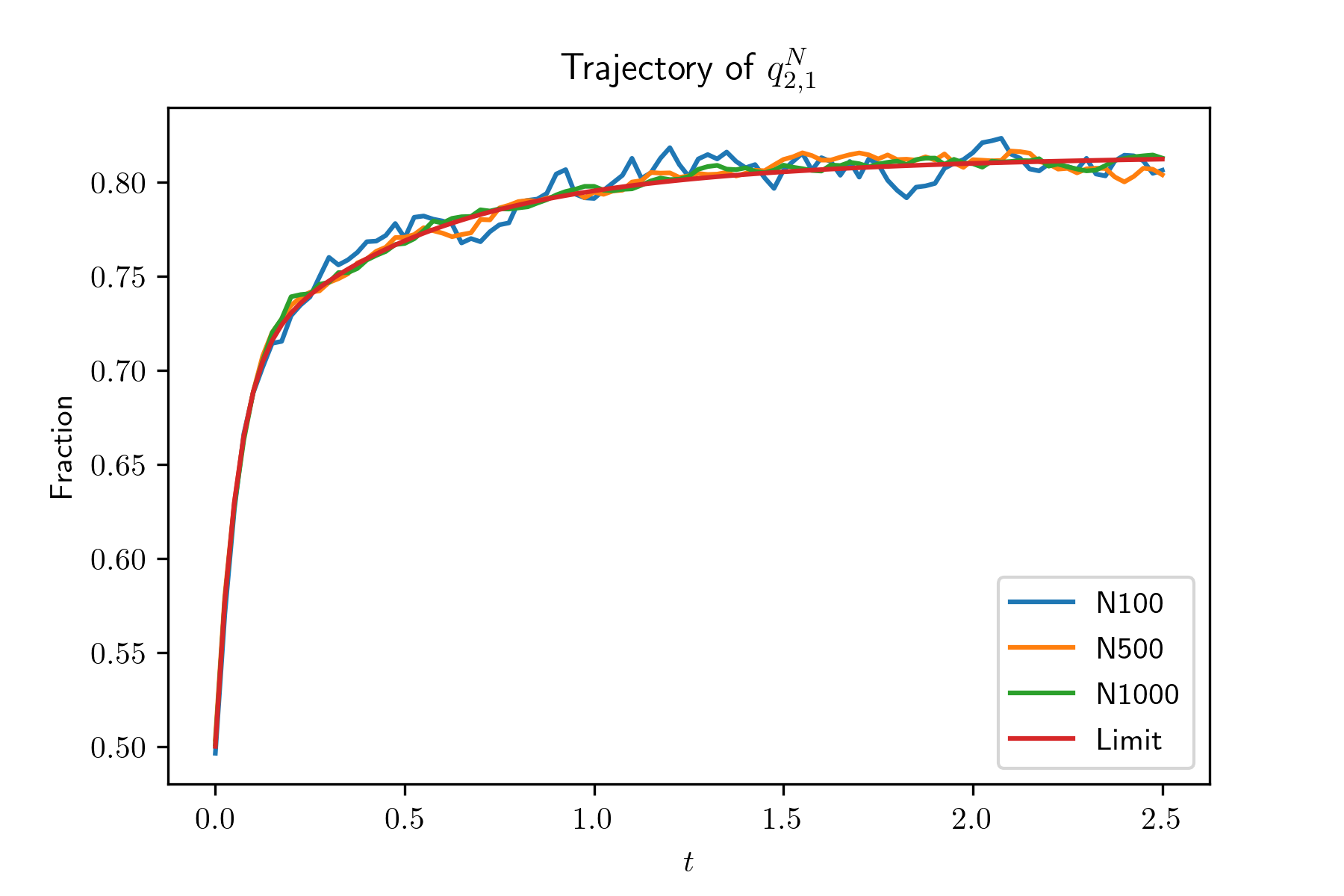}
\endminipage\hfill
\minipage{0.32\textwidth}
  \includegraphics[width=\linewidth]{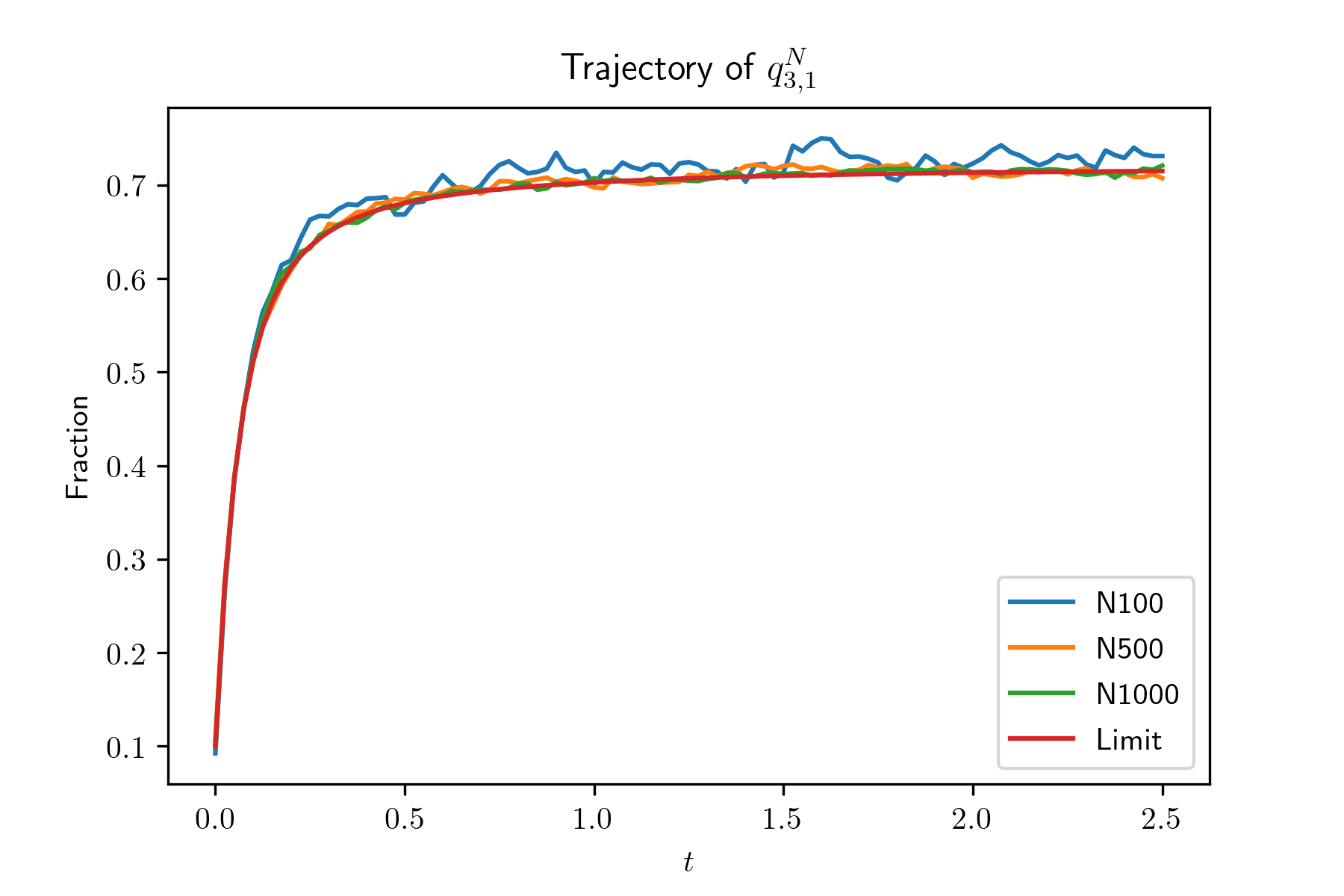}
\endminipage\par
\vskip\floatsep
\minipage{0.32\textwidth}
  \includegraphics[width=\linewidth]{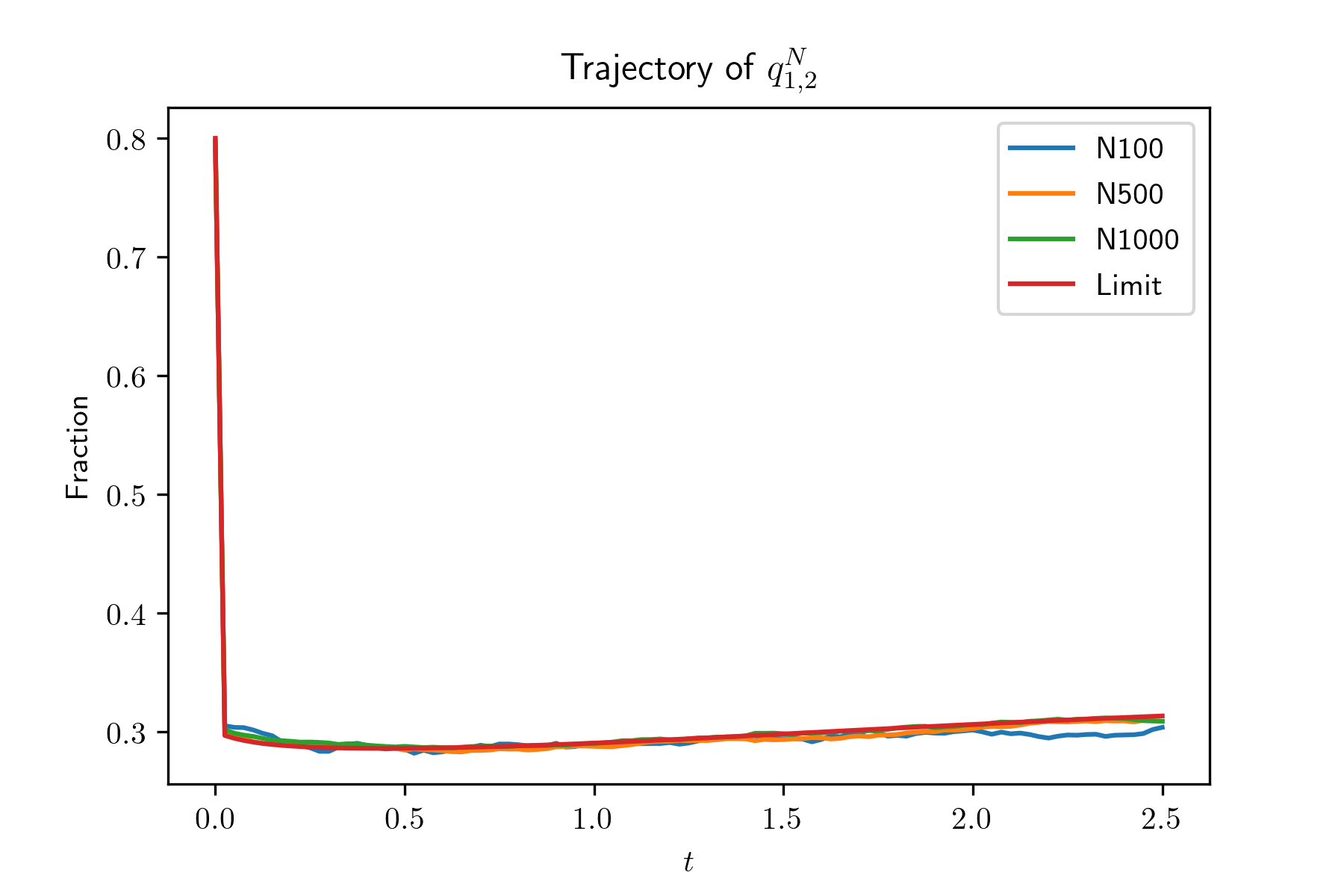}
\endminipage\hfill
\minipage{0.32\textwidth}
  \includegraphics[width=\linewidth]{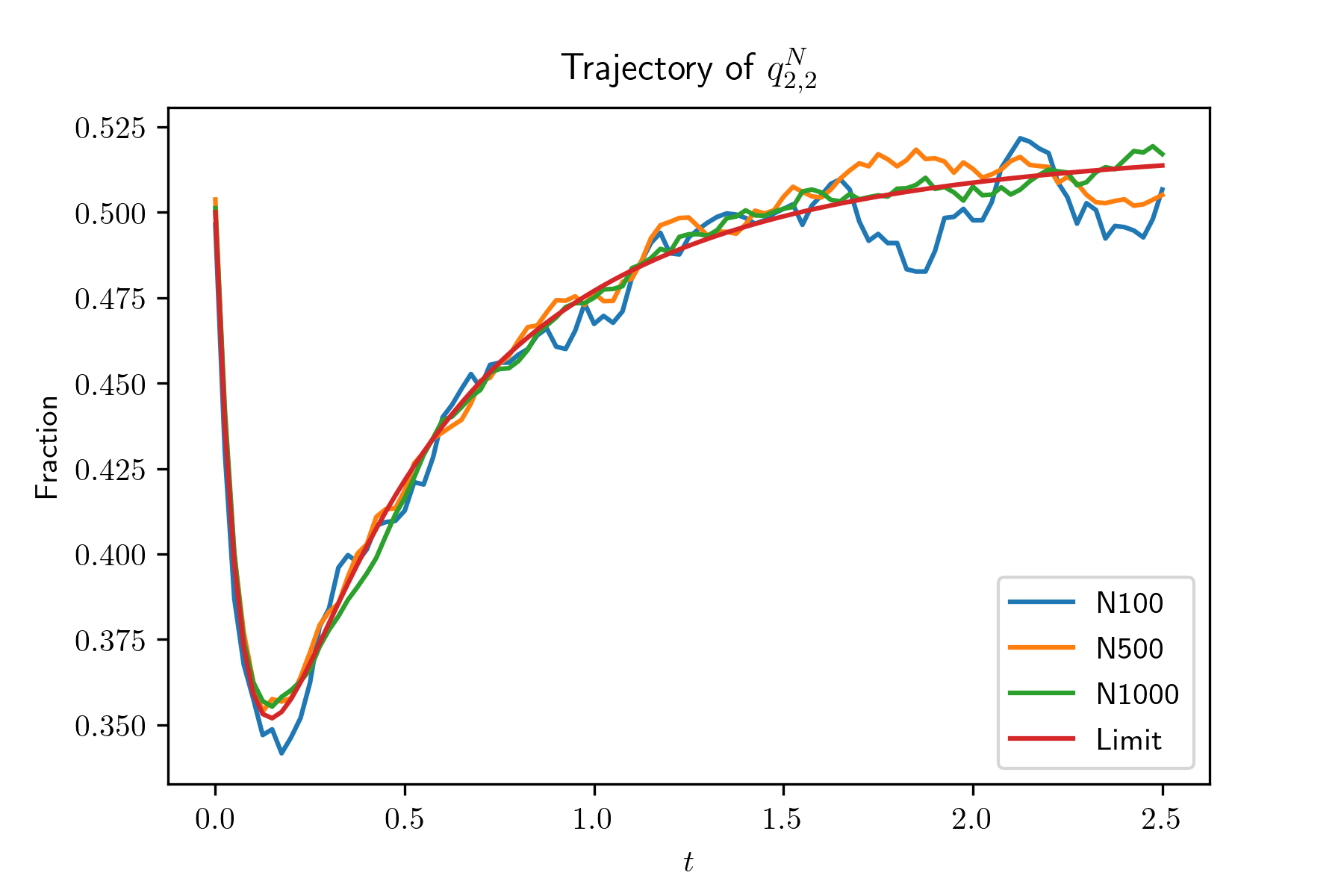}
\endminipage\hfill
\minipage{0.32\textwidth}
  \includegraphics[width=\linewidth]{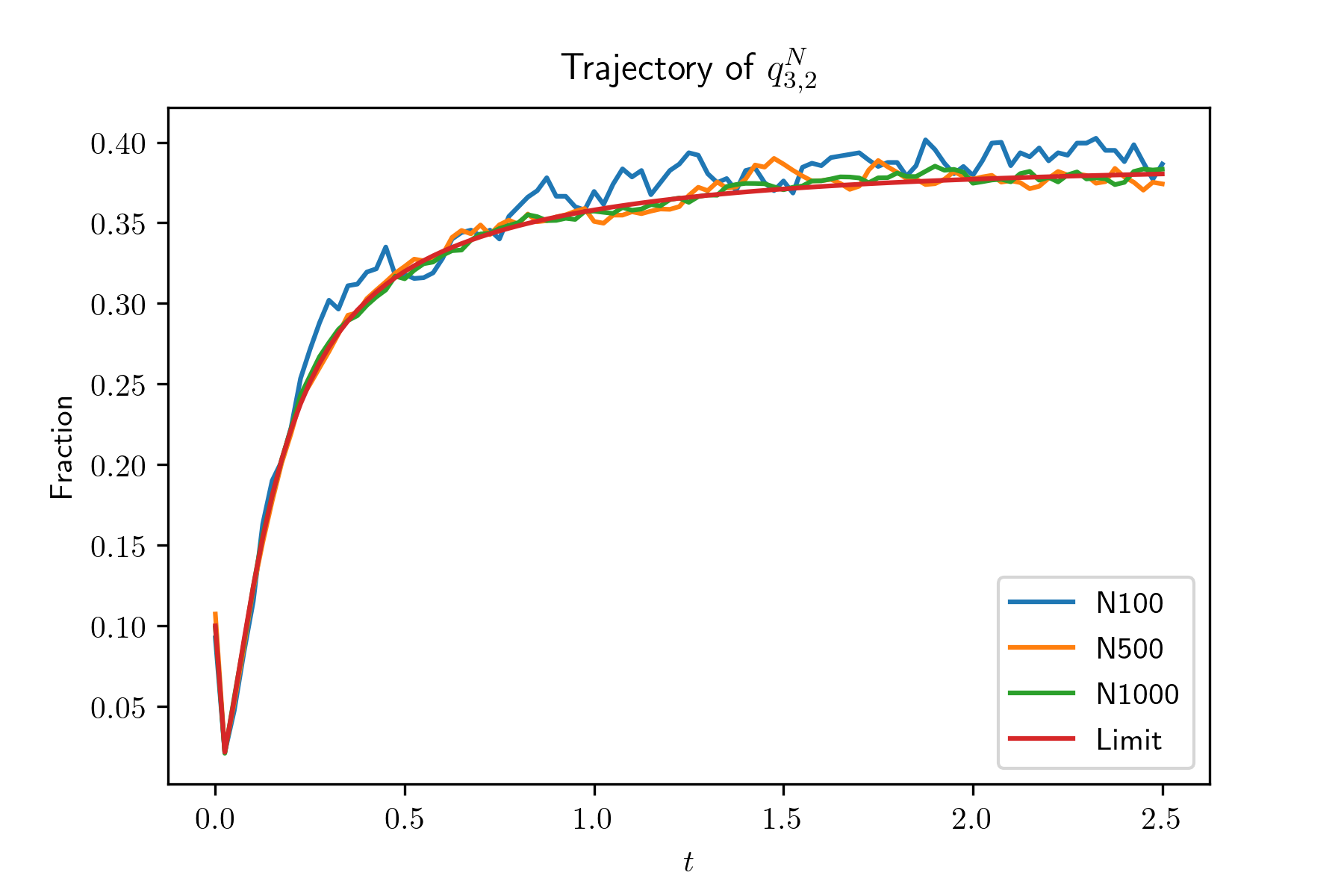}
\endminipage
\caption{The simulated trajectories of $q^N_{m,1}$ and $q^N_{m,2}$, $m=1,2,3$ converging to the solution of the system of ODEs as $N$ increases.} \label{fig:convergence}
\end{figure}

\paragraph{Uniqueness of the fixed point of the limit system.} From Theorem~\ref{thm: exists-unique}, we have that for all $\mathbf{q}\in\mathcal{S}$, $\lim_{\rightarrow\infty}\bar{\mathbf{q}}(t,\mathbf{q}_0)=\mathbf{q}^*$. In order to verify this, we simulation of the evolution of $\bar{\mathbf{q}}(t,\mathbf{q}_0)$ with different $\mathbf{q}_0\in\mathcal{S}$ (i.e., consider different $Q$ mentioned above).
We also simulate the system with $Q_1=\begin{bmatrix} 0.4 & 0.3 & 0.3\\ 0.1 & 0.8 & 0.1 \\0.3 & 0.6 & 0.1  \end{bmatrix}$ and $Q_2=\begin{bmatrix} 0.6 & 0.3 & 0.1\\ 0.8 & 0.1 & 0.1 \\0.7 & 0.2 & 0.1  \end{bmatrix}$. Figure~\ref{fig:unique} shows that with different $\mathbf{q}\in\mathcal{S}$, $\lim_{t\rightarrow\infty}q_{m,1}(t)$, $m=1,2,3$, are the same. If $q_{m,1}$, $m=1,2,3$ are fixed, then the values of all $q_{m,l}$, $l\geq 2$, $m=1,2,3$ are fixed as well by using \eqref{eq:fixed-point}.  Hence, Figure~\ref{fig:unique} verifies the uniqueness of the fixed point.
\begin{figure}
    \minipage{0.32\textwidth}
  \includegraphics[width=\linewidth]{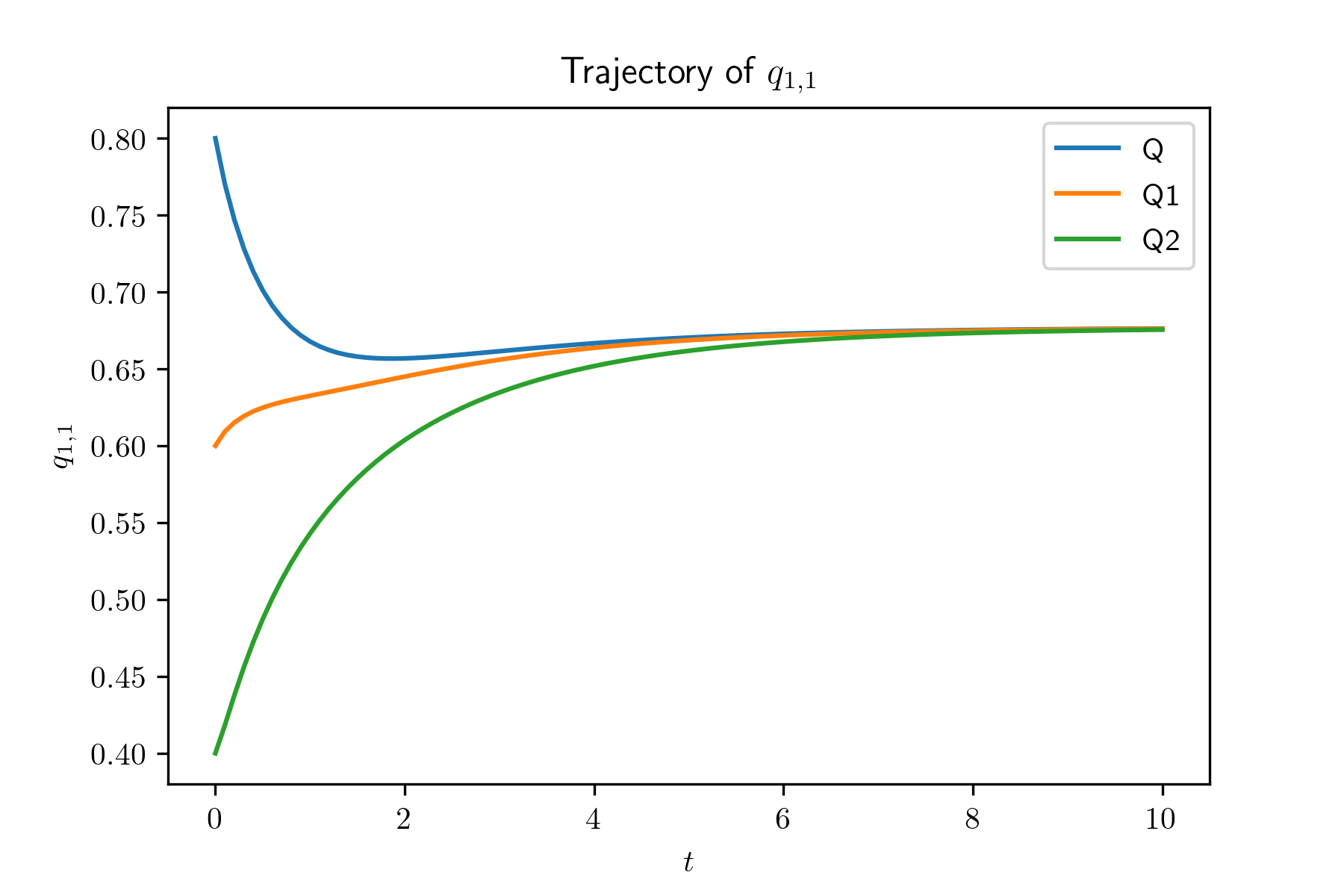}
\endminipage\hfill
\minipage{0.32\textwidth}
  \includegraphics[width=\linewidth]{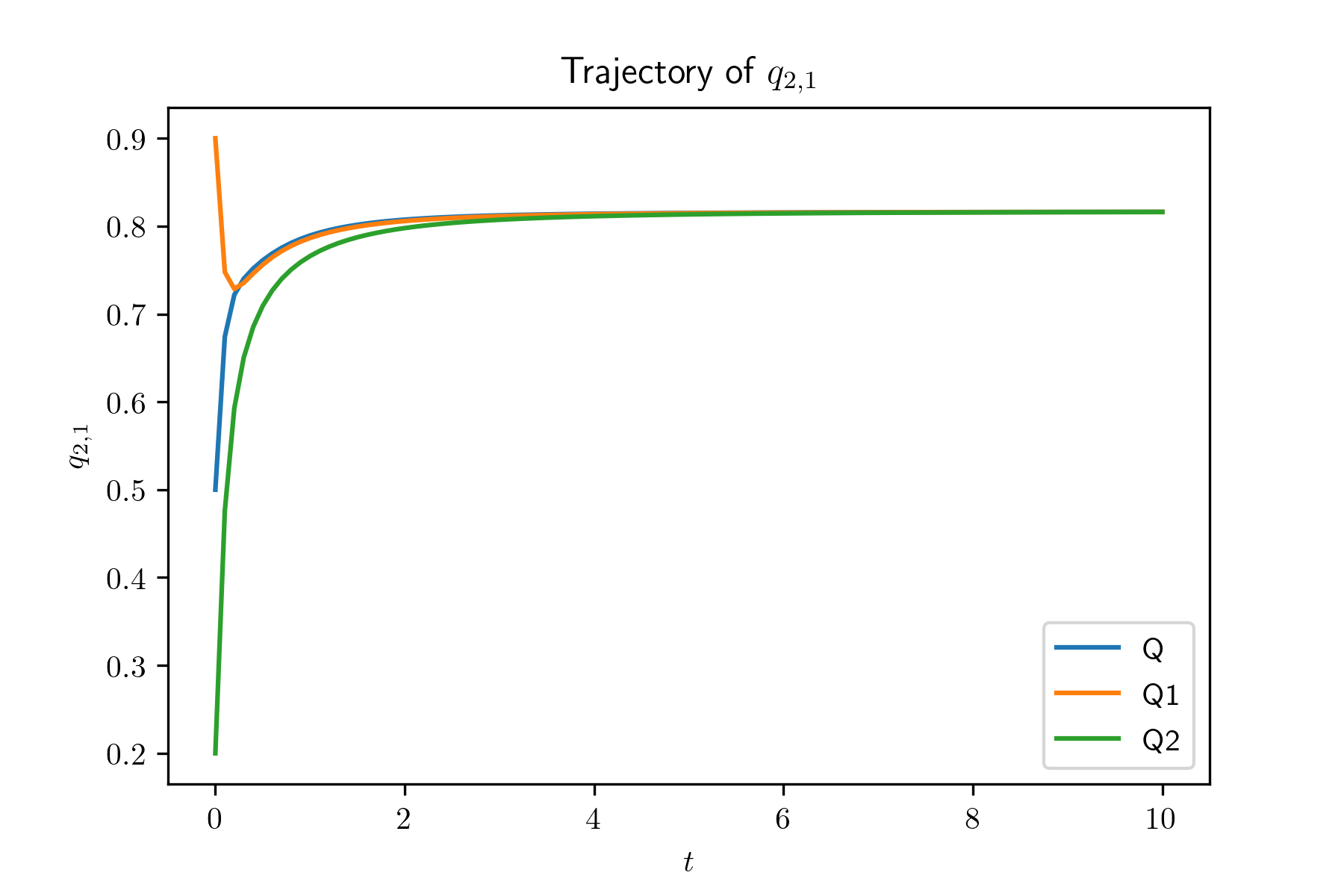}
\endminipage\hfill
\minipage{0.32\textwidth}
  \includegraphics[width=\linewidth]{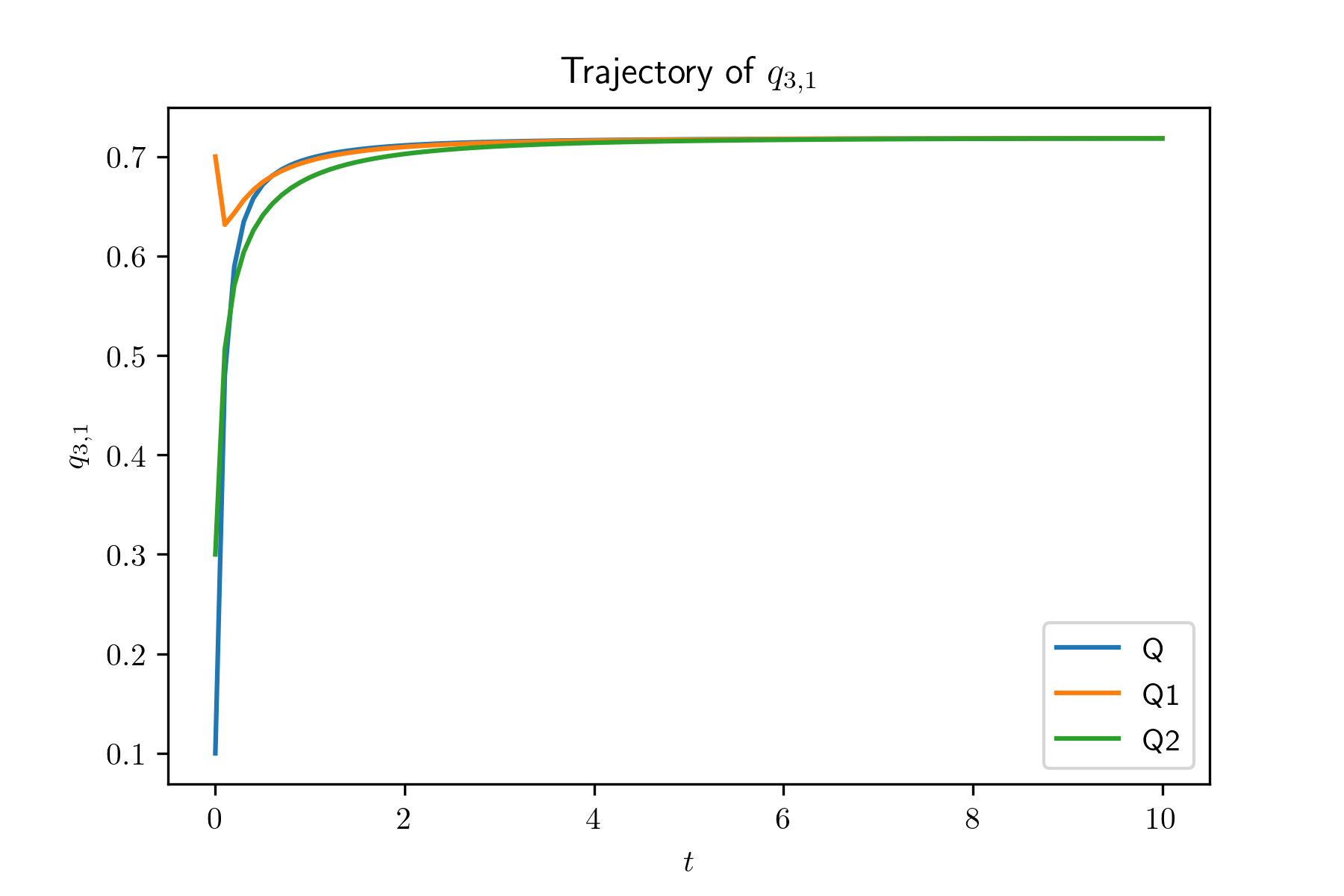}
\endminipage
    \caption{Multiple trajectories of $q_{m,1}$, $m=1,2,3$ in the limit system converging to the fixed point.}
    \label{fig:unique}
\end{figure}

\bibliographystyle{plain}
\bibliography{references,references-debankur,references-Ruoyu}

\appendix

\section{Proofs for Stability Results}\label{app:proof-lem-stability}

The goal of this appendix is to prove Proposition~\ref{prop:p-matrix-existence}. 
We start by proving Lemma~\ref{lem:relationship-v-u-pkm}, for which we need the next technical lemma. 
This lemma will help us to upper bound the probability that a new task will be assigned to a specific subset of servers (in particular, \eqref{eq:subset-rho-2}).

\begin{lemma}\label{lem:opt-prob}
Consider the following optimization problem:
\begin{align*}
    \max \sum_{i=1}^N{x_i\choose d}
   \text{ s.t. } & \sum_{i=1}^Nx_i= C\text{ and }
    x_i\in[0,D]
\end{align*}
where $C$ and $D$ are positive integers. Let $k^*=\lfloor C/D \rfloor$. Then, the optimal value is  $k^*{D\choose d}+{C-Dk^*\choose d}$, if $N> k^*$; otherwise, the optimal value is $N{D\choose d}$.
\end{lemma}
\begin{proof}
We will prove by contradictions.
Suppose the maximizer $\{x_i^* : i=1,\dotsc,N\}$ contains some $x_j^*,x_k^* \in \{1,\dotsc,D-1\}$ for some $j \ne k$.
Note that
$$\binom{x_j^*}{d} + \binom{x_k^*}{d} < \binom{\tilde{x}_j}{d} + \binom{\tilde{x}_k}{d},$$
where $\tilde{x}_j=\min\{x_j^*+x_k^*,D\}$ and $\tilde{x}_k=x_j^*+x_k^*-\tilde{x}_j$, that is, the pair $(x_j^*,x_k^*)$ gives smaller value than the extremer pair $(\tilde{x}_j, \tilde{x}_k)$.
This contradicts the assumption that $\{x_i^* : i=1,\dotsc,N\}$ is the maximizer.
Therefore the maximizer $\{x_i^* : i=1,\dotsc,N\}$ must contain at most one $x_j^* \in \{1,\dotsc,D-1\}$ with all the other $x_i^*$ being either $0$ or $D$.
This completes the proof.
\end{proof}

\begin{proof}[Proof of Lemma~\ref{lem:relationship-v-u-pkm}]
Suppose that \eqref{eq:max-rho-set} holds. Since $\mathcal{M}$ is finite, then there exists a $\rho\in(0,1)$ such that $\frac{\lambda\xi}{u_m}\sum_{k\in\mathcal{K}}\frac{w_k p_{k,m}}{\delta_k} <\rho$ for all $m\in\mathcal{M}$. Fix any $\varepsilon\in(0,\frac{1-\rho}{1+3\rho})$. Recall $\delta^N_i=|\mathcal{N}^N_w(i)|$. By our model assumption and Condition~\ref{cond-1}, there exists $N_{\varepsilon}\in \N_0$ such that for all $m\in\mathcal{M}$ and $j\in V^N_m$, 
\begin{equation}\label{eq:uniform-deg-server}
     p_{k,m}w_kW(N) (1-\varepsilon) \leq \deg^N_v(k,j)\leq   p_{k,m}w_kW(N)  (1+\varepsilon) ,\quad \forall k\in\mathcal{K},
\end{equation}
and for all $K\in\mathcal{K}$ and $i\in W^N_k$,
\begin{equation}\label{eq:uniform-deg-dispatcher}
    N \delta_k  (1-\varepsilon) \leq \delta^N_i\leq  N \delta_k(1+\varepsilon).
\end{equation}
Consider the $N$-th system. Consider any nonempty subset $U\subseteq V$ of servers. If $|U|\leq C(\lambda,\rho)\coloneqq \frac{\lambda}{\rho \min_{m\in\mathcal{M}}u_m}$, then there exists an $N_1\in\N_0$ such that for all $N\geq (N_{\varepsilon}\vee N_1)$, 
\begin{equation*}
    \Big(\sum_{j\in U}\sum_{m\in\mathcal{M}}\mathds{1}_{(j\in V^N_m)}u_m\Big)^{-1}\sum_{i\in W^N}\sum_{\substack{S\subseteq(U\cap \mathcal{N}^N_w(i)):\\|S|=d}}\frac{\lambda}{{\delta^N_i\choose d}}
    \leq  \frac{1}{|U|\min_{m\in\mathcal{M}}u_m}\sum_{k\in\mathcal{K}}\sum_{i\in W^N_k}\frac{\lambda {|C(\lambda,\rho)|\choose d} }{{\delta^N_i \choose d}}\leq \rho,
\end{equation*}
since for all $i\in W^N$, $\delta^N_i$ goes to infinity as $N\rightarrow\infty$ uniformly by \eqref{eq:uniform-deg-dispatcher}. Next, consider the case $|U|>C(\lambda,\rho)$.
Denote $\alpha_m=|U\cap V^N_m|/|V^N_m|$ for each $m\in\mathcal{M}$. 
Then, 
\begin{equation}\label{eq:subset-rho-1}
    \Big(\sum_{j\in U}\sum_{m\in\mathcal{M}}\mathds{1}_{(j\in V^N_m)}u_m\Big)^{-1}\sum_{i\in W^N}\sum_{\substack{S\subseteq(U\cap \mathcal{N}^N_w(i)):\\|S|=d}}\frac{\lambda}{{\delta^N_w(i)\choose d}}
    \leq \Big(\sum_{m\in\mathcal{M}}\lfloor |V_m^N|\alpha_m \rfloor u_m \Big)^{-1}\sum_{k\in\mathcal{K}}\sum_{i\in W^N_k}\frac{\lambda {|U\cap \mathcal{N}^N_w(i)|\choose d} }{{\delta^N_w(i)\choose d}}.
\end{equation}
By \eqref{eq:uniform-deg-server}, we have that for each $k\in\mathcal{K}$
\begin{equation}\label{eq:edge-U}
    \begin{split}
        &\sum_{i\in W^N_k}|U\cap \mathcal{N}^N_w(i)|=\sum_{j\in U}\deg^N_v(k,j)\leq  \sum_{m\in\mathcal{M}}|V^N_m|\alpha_m p_{k,m}w_kW(N) (1+\varepsilon).
    \end{split}
\end{equation}
By Lemma~\ref{lem:opt-prob} and \eqref{eq:uniform-deg-dispatcher}, \eqref{eq:edge-U}, 
\begin{equation}\label{eq:subset-rho-2}
\begin{split}
    \eqref{eq:subset-rho-1}& \leq  \Big(\sum_{m\in\mathcal{M}}\lfloor |V_m^N|\alpha_m \rfloor u_m \Big)^{-1}\lambda\sum_{k\in\mathcal{K}}\Big(\left\lfloor\frac{\sum_{m\in\mathcal{M}}|V^N_m|\alpha_m p_{k,m}w_kW(N)(1+\varepsilon) }{\delta_k N (1-\varepsilon)}\right\rfloor+1\Big) \frac{{\delta_kN(1+\varepsilon) \choose d}}{{\delta_kN(1-\varepsilon) \choose d}}\\
    & \leq  C_1(N)\Big(\frac{1+\varepsilon}{1-\varepsilon}\Big)^d\Big(\sum_{m\in\mathcal{M}}v_m\alpha_m u_m \Big)^{-1}\lambda\xi\sum_{k\in\mathcal{K}}\Big(w_k\left\lfloor\frac{\sum_{m\in\mathcal{M}}v_m\alpha_m p_{k,m} (1+\varepsilon)}{\delta_k(1-\varepsilon)}\right\rfloor+\frac{1}{N}\Big)\\
\end{split}
\end{equation}
where $C_1(N)$ only depends on $N$ and goes to 1 as $N\rightarrow\infty$. 
Let $$\mathcal{K}'\coloneqq\Big\{k\in\mathcal{K}:\left\lfloor\frac{\sum_{m\in\mathcal{M}}v_m\alpha_m p_{k,m} (1+\varepsilon)}{\delta_k(1-\varepsilon)}\right\rfloor\geq 1\Big\}.$$ 
If $\mathcal{K}'=\emptyset$, then 
\begin{equation}\label{eq:subset-rho-3}
    \begin{split}
        \eqref{eq:subset-rho-2}& \leq  C_1(N)\Big(\frac{1+\varepsilon}{1-\varepsilon}\Big)^d\frac{\lambda \xi}{N\sum_{m\in\mathcal{M}}v_m\alpha_m u_m}\\
        & \leq  C_1(N)\Big(\frac{1+\varepsilon}{1-\varepsilon}\Big)^d\frac{\lambda \xi}{C(\lambda,\rho)\min_{m\in\mathcal{M}}u_m}\leq \rho\Big(\frac{1+\varepsilon}{1-\varepsilon}\Big)^d.
    \end{split}
\end{equation}
Consider the case $\mathcal{K}'\neq \emptyset$. Then, we get 
\begin{equation}\label{eq:subset-rho-4}
    \begin{split}
        \eqref{eq:subset-rho-2}& \leq  C_1(N)\Big(\frac{1+\varepsilon}{1-\varepsilon}\Big)^d\Big(\sum_{m\in\mathcal{M}}v_m\alpha_m u_m \Big)^{-1}\lambda\xi\Big(\sum_{k\in\mathcal{K}}w_k\frac{\sum_{m\in\mathcal{M}}v_m\alpha_m p_{k,m} (1+\varepsilon)}{\delta_k(1-\varepsilon)}+\frac{1}{N}\Big)\\
        & \leq  C_1(N)\Big(\frac{1+\varepsilon}{1-\varepsilon}\Big)^d\Big(\sum_{m\in\mathcal{M}}v_m\alpha_m u_m \Big)^{-1}\lambda\xi\Big(\sum_{k\in\mathcal{K}}w_k\frac{\sum_{m\in\mathcal{M}}v_m\alpha_m p_{k,m} (1+\varepsilon)}{\delta_k(1-\varepsilon)}\Big)\\
        &\quad +C_1(N)\Big(\frac{1+\varepsilon}{1-\varepsilon}\Big)^d\frac{\lambda\xi}{N\sum_{m\in\mathcal{M}}v_m\alpha_mu_m}. 
    \end{split}
\end{equation}
By \eqref{eq:max-rho-set}, we have that for all $m\in\mathcal{M}$ and $\alpha_m\in(0,1)$, 
$$(\alpha_m v_mu_m)^{-1}\lambda\xi\Big(\sum_{k\in\mathcal{K}}w_k\frac{\alpha_m p_{k,m}v_m}{\delta_k}\Big)\frac{1+\varepsilon}{1-\varepsilon}\leq \frac{\rho(1+\varepsilon)}{1-\varepsilon}<\frac{1+\rho}{2},$$
which implies that 
\begin{equation}\label{eq:1-eps-sum-m}
    \lambda \xi \Big(\sum_{k\in\mathcal{K}}w_k\frac{\sum_{m\in\mathcal{M}}v_m\alpha_m p_{k,m}(1+\varepsilon) }{\delta_k(1-\varepsilon)}\Big) < \frac{1+\rho}{2}\Big(\sum_{m\in\mathcal{M}}v_m\alpha_m u_m \Big).
\end{equation}
Since $\mathcal{K}'$ is nonempty, then we assume $k'\in\mathcal{K}$ is in $\mathcal{K}'$, i.e., $\sum_{m\in\mathcal{M}}v_m\alpha_mp_{k',m}(1+\varepsilon)\geq \delta_{k'}(1-\varepsilon)$. Hence, 
\begin{equation}
    \begin{split}
        \frac{\lambda\xi}{N\sum_{m\in\mathcal{M}}v_m\alpha_mu_m}\leq \frac{\lambda\xi (1+\varepsilon)}{N\delta_{k'}(1-\varepsilon)\min_{m\in\mathcal{M}}u_m}\leq \frac{\lambda\xi \rho}{N\delta_{k'}\min_{m\in\mathcal{M}}u_m},
    \end{split}
\end{equation}
which implies that there exists $N_2\in\N_0$ such that for all $N\geq N_2$,
\begin{equation}\label{eq:lambda-xi-N}
    \begin{split}
        \frac{\lambda\xi}{N\sum_{m\in\mathcal{M}}v_m\alpha_mu_m}\leq \frac{\lambda\xi (1+\varepsilon)}{N\delta_{k'}(1-\varepsilon)\min_{m\in\mathcal{M}}u_m}<C_2(N)\xrightarrow{N\rightarrow\infty}0.
    \end{split}
\end{equation}
We choose $\varepsilon$ such that $\Big(\frac{1+\varepsilon}{1-\varepsilon}\Big)^d\frac{1+\rho}{2}<1$.
By \eqref{eq:1-eps-sum-m} and \eqref{eq:lambda-xi-N}, we have that there exists a positive integer $N_3\geq (N_{\varepsilon}\vee N_1 \vee N_2)$ such that for all $N_3\in \N_0$, 
\begin{equation}\label{eq:A-16}
        \Big(\sum_{j\in U}\sum_{m\in\mathcal{M}'}\mathds{1}_{(j\in V^N_m)}u_m\Big)^{-1}\sum_{i\in W^N}\sum_{\substack{S\subseteq(U\cap \mathcal{N}^N_w(i)):\\|S|=d}}\frac{\lambda}{{\delta^N_w(i)\choose d}}< \Big(\frac{1+\varepsilon}{1-\varepsilon}\Big)^d\Big(C_1(N)\frac{1+\rho}{2}+C_2(N)\Big)<1.
\end{equation}
We choose $\varepsilon$ such that $\Big(\frac{1+\varepsilon}{1-\varepsilon}\Big)^d\frac{1+\rho}{2}<1$.
Now, since the subset $U\subseteq V^N$ is arbitrary, then for all $N\geq N_3$, the $N$-th system is stable under JSQ($d$) policy. 
\end{proof}

\begin{proof}[Proof of Proposition~\ref{prop:p-matrix-existence}]
By Lemma~\ref{lem:relationship-v-u-pkm}, 
it is sufficient to show that there exists some $\mathbf{p}$ such that for each $m\in\mathcal{M}$, 
\begin{equation}\label{eq:d=1-1}
    \lambda\xi\sum_{k\in\mathcal{K}}w_k\frac{v_m p_{k,m}}{\delta_k}<v_m u_m.
\end{equation}
Let $x_{m,k}=\frac{v_m p_{k,m}}{\delta_k}\in[0,1]$, $k\in\mathcal{K}$, $m\in\mathcal{M}$ with $\sum_{m\in\mathcal{M}}x_{m,k}=1$. Now, we can formulate a linear optimization problem as the following:
the objective is $ \min\ \rho $ and the 
constraints are 
\begin{equation}\label{eq:constraints}
    \begin{split}
        \lambda\xi \sum_{k\in\mathcal{K}}w_kx_{k,m}\  & \leq \  \rho v_m u_m,\quad \forall m\in\mathcal{M},\\
        \sum_{m\in\mathcal{M}}x_{k,m}\ &= \ 1,\quad \forall k\in\mathcal{K},\\
        x_{k,m}\ \in & \ [0,1],\quad \forall k\in \mathcal{K},m\in\mathcal{M}.
    \end{split}
\end{equation}
\noindent
Next we construct a specific solution $\mathbf{x}'=(x'_{k,m},k\in\mathcal{K},m\in\mathcal{M})$ satisfying the above constraints~\eqref{eq:constraints} with
$\rho_0 = \lambda \xi/\sum_{m\in\mathcal{M}}v_mu_m$. Note that $\rho_0 < 1$ by \eqref{eq:capacity}.
For convenience, we denote $x'_{k,0}=0$ for all $k\in\mathcal{K}$. First, consider $k=1$. 
Let $x'_{1,1}=\frac{\min\big(\rho_0 v_1u_1,  \lambda\xi w_1\big)}{\lambda\xi w_1}$ and for $m\geq 2$, $$x'_{1,m}=\frac{\min\big(\rho_0 v_mu_m,  \lambda\xi w_1(1-\sum_{m'<m}x'_{1,m'})\big)}{\lambda\xi w_1}.$$
Since $\lambda \xi w_1\leq \lambda \xi = \rho_0 \sum_{m\in\mathcal{M}}v_mu_m$, then $\sum_{m\in\mathcal{M}}x'_{1,m}=1$  
and $m_1\coloneqq\min\{m\in\mathcal{M}:\rho_0v_mu_m-x'_{1,m}\lambda\xi w_1>0\}\in\mathcal{M}$. 
Then, consider $k=2$. For all $m<m_1$, let $x'_{2,m}=0$. Let 
$$x'_{2,m_1}=\frac{\min(\rho_0v_{m_1}u_{m_1}-x'_{1,m_1}\lambda\xi w_1, \lambda\xi w_2)}{\lambda\xi w_2},$$
and 
$$x'_{2,m}=\frac{\min\big(\rho_0 v_mu_m,  \lambda\xi w_2(1-\sum_{m'<m}x'_{2,m'})\big)}{\lambda\xi w_2},\quad m>m_1.$$
Again, since $\lambda \xi (w_1+w_2)\leq \lambda \xi \leq \rho_0 \sum_{m\in\mathcal{M}}v_mu_m$, then $\sum_{m\in\mathcal{M}}x'_{2,m}=1$ and $m_2\coloneqq\min\{m\geq m_1:\rho_0v_mu_m-x'_{2,m}\lambda\xi w_2>0\}\in\mathcal{M}$. 
We can construct $x'_{k,m}$, $m\in\mathcal{M}$, $k\geq 3$ by following the steps of the construction of $x'_{2,m}$, $m\in\mathcal{M}$. Hence, we get a specific solution $\mathbf{x}'$ satisfying \eqref{eq:constraints} with $\rho_0<1$. Therefore, $\min\ \rho$ is strictly less than 1 and our desired result holds.
\end{proof}

\section{Approximation of Graph Structure for Large $N$ Systems}\label{app:approx-graph}

\begin{proof}[Proof of Lemma~\ref{lem: LLN-xi}]
Consider any fixed $m\in\mathcal{M}$ and fixed $j\in V_m$. Also, fix any $k\in \mathcal{K}$ and $(M_2,...,M_d)\in \mathcal{M}^{d-1}$. 
\allowdisplaybreaks
\begin{align}
    &\Big|\sum_{i\in W^N_k}\xi^N_{i,j}\sum_{\substack{(j_2,...,j_d)\in \set^N (j)\\ s.t. \quad j_2\in V^N_{M_2},...,j_d\in V^N_{M_d}}}\frac{\xi^N_{i,j_2}\times \cdots \times \xi^N_{i,j_d}}{{\delta^N_i\choose d}(d-1)!}-d\xi\frac{p_{k,m}w_k}{\delta_k}\prod_{h=2}^d\frac{v_{M_h}p_{k,M_h}}{\delta_k}\Big|\notag\\
    & \leq \Big|\sum_{i\in W^N_k}\xi^N_{i,j}\sum_{\substack{(j_2,...,j_d)\in \set^N (j)\\ s.t. \quad j_2\in V^N_{M_2},...,j_d\in V^N_{M_d}}}\frac{\xi^N_{i,j_2}\times \cdots \times \xi^N_{i,j_d}}{{\delta^N_i\choose d}(d-1)!}-\sum_{i\in W^N_k}\xi^N_{i,j}\frac{{\delta^N_i\choose d-1}}{{\delta^N_i\choose d}}\prod_{h=2}^d\frac{v_{M_h}p_{k,M_h}}{\delta_k}\Big| \label{eq:avg-deg-1-1}\\
    &\quad +\Big|\sum_{i\in W^N_k}\xi^N_{i,j}\frac{{\delta^N_i\choose d-1}}{{\delta^N_i\choose d}}\prod_{h=2}^d\frac{v_{M_h}p_{k,M_h}}{\delta_k}-d\xi\frac{p_{k,m}w_k}{\delta_k}\prod_{h=2}^d\frac{v_{M_h}p_{k,M_h}}{\delta_k}\Big| \label{eq:avg-deg-1-2}
\end{align}

First, 
\begin{align}
    &\max_{i\in W^N_k}\Big|\sum_{\substack{(j_2,...,j_d)\in \set^N (j)\\ s.t. \quad j_2\in V^N_{M_2},...,j_d\in V^N_{M_d}}}\frac{\xi^N_{i,j_2}\times \cdots \times \xi^N_{i,j_d}}{{\delta^N_i\choose d-1}(d-1)!}-\prod_{h=2}^d\frac{v_{M_h}p_{k,M_h}}{\delta_k}\Big|\nonumber\\
       \leq  & \max_{i\in W^N_k}\Big|\sum_{\substack{(j_2,...,j_d)\in \set^N (j)\\ s.t. \quad j_2\in V^N_{M_2},...,j_d\in V^N_{M_d}}}\frac{\xi^N_{i,j_2}\times \cdots \times \xi^N_{i,j_d}}{{\delta^N_i\choose d-1}(d-1)!}-\frac{\deg^N_w(i,M_2)\times \cdots \times \deg^N_w(i,M_d)}{{\delta^N_i\choose d-1}(d-1)!}\Big|\nonumber\\
       &+\max_{i\in W^N_k}\Big|\frac{\deg^N_w(i,M_2)\times \cdots \times \deg^N_w(i,M_d)}{{\delta^N_i\choose d-1}(d-1)!} -\prod_{h=2}^d\frac{v_{M_h}p_{k,M_h}}{\delta_k}\Big|.\nonumber
\end{align}

For large enough $N$,
\begin{align}
        &\max_{i\in W^N_k}\Big|\sum_{\substack{(j_2,...,j_d)\in \set^N (j)\\ s.t. \quad j_2\in V^N_{M_2},...,j_d\in V^N_{M_d}}}\frac{\xi^N_{i,j_2}\times \cdots \times \xi^N_{i,j_d}}{{\delta^N_i\choose d-1}(d-1)!}-\frac{\deg^N_w(i,M_2)\times \cdots \times \deg^N_w(i,M_d)}{{\delta^N_i\choose d-1}(d-1)!}\Big| \notag \\
        & \leq  \max_{i\in W^N_k}\frac{d(d-1)}{{\delta^N_i\choose d-1}(d-1)!}\max_{m\in\mathcal{M}}(\deg^N_w(i,m))^{d-2} \notag \\
        & \leq  \frac{d(d-1)}{\min_{i\in W^N_k}{\delta^N_i\choose d-1}(d-1)!}\max_{i\in W^N_k}\max_{m\in\mathcal{M}}(\deg^N_w(i,m))^{d-2} \notag \\
        & \leq  c^N(m,k)d(d-1)\frac{(N\max_{m\in\mathcal{M}}v_mp_{k,m})^{d-2}}{(N\delta_k)^{d-1}}\xrightarrow{N\rightarrow\infty}0, \label{eq:diff-connection}
\end{align}
where $c^N(m,k)$ goes to 1 as $N$ goes to infinity, and only depends on $k$ and $m$ for each $N$.
The last inequality comes from Condition~\ref{cond-1}, Lemma~\ref{lem:uniformity-degree}, and $\frac{\delta^N_i\times\cdots\times (\delta^N_i-d+2)}{(\delta^N_i)^{d-1}}\xrightarrow{N\rightarrow\infty}1$.
Similarly, we have 
\begin{align}
        &\max_{i\in W^N_k}\Big|\frac{\deg^N_w(i,M_2)\times \cdots \times \deg^N_w(i,M_d)}{{\delta^N_i\choose d-1}(d-1)!} -\prod_{h=2}^d\frac{v_{M_h}p_{k,M_h}}{\delta_k}\Big| \notag \\
        & \leq  \max \Big( \prod_{h=2}^d \big(\frac{\max_{i\in W^N_k}\deg^N_w(i,M_h)}{\min_{i\in W^N_k}(\delta^N_i-d)}-\frac{v_{M_h}p_{k,M_h}}{\delta_k}\big), \prod_{h=2}^d \big(\frac{\min_{i\in W^N_k}\deg^N_w(i,M_h)}{\max_{i\in W^N_k}\delta^N_i}-\frac{v_{M_h}p_{k,M_h}}{\delta_k}\big)\Big) \notag \\
        & \leq c^N(m,k,M_2,...,M_d)\xrightarrow{N\rightarrow\infty}0, \label{eq:avg-prob}
\end{align}
where $c^N(m,k,M_2,...,M_d)$ depends on $m,k,M_2,...,M_d$.
By \eqref{eq:diff-connection} and \eqref{eq:avg-prob}, we have 
\begin{equation}\label{eq:c1-m-k-Mh}
        \max_{i\in W^N_k}\Big|\sum_{\substack{(j_2,...,j_d)\in \set^N (j)\\ s.t. \quad j_2\in V^N_{M_2},...,j_d\in V^N_{M_d}}}\frac{\xi^N_{i,j_2}\times \cdots \times \xi^N_{i,j_d}}{{\delta^N_i\choose d-1}(d-1)!}-\prod_{h=2}^d\frac{v_{M_h}p_{k,M_h}}{\delta_k}\Big| 
         \leq  c_1^N(m,k,M_2,...,M_d)\xrightarrow{N\rightarrow\infty}0,
\end{equation}
where $c_1^N(m,k,M_2,...,M_d)$ depends on $m,k,M_2,...,M_d$.
\noindent
By Lemma~\ref{lem:uniformity-degree}, we have  
\begin{align*}
    \lim_{N\rightarrow\infty} \max_{i\in W^N_k}\frac{N{\delta^N_i\choose d-1}}{{\delta^N_i\choose d}} & =\lim_{N\rightarrow\infty}\min_{i\in W^N_k}\frac{N{\delta^N_i\choose d-1}}{{\delta^N_i\choose d}}=\frac{d}{\delta_k}, \\
    \lim_{N\rightarrow\infty} \max_{j\in V^N_m}\frac{\deg^N_v(k,j)}{N} & =\lim_{N\rightarrow\infty}\min_{j\in V^N_m}\frac{\deg^N_v(k,j)}{N}=\xi p_{k,m}w_k.
\end{align*}
Then, 
\begin{align}
    \Big|\sum_{i\in W^N_k}\xi^N_{i,j}\frac{{\delta^N_i\choose d-1}}{{\delta^N_i\choose d}}-d\xi \frac{p_{k,m}w_k}{\delta_k}\Big|
    & \leq  \Big|\sum_{i\in W^N_k}\xi^N_{i,j}\frac{{\delta^N_i\choose d-1}}{{\delta^N_i\choose d}}-\deg^N_v(k,j)\frac{d}{N\delta_k}\Big|+\Big|\deg^N_v(k,j)\frac{d}{N\delta_k}-d\xi \frac{p_{k,m}w_k}{\delta_k}\Big| \notag \\
    & \leq  c^N_1(m,k) \xrightarrow{N\rightarrow\infty}0, \label{eq:c1-m-k}
\end{align}
where $c^N_1(m,k)$ only depends on $m$ and $k$.

\noindent
Consider $\eqref{eq:avg-deg-1-1}$.
\begin{align}
    &\Big|\sum_{i\in W^N_k}\xi^N_{i,j}\sum_{\substack{(j_2,...,j_d)\in \set^N (j)\nonumber\\ s.t. \quad j_2\in V^N_{M_2},...,j_d\in V^N_{M_d}}}\frac{\xi^N_{i,j_2}\times \cdots \times \xi^N_{i,j_d}}{{\delta^N_i\choose d}(d-1)!}-\sum_{i\in W^N_k}\xi^N_{i,j}\frac{{\delta^N_i\choose d-1}}{{\delta^N_i\choose d}}\prod_{h=2}^d\frac{v_{M_h}p_{k,M_h}}{\delta_k}\Big|\nonumber\\
    &=  \Big|\sum_{i\in W^N_k}\xi^N_{i,j}\frac{{\delta^N_i\choose d-1}}{{\delta^N_i\choose d}}\sum_{\substack{(j_2,...,j_d)\in \set^N (j)\\ s.t. \quad j_2\in V^N_{M_2},...,j_d\in V^N_{M_d}}}\frac{\xi^N_{i,j_2}\times \cdots \times \xi^N_{i,j_d}}{{\delta^N_i\choose d-1}(d-1)!}-\sum_{i\in W^N_k}\xi^N_{i,j}\frac{{\delta^N_i\choose d-1}}{{\delta^N_i\choose d}}\prod_{h=2}^d\frac{v_{M_h}p_{k,M_h}}{\delta_k}\Big|\nonumber\\
    & \leq  \sum_{i\in W^N_k}\xi^N_{i,j}\frac{{\delta^N_i\choose d-1}}{{\delta^N_i\choose d}}\Big|\sum_{\substack{(j_2,...,j_d)\in \set^N (j)\nonumber\\ s.t. \quad j_2\in V^N_{M_2},...,j_d\in V^N_{M_d}}}\frac{\xi^N_{i,j_2}\times \cdots \times \xi^N_{i,j_d}}{{\delta^N_i\choose d-1}(d-1)!}-\prod_{h=2}^d\frac{v_{M_h}p_{k,M_h}}{\delta_k}\Big|\nonumber\\
    &\overset{(a)}{\leq} \sum_{i\in W^N_k}\xi^N_{i,j}\frac{{\delta^N_i\choose d-1}}{{\delta^N_i\choose d}}c_2^N(m,k,M_2,...,M_d)\nonumber\\
    &\overset{(b)}{\leq}  c_2^N(m,k,M_2,...,M_d)c^N_2(m,k)d\xi\frac{p_{k,m}w_k}{\delta_k}\xrightarrow{N\rightarrow\infty}0,
\end{align}
where $c_2^N(m,k,M_2,...,M_d)\xrightarrow{N\rightarrow\infty}0$  and $c^N_2(m,k)\xrightarrow{N\rightarrow\infty}1$.
(a) is from \eqref{eq:c1-m-k-Mh} and (b) is from \eqref{eq:c1-m-k}. Hence, \eqref{eq:avg-deg-1-1} goes to $0$ as $N\rightarrow\infty$. Then,
\begin{equation}
    \begin{split}
        &\Big|\sum_{i\in W^N_k}\xi^N_{i,j}\sum_{\substack{(j_2,...,j_d)\in \set^N (j)\\ s.t. \quad j_2\in V^N_{M_2},...,j_d\in V^N_{M_d}}}\frac{\xi^N_{i,j_2}\times \cdots \times \xi^N_{i,j_d}}{{\delta^N_i\choose d}(d-1)!}-d\xi\frac{p_{k,m}w_k}{\delta_k}\prod_{h=2}^d\frac{v_{M_h}p_{k,M_h}}{\delta_k}\Big|\\
        & \leq  c_3^N(m,k,M_2,...,M_d)\xrightarrow{N\rightarrow\infty}0,
    \end{split}
\end{equation}
where $c_3^N(m,k,M_2,...,M_d)$ only depends on $m,k,M_2,...,M_d$.
Since $k\in \mathcal{K}$ and $(M_2,...,M_d)\in \mathcal{M}^{d-1}$ are arbitrary, and $\mathcal{K}$ and $\mathcal{M}^{d-1}$ are finite sets, we have 
\begin{equation}
    \begin{split}
        &\max_{k\in\mathcal{K}}\max_{(M_2,...,M_d)\in \mathcal{M}^{d-1}}\Big|\sum_{i\in W^N_k}\xi^N_{i,j}\sum_{\substack{(j_2,...,j_d)\in \set^N (j)\\ s.t. \quad j_2\in V^N_{M_2},...,j_d\in V^N_{M_d}}}\frac{\xi^N_{i,j_2}\times \cdots \times \xi^N_{i,j_d}}{{\delta^N_i\choose d}(d-1)!}-d\xi\frac{p_{k,m}w_k}{\delta_k}\prod_{h=2}^d\frac{v_{M_h}p_{k,M_h}}{\delta_k}\Big|\\
        & \leq  c^N(m)\xrightarrow{N\rightarrow\infty}0,
    \end{split}
\end{equation}
where $c^N(m)$ only depends on $m$. Since $c^N(m)$ does not depend on $j\in V^N_m$, \eqref{eq:ave-deg} holds.
\end{proof}

\begin{proof}[Proof of Lemma~\ref{lem:bound-sett}]
Fix any $m\in\mathcal{M}$ and $j\in V_m$. Consider \eqref{eq:ave-deg-sett-1}. When $\xi^N_{i,j}=1$, by the definition \eqref{defn:sett-j} of $\sett^N(\cdot)$, 
\begin{equation*}
    \sum_{\sett^N{(j)}} \frac{\xi^N_{i,j}\times\xi^N_{i,j_2}\times\cdots\times \xi^N_{i,j_d}}{{\delta^N_{i}\choose d}(d-1)!} \frac{\xi^N_{i,j}\times\xi^N_{i,j'_2}\times\cdots\times \xi^N_{i,j'_d}}{{\delta^N_{i}\choose d}(d-1)!}=\frac{\Big[(d-1)!{\delta^N_i-1\choose d-1}\Big]^2-(2d-2)!{\delta^N_i-1\choose 2d-2}}{{\delta^N_{i}\choose d}^2\big((d-1)!\big)^2}.
\end{equation*}
Also, by Lemma~\ref{lem:uniformity-degree}, we have that for all $k\in\mathcal{K}$ and $i\in W_k$,
\begin{align*}
        \frac{\Big[(d-1)!{\delta^N_{i}-1\choose d}\Big]^2-(2d-2)!{\delta^N_i-1\choose 2d-2}}{{\delta^N_{i}\choose d}^2\big((d-1)!\big)^2}
        & \leq  \frac{\Big[(d-1)!\max_{i\in W^N_k}{\delta^N_{i}-1\choose d}\Big]^2-(2d-2)!\min_{i\in W^N_k}{\delta^N_i-1\choose 2d-2}}{\min_{i\in W^N_k}{\delta^N_{i}\choose d}^2\big((d-1)!\big)^2}\\
        & \leq  c_1(N) \frac{\Big[(d-1)!{N\delta_k\choose d-1}\Big]^2-(2d-2)!{N\delta_k\choose 2d-2}}{{N\delta_k\choose d}^2\big((d-1)!\big)^2}
\end{align*}
where $c_1(N)$ only depends on $N$ and goes to $1$ as $N\rightarrow\infty$. By Lemma~\ref{lem:uniformity-degree}, we have that for all $k\in\mathcal{K}$,
$\max_{j\in W^N_m}\deg^N_v(k,j)\leq c_2(N,m) |W^N_k|p_{k,m}$ where $c_2(N,m)$ only depends on $N$ and $m$, and goes to $1$ as $N\rightarrow\infty$. Hence,
\begin{align*}
        &\sum_{i\in W^N}\sum_{\sett^N{(j)}}\frac{\xi^N_{i,j}\times\xi^N_{i,j_2}\times\cdots\times \xi^N_{i,j_d}}{{\delta^N_{i}\choose d}(d-1)!} \frac{\xi^N_{i,j}\times\xi^N_{i,j'_2}\times\cdots\times \xi^N_{i,j'_d}}{{\delta^N_{i}\choose d}(d-1)!}\\
        &=\sum_{k\in \mathcal{K}}\sum_{i\in W^N_k}\frac{\Big[(d-1)!{\delta^N_{i}-1\choose d}\Big]^2-(2d-2)!{\delta^N_i-1\choose 2d-2}}{{\delta^N_{i}\choose d}^2\big((d-1)!\big)^2}\\
        &\leq  c_1(N)\sum_{k\in \mathcal{K}}\deg^N_v(k,j)\frac{\Big[(d-1)!{N\delta_k\choose d-1}\Big]^2-(2d-2)!{N\delta_k\choose 2d-2}}{{N\delta_k\choose d}^2\big((d-1)!\big)^2}\\
        & \leq c_1(N)c_2(N,m)\sum_{k\in \mathcal{K}}|W^N_k|p_{k,m}\frac{\Big[(d-1)!{N\delta_k\choose d-1}\Big]^2-(2d-2)!{N\delta_k\choose 2d-2}}{{N\delta_k\choose d}^2\big((d-1)!\big)^2}.
\end{align*}
Let $c_3(N)=\max_{m\in\mathcal{M}}c_1(N)c_2(N,m)$ with $c_3(N)\xrightarrow{N\rightarrow\infty}1$. Then, we have that for large enough $N$,
\begin{equation}
    \begin{split}
        &\sum_{i\in W^N}\sum_{\sett^N{(j)}}\frac{\xi^N_{i,j}\times\xi^N_{i,j_2}\times\cdots\times \xi^N_{i,j_d}}{{\delta^N_{i}\choose d}(d-1)!} \frac{\xi^N_{i,j}\times\xi^N_{i,j'_2}\times\cdots\times \xi^N_{i,j'_d}}{{\delta^N_{i}\choose d}(d-1)!}\\
        & \leq  c_3(N)\sum_{k\in \mathcal{K}}|W^N_k|p_{k,m}\frac{\Big[(d-1)!{N\delta_k\choose d-1}\Big]^2-(2d-2)!{N\delta_k\choose 2d-2}}{{N\delta_k\choose d}^2\big((d-1)!\big)^2}\\
        & \leq  2\sum_{k\in \mathcal{K}}|W^N_k|p_{k,m}\frac{\Big[(d-1)!{N\delta_k\choose d-1}\Big]^2-(2d-2)!{N\delta_k\choose 2d-2}}{{N\delta_k\choose d}^2\big((d-1)!\big)^2}
    \end{split}
\end{equation}
 Since $\lim_{N\rightarrow\infty}|W^N_k|/N=\xi w_k$ and $\Big[(d-1)!{x\choose d-1}\Big]^2-(2d-2)!{x\choose 2d-2}\leq C_3 x^{2d-3}$ for some constant $C_3$, then by choosing $C_1$ appropriately, \eqref{eq:ave-deg-sett-1} holds for all large enough $N$. We can get \eqref{eq:ave-deg-sett-2} in a similar way.
\end{proof}

\section{Unique Solution of ODE~\eqref{eq:ode-sys-2}}\label{app:unique-sol-ode}
\begin{proof}[Proof of Lemma~\ref{lem:unique-sol-ode}]
Recall $\bar{\mathbf{q}}(t,\mathbf{q}_0)$ is a solution of \eqref{eq:ode-sys-2} given the initial point $\mathbf{q}^N(0)=\mathbf{q}_0$. For convenience, we denote $\bar{\mathbf{q}}(t,\mathbf{q}_0)$ as $\bar{\mathbf{q}}(t)$ and write the ODE \eqref{eq:ode-sys-2} as the following:
\begin{equation}\label{eq:thm-ode-1}
    \bar{\mathbf{q}}(0)=\mathbf{q}_0, \quad
    \Dot{\mathbf{q}}(t)=\bar{\mathbf{h}}(\bar{\mathbf{q}}(t)),
\end{equation}
where for all $m\in\mathcal{M}$,
\begin{align}
    \bar{h}_{m,0}(\mathbf{q})&=0, \notag \\
    \bar{h}_{m,l}(\mathbf{q})&=-u_m(q_{m,l}-q_{m,l+1})+\lambda\xi(q_{m,l-1}-q_{m,l})\sum_{k\in\mathcal{K}}\frac{p_{k,m}w_k}{\delta_k}\frac{(\tilde{q}_{k,l-1})^d-(\tilde{q}_{k,l})^d}{\tilde{q}_{k,l-1}-\tilde{q}_{k,l}}, \quad l\geq 1.\label{eq:thm-ode-4}
\end{align}
 observe that under \eqref{eq:thm-ode-4}, if $q_{m,l}(t)=q_{m,l+1}(t)$ for some $m\in\mathcal{M},l\in\N_0,t\geq 0$, then $\bar{h}_{m,l}(\mathbf{q}(t))\geq 0$ and $\bar{h}_{m,l+1}(\mathbf{q}(t))\leq 0$; if $q_{m,l}(t)=0$ for some $m\in\mathcal{M},l\in\N_0,t\geq 0$, then $\bar{h}_{m,l}(\mathbf{q}(t))\geq 0$. Hence, if $\mathbf{q}\in\bar{\mathcal{S}}$, then any solution of \eqref{eq:thm-ode-1}-\eqref{eq:thm-ode-4} remains within $\bar{\mathcal{S}}$. In order to show the existence and the uniqueness, we use the Picard successive approximation method(\cite[Theorem 1(i)]{MS99}). In the rest of the proof, we use the norm:
 $$\norm{\mathbf{q}}=\sup_{m\in\mathcal{M}}\sup_{l\in\N_0}\frac{|q_{m,l}|}{l+1}.$$
 For any $\mathbf{q}$, $\mathbf{q}'\in\bar{\mathcal{S}}$, 
\begin{equation}\label{eq:lip-K2}
    \norm{\bar{\mathbf{h}}(\mathbf{q})} \leq  K_1,\quad
    \norm{\bar{\mathbf{h}}(\mathbf{q})-\bar{\mathbf{h}}(\mathbf{q}')}\leq  K_2\norm{\mathbf{q}-\mathbf{q}'},
\end{equation}
where $K_1\coloneqq\max_{m\in\mathcal{M}} u_m + \lambda\xi$ and $K_2\coloneqq 2 \max_{m\in\mathcal{M}} u_m +2d\lambda\xi$.
For $t\geq 0$, let $\mathbf{q}^{(0)}(t)=\mathbf{q}_0$, and by Picard successive approximation method, let 
\begin{equation*}
    \mathbf{q}^{(n)}(t)=\mathbf{q}_0+\int_0^t\bar{\mathbf{h}}(\mathbf{q}^{(n-1)}(s))ds,\quad n\in\N.
\end{equation*}
By induction, we have that $\mathbf{q}^{(n)}(t)$ is continuous w.r.t. $t$ on $[0,\infty)$ for all $n$, and that 
\begin{equation*}
    \norm{\mathbf{q}^{(n+1)}(t)-\mathbf{q}^{(n)}(t)}\leq \frac{K_1K_2^nt^{n+1}}{(n+1)!},\quad \forall n\in\N,t\geq 0.
\end{equation*}
Hence, for all $t\geq 0$, $\mathbf{q}^{(\infty)}=\lim_{n\rightarrow\infty}\mathbf{q}^{(n)}$ exists uniformly for $s\in[0,t]$. Also, by \eqref{eq:lip-K2} and Dominated Convergence Theorem, the following holds
\begin{equation}
    \mathbf{q}^{(\infty)}(t)=\mathbf{q}_0+\int_0^t\bar{\mathbf{h}}(\mathbf{q}^{(\infty)}(s))ds.
\end{equation}
Next, we show the uniqueness by contradiction. Assume that $\tilde{\mathbf{q}}^{(\infty)}$ also satisfies 
$$\tilde{\mathbf{q}}^{(\infty)}(t)=\mathbf{q}_0+\int_0^t\bar{\mathbf{h}}(\tilde{\mathbf{q}}^{(\infty)}(s))ds.$$
Then, we have 
$$\tilde{\mathbf{q}}^{(\infty)}(t)-\mathbf{q}^{(n)}(t)=\int_0^t\big[\bar{\mathbf{h}}(\tilde{\mathbf{q}}^{(\infty)}(s))-\bar{\mathbf{h}}(\mathbf{q}^{(n-1)}(s))\big]ds.$$
Similarly, we get
$$\norm{\tilde{\mathbf{q}}^{(\infty)}(t)-\mathbf{q}^{(n)}(t)}\leq \frac{K_1K_2^nt^{n+1}}{(n+1)!}$$
which implies that $\tilde{\mathbf{q}}^{(\infty)}(t)=\lim_{n\rightarrow\infty}\mathbf{q}^{(n)}(t)=\mathbf{q}^{(\infty)}$.
\end{proof}

\section{Proof of Proposition~\ref{prop: num-bad-disp}}\label{app:proof-of-bad-dispatcher}
\begin{lemma}\label{lem:sup-q-ml-t}
If $\mathbf{q}^N(0)$ weakly converges to $\mathbf{q}(0)=\mathbf{q}^{\infty}\in\mathcal{S}$, then for any $\varepsilon>0$, $\delta>0$, and $T>0$, there exist $\ell\in\N_0$ and $N_{\ell}\in \N_0$, depending on $\mathbf{q}^{\infty}$, $\varepsilon$, $\delta$, and $T$, such that, for all $N\geq N_1$,
\begin{equation}
    \PP\Big(\sup_{t\in[0,T]}\sup_{m\in\mathcal{M}}q^N_{m,\ell}(t)\geq \varepsilon\Big)<\delta.
\end{equation}
\end{lemma}
\begin{proof}
Fix any $\varepsilon>0$ and $\delta>0$. Since $\mathbf{q}^{\infty}\in\mathcal{S}$, then there exists $\ell_1\in\N_0$ such that $\sup_{m\in\mathcal{M}}q^{\infty}_{m,\ell_1}\leq \varepsilon/4$. By the weak convergence $\mathbf{q}^N(0)\Rightarrow \mathbf{q}^{\infty}$, there exists $N_1\in\N_0$ such that for all $N\geq N_1$, 
\begin{equation}\label{eq:5.9-1}
    \PP\big(q^N_{m,\ell_1}(0)\leq \varepsilon/2\big)\leq \PP\Big(\norm{\mathbf{q}^N(0)-\mathbf{q}^{\infty}}_1\geq \varepsilon/4\Big)<\frac{\delta}{2}.
\end{equation}
Let $\ell=\ell_1+\sup_{m\in\mathcal{M}}\lceil \frac{4 \xi \lambda T}{v_m \varepsilon}  \rceil$. Hence, 
\begin{equation}\label{eq:5.9-3}
\begin{split}
    &\PP\Big(\sup_{t\in[0,T]}\sup_{m\in\mathcal{M}}q^N_{m,\ell}(t)\geq \varepsilon\Big)\\
    & \leq \PP\Big(\sup_{t\in[0,T]}\sup_{m\in\mathcal{M}}q^N_{m,\ell}(t)\geq \varepsilon| \sup_{m\in\mathcal{M}}q^N_{m,\ell_1}(0)<\varepsilon/2\Big)+\PP\big(q^N_{m,\ell_1}(0)\leq \varepsilon/2\big).
\end{split}
\end{equation}
Since given $\sup_{m\in\mathcal{M}}q^N_{m,\ell_1}(0)<\varepsilon/2$, i.e., for all $m\in\mathcal{M}$, $q^N_{m,\ell_1}(0) |V^N_m|<\varepsilon/2|V^N_m|$, then if for some $t\in[0,T]$ and $m\in\mathcal{M}$, $q^N_{m,\ell}(t)\geq \varepsilon$, i.e., $q^N_{m,\ell}(t)|V^N_m|\geq \varepsilon|V^N_m|$, there must be at least $\inf_{m\in\mathcal{M}}|V^N_m|\varepsilon(\ell-\ell_1)/2$ tasks arriving in the system. 
By using the standard concentration inequality for Poisson random variables (see \cite[Theorem 2.3(b)]{MC98}), we have
\begin{equation}\label{eq:poi-bound-tail}
    \begin{split}
        & \PP\Big(\sup_{t\in[0,T]}\sup_{m\in\mathcal{M}}q^N_{m,\ell}(t)\geq \varepsilon| \sup_{m\in\mathcal{M}}q^N_{m,\ell_1}(0)<\varepsilon/2\Big)
         \leq \PP\big(\mathrm{Po}(W(N)\lambda)\geq\inf_{m\in\mathcal{M}}|V^N_m|\varepsilon(\ell-\ell_1)/2 \big)\\
        & \leq \PP\big(\mathrm{Po}(N \xi \lambda T)\geq 2 C(N) N \xi \lambda T  \big)
         \leq \exp \Big(-\frac{\big((2C(N)-1)N\xi\lambda T\big)^2}{2\big(N\xi \lambda T+((2C(N)-1)N\xi\lambda T)/3\big)}\Big)\xrightarrow{N\rightarrow\infty}0,
    \end{split}
\end{equation}
where $\mathrm{Po}(\cdot)$ is a unit-rate Poisson random variable, and $C(N)$ is a positive constant only dependent on $N$ and goes to $1$ as $N$ goes to infinity. The second inequality comes from the assumption that $W(N)/N\rightarrow\xi$ and $|V^N_m|/N\rightarrow v_m$, $\forall m\in\mathcal{M}$. By \eqref{eq:poi-bound-tail}, there exists $N_2\in\N_0$ such that for all $N\geq N_2$,
\begin{equation}\label{eq:5.9-2}
    \PP\Big(\sup_{t\in[0,T]}\sup_{m\in\mathcal{M}}q^N_{m,\ell}(t)\geq \varepsilon| \sup_{m\in\mathcal{M}}q^N_{m,\ell_1}(0)<\varepsilon/2\Big)< \frac{\delta}{2}.
\end{equation} 
Let $N_0=\max(N_1,N_2)$. By \eqref{eq:5.9-1}, \eqref{eq:5.9-3} and \eqref{eq:5.9-2}, 
\begin{equation*}
\PP\Big(\sup_{t\in[0,T]}\sup_{m\in\mathcal{M}}q^N_{m,\ell}(t)\geq \varepsilon\Big)<\delta.
\end{equation*}
\end{proof}

\begin{lemma}\label{lem:appro-x-kml}
For each $m\in\mathcal{M}$ and $k\in\mathcal{K}$,
\begin{equation}
    \sup_{U\subseteq V^N_m} \Big|\frac{|E^N_k(U)|}{|E^N_k(V^N)|}-\frac{v_mp_{k,m}}{\delta_k}\frac{|U|}{|V^N_m|}\Big|\rightarrow0 \text{ as }N\rightarrow\infty.
\end{equation}
\end{lemma}
\begin{proof}
Fix any $\varepsilon>0$. By Condition~\ref{cond-1} and Lemma~\ref{lem:uniformity-degree}, there exist $N(\varepsilon)\in\N_0$ such that for all $N\geq N(\varepsilon)$, 
\begin{equation}
    (1-\varepsilon)p_{k,m}|W^N_k||U|\leq |E^N_k(U)| \leq (1+\varepsilon)p_{k,m}|W^N_k||U|,\quad \forall U\subseteq V^N_m,
\end{equation}
and
\begin{equation}
    (1-\varepsilon)\sum_{m\in\mathcal{M}}p_{k,m}|W^N_k||V^N_m|\leq |E^N_k(V^N)|\leq (1+\varepsilon)\sum_{m\in\mathcal{M}}p_{k,m}|W^N_k||V^N_m|.
\end{equation}
Hence, for all $N\geq N(\varepsilon)$,
\begin{equation}
    \sup_{U\subseteq V^N_m} \Big|\frac{|E^N_k(U)||V^N_m|}{|E^N_k(V^N)||U|}-\frac{v_mp_{k,m}}{\delta_k}\Big|\leq \max\Big\{\varepsilon_1(\varepsilon,N), \varepsilon_2(\varepsilon,N)\Big\},
\end{equation}
where $\varepsilon_1(\varepsilon,N)=\Big|\frac{(1-\varepsilon)p_{k,m}|W^N_k||V^N_m|}{(1+\varepsilon)\sum_{m\in\mathcal{M}}p_{k,m}|W^N_k||V^N_m|}-\frac{v_mp_{k,m}}{\delta_k}\Big|$ and $\varepsilon_2(\varepsilon,N)=\Big|\frac{(1+\varepsilon)p_{k,m}|W^N_k||V^N_m|}{(1-\varepsilon)\sum_{m\in\mathcal{M}}p_{k,m}|W^N_k||V^N_m|}-\frac{v_mp_{k,m}}{\delta_k}\Big|$.
Again, by Condition~\ref{cond-1} and Lemma~\ref{lem:uniformity-degree}, 
\begin{equation}\label{eq:5.53}
    \begin{split}
        & \lim_{N\rightarrow\infty}\sup_{U\subseteq V^N_m} \Big|\frac{|E^N_k(U)||V^N_m|}{|E^N_k(V^N)||U|}-\frac{v_mp_{k,m}}{\delta_k}\Big|\\
        & \leq \lim_{N\rightarrow\infty} \max\Big\{\varepsilon_1(\varepsilon,N), \varepsilon_2(\varepsilon,N)\Big\}\\
        & = \max\Big\{\Big|\frac{(1-\varepsilon)v_mp_{k,m}}{(1+\varepsilon)\delta_k}-\frac{v_mp_{k,m}}{\delta_k}\Big|, \Big|\frac{(1+\varepsilon)v_mp_{k,m}}{(1-\varepsilon)\delta_k}-\frac{v_mp_{k,m}}{\delta_k}\Big|\Big\}
    \end{split}
\end{equation}
Since \eqref{eq:5.53} holds for any $\varepsilon>0$, we have 
\begin{equation}\label{eq:5.54}
    \begin{split}
        & \lim_{N\rightarrow\infty}\sup_{U\subseteq V^N_m} \Big|\frac{|E^N_k(U)||V^N_m|}{|E^N_k(V^N)||U|}-\frac{v_mp_{k,m}}{\delta_k}\Big|\\
        & \leq \lim_{\varepsilon\downarrow 0}\max\Big\{\Big|\frac{(1-\varepsilon)v_mp_{k,m}}{(1+\varepsilon)\delta_k}-\frac{v_mp_{k,m}}{\delta_k}\Big|, \Big|\frac{(1+\varepsilon)v_mp_{k,m}}{(1-\varepsilon)\delta_k}-\frac{v_mp_{k,m}}{\delta_k}\Big|\Big\}=0.
    \end{split}
\end{equation}
\end{proof}

\begin{proof}[Proof of Proposition~\ref{prop: num-bad-disp}] 
Consider any fixed $k\in\mathcal{K}$. Also, fix $\varepsilon_1>0$ and $\varepsilon_2>0$. By the triangle inequality, we have
\allowdisplaybreaks
\begin{align}\label{eq:prop-5.9-1}
    & \PP\Big(\sup_{t\in[0,T]}\Big|\Big\{i\in W^N_k:\sum_{m\in\mathcal{M}}\sum_{l\in\N_0}|\hat{x}^N_{i,m,l}(t)-x^N_{k,m,l}(t)|>\varepsilon_1\Big\}\Big|\geq \varepsilon_2 M(N)/K\Big)\nonumber\\
    & \leq \PP\Big(\sup_{t\in[0,T]}\Big|\Big\{i\in W^N_k:\sum_{m\in\mathcal{M}}\sum_{0\leq l\leq \ell-1}|\hat{x}^N_{i,m,l}(t)-x^N_{k,m,l}(t)|>\varepsilon_1/4\Big\}\Big|\geq \varepsilon_2 M(N)/(4K)\Big)\nonumber\\
    &\quad  +\PP\Big(\sup_{t\in[0,T]}\Big|\Big\{i\in W^N_k:\sum_{m\in\mathcal{M}}\sum_{l\geq \ell}\hat{x}^N_{i,m,l}(t)>\varepsilon_1/2\Big\}\Big|\geq \varepsilon_2 M(N)/(2K)\Big)\nonumber\\
    &\quad  +\PP\Big(\sup_{t\in[0,T]}\Big|\Big\{i\in W^N_k:\sum_{m\in\mathcal{M}}\sum_{l\geq \ell}x^N_{k,m,l}(t)>\varepsilon_1/4\Big\}\Big|\geq \varepsilon_2 M(N)/(4K)\Big)\nonumber\\
    & \leq \sum_{0\leq l\leq \ell-1}\PP\Big(\sup_{t\in[0,T]}\Big|\Big\{i\in W^N_k:\sum_{m\in\mathcal{M}}|\hat{x}^N_{i,m,l}(t)-x^N_{k,m,l}(t)|>\varepsilon_1/(4\ell)\Big\}\Big|\geq \varepsilon_2 M(N)/(4\ell K)\Big)\nonumber\\
    &\quad + \PP\Big(\sup_{t\in[0,T]}\Big|\Big\{i\in W^N_k:\sum_{m\in\mathcal{M}}\big|\sum_{l\geq \ell}(\hat{x}^N_{i,m,l}(t)-x^N_{k,m,l}(t))\big|>\varepsilon_1/4\Big\}\Big|\geq \varepsilon_2 M(N)/(4K)\Big)\nonumber\\
    &\quad +2\PP\Big(\sup_{t\in[0,T]}\Big|\Big\{i\in W^N_k:\sum_{m\in\mathcal{M}}\sum_{l\geq \ell}x^N_{k,m,l}(t)>\varepsilon_1/4\Big\}\Big|\geq \varepsilon_2 M(N)/(4K)\Big)
\end{align}
By the triangle inequality and Markov's inequality, 
\allowdisplaybreaks
\begin{align}
        &\sum_{0\leq l\leq \ell-1}\PP\Big(\sup_{t\in[0,T]}\Big|\Big\{i\in W^N_k:\sum_{m\in\mathcal{M}}|\hat{x}^N_{i,m,l}(t)-x^N_{k,m,l}(t)|>\varepsilon_1/(4\ell)\Big\}\Big|\geq \varepsilon_2 M(N)/(4\ell K)\Big)\nonumber\\
        & \leq \sum_{0\leq l\leq \ell-1} \Big(\PP\Big(\sup_{t\in[0,T]}\Big|\Big\{i\in W^N_k:\sum_{m\in\mathcal{M}}\big|\hat{x}^N_{i,m,l}(t)-\frac{|E^N_k(U^N_{m,l}(t))|}{|E^N_k(V^N)|}\big|>\varepsilon_1/(8\ell)\Big\}\Big|\geq \varepsilon_2 M(N)/(4\ell K)\Big)\nonumber\\
        &\hspace{1.5cm} + \PP\Big(\sup_{t\in[0,T]}\sum_{m\in\mathcal{M}}\Big|\frac{|E^N_k(U^N_{m,l}(t))|}{|E^N_k(V^N)|}-x^N_{k,m,l}\Big|>\varepsilon_1/(8\ell)\Big)\Big)\nonumber\\
        & \leq \frac{4\ell K}{\varepsilon_2 M(N)}\sum_{0\leq l\leq \ell-1}\E \Big(\sup_{t\in[0,T]}\Big|\Big\{i\in W^N_k:\sum_{m\in\mathcal{M}}|\hat{x}^N_{i,m,l}(t)-\frac{|E^N_k(U^N_{m,l}(t))|}{|E^N_k(V^N)|}|>\varepsilon_1/(8\ell)\Big\}\Big|\Big)\nonumber\\
        &\hspace{1.5cm} +\sum_{0\leq l\leq \ell-1}\PP\Big(\sum_{m\in\mathcal{M}}\sup_{U\in V^N_m}\Big|\frac{|E^N_k(U)|}{|E^N_k(V^N)|}-\frac{v_m}{\delta_k}\frac{|U|}{|V^N_m|}\Big|>\varepsilon_1/(8\ell)\Big)\nonumber\\
        & \leq \frac{4\ell K}{\varepsilon_2 M(N)}\sum_{0\leq l\leq \ell-1}\sum_{m\in\mathcal{M}}\E \Big(\sup_{t\in[0,T]}\Big|\Big\{i\in W^N_k:|\hat{x}^N_{i,m,l}(t)-\frac{|E^N_k(U^N_{m,l}(t))|}{|E^N_k(V^N)|}|>\varepsilon_1/(8M\ell)\Big\}\Big|\Big)\nonumber\\
        &\hspace{1.5cm} +\sum_{0\leq l\leq \ell-1}\sum_{m\in\mathcal{M}}\PP\Big(\sup_{U\in V^N_m}\Big|\frac{|E^N_k(U)|}{|E^N_k(V^N)|}-\frac{v_mp_{k,m}}{\delta_k}\frac{|U|}{|V^N_m|}\Big|>\varepsilon_1/(8M\ell)\Big)\nonumber\\
        & \leq \frac{4\ell K}{\varepsilon_2 M(N)}\sum_{0\leq l\leq \ell-1}\sum_{m\in\mathcal{M}}\sup_{U\in V^N_m}\Big|\Big\{i\in W^N_k:\big|\frac{|\mathcal{N}^N_w(i)\cap U|}{|\mathcal{N}^N_w(i)|}-\frac{|E^N_k(U)|}{|W^N_k(V^N)|}\big|>\varepsilon_1/(8M\ell)\Big\}\Big|\nonumber\\
        &\hspace{1.5cm} +\sum_{0\leq l\leq \ell-1}\sum_{m\in\mathcal{M}}\PP\Big(\sup_{U\in V^N_m}\Big|\frac{|E^N_k(U)|}{|E^N_k(V^N)|}-\frac{v_mp_{k,m}}{\delta_k}\frac{|U|}{|V^N_m|}\Big|>\varepsilon_1/(8M\ell)\Big)\nonumber\\
        & \leq \frac{4\ell^2 K M}{\varepsilon_2 M(N)}\sup_{U\in V^N}\Big|\Big\{i\in W^N_k:\big|\frac{|\mathcal{N}^N_w(i)\cap U|}{|\mathcal{N}^N_w(i)|}-\frac{|E^N_k(U)|}{|W^N_k(V^N)|}\big|>\varepsilon_1/(8M\ell)\Big\}\Big|\nonumber\\
        &\hspace{1.5cm} +\sum_{0\leq l\leq \ell-1}\sum_{m\in\mathcal{M}}\PP\Big(\sup_{U\in V^N_m}\Big|\frac{|E^N_k(U)|}{|E^N_k(V^N)|}-\frac{v_mp_{k,m}}{\delta_k}\frac{|U|}{|V^N_m|}\Big|>\varepsilon_1/(8M\ell)\Big)
\end{align}
By Lemma~\ref{lem:appro-x-kml}, there exists $N_1\in\N_0$ such that for all $N\geq N_1$, 
\begin{equation}
    \sup_{m\in\mathcal{M}}\sup_{U\subseteq V^N_m}\Big|\frac{|E^N_k(U)|}{|E^N_k(V^N)|}-\frac{v_mp_{k,m}}{\delta_k}\frac{|U|}{|V^N_m|}\Big|\leq \varepsilon_1/(8M\ell),
\end{equation}
implying
\begin{equation}\label{eq:using-proportion-1}
    \begin{split}
        &\sum_{0\leq l\leq \ell-1}\PP\Big(\sup_{t\in[0,T]}\Big|\Big\{i\in W^N_k:\sum_{m\in\mathcal{M}}|\hat{x}^N_{i,m,l}(t)-x^N_{k,m,l}(t)|>\varepsilon_1/(4\ell)\Big\}\Big|\geq \varepsilon_2 M(N)/(4\ell K)\Big)\\
        & \leq \frac{4\ell^2 K M}{\varepsilon_2 M(N)}\sup_{U\in V^N}\Big|\Big\{i\in W^N_k:\big|\frac{|\mathcal{N}^N_w(i)\cap U|}{|\mathcal{N}^N_w(i)|}-\frac{|E^N_k(U)|}{|W^N_k(V^N)|}\big|>\varepsilon_1/(8M\ell)\Big\}\Big|
    \end{split}
\end{equation}
\noindent
Similarly, we have that there exists $N_2\in\N_0$ such that $N\geq N_2$,
\allowdisplaybreaks
\begin{align}\label{eq:using-proportion-2}
    & \PP\Big(\sup_{t\in[0,T]}\Big|\Big\{i\in W^N_k:\sum_{m\in\mathcal{M}}\big|\sum_{l\geq \ell}(\hat{x}^N_{i,m,l}(t)-x^N_{k,m,l}(t))\big|>\varepsilon_1/4\Big\}\Big|\geq \varepsilon_2 M(N)/(4K)\Big)\nonumber\\
    & \leq \frac{4K}{\varepsilon_2 M(N)} \sup_{U\in V^N}\Big|\Big\{i\in W^N_k:\big|\frac{|\mathcal{N}^N_w(i)\cap U|}{|\mathcal{N}^N_w(i)|}-\frac{|E^N_k(U)|}{|W^N_k(V^N)|}\big|>\varepsilon_1/4\Big\}\Big|.
\end{align}
By \eqref{eq:prop-5.9-1},\eqref{eq:using-proportion-1} and \eqref{eq:using-proportion-2}, there exists $N_3=\max(N_1,N_2)$ such that for all $N\geq N_3$, 
\begin{equation}
\begin{split}
    & \PP\Big(\sup_{t\in[0,T]}\Big|\Big\{i\in W^N_k:\sum_{m\in\mathcal{M}}\sum_{l\in\N_0}|\hat{x}^N_{i,m,l}(t)-x^N_{k,m,l}(t)|>\varepsilon_1\Big\}\Big|\geq \varepsilon_2 M(N)/K\Big)\\
    & \leq \frac{8\ell^2 KM}{\varepsilon_2 M(N)} \sup_{U\in V^N}\Big|\Big\{i\in W^N_k:\big|\frac{|\mathcal{N}^N_w(i)\cap U|}{|\mathcal{N}^N_w(i)|}-\frac{|E^N_k(U)|}{|W^N_k(V^N)|}\big|>\varepsilon_1/4\Big\}\Big|\\
    &\quad + 2\PP\Big(\sup_{t\in[0,T]}\Big|\Big\{i\in W^N_k:\sum_{m\in\mathcal{M}}\sum_{l\geq \ell}x^N_{k,m,l}(t)>\varepsilon_1/4\Big\}\Big|\geq \varepsilon_2 M(N)/(4K)\Big).
\end{split}
\end{equation}
Fix any $\varepsilon_3>0$
By Definition~\ref{defn:prop-sparse}, there exists $N_4\in \N_0$ such that for all $N\geq N_4$, 
\begin{equation}
    \sup_{U\in V^N}\Big|\Big\{i\in W^N_k:\big|\frac{|\mathcal{N}^N_w(i)\cap U|}{|\mathcal{N}^N_w(i)|}-\frac{|E^N_k(U)|}{|W^N_k(V^N)|}\big|>\varepsilon_1/4\Big\}\Big|\leq \frac{\varepsilon_2 M(N)\varepsilon_3}{16\ell^2 KM}.
\end{equation}
By Lemma~\ref{lem:sup-q-ml-t}, there exists $N_5\in\N_0$ such that for all $N\geq N_5$,
\begin{equation}
    \PP\Big(\sup_{t\in[0,T]}\Big|\Big\{i\in W^N_k:\sum_{m\in\mathcal{M}}\sum_{l\geq \ell}x^N_{k,m,l}(t)>\varepsilon_1/4\Big\}\Big|\geq \varepsilon_2 M(N)/(4K)\Big)\leq \frac{\varepsilon_3}{4}.
\end{equation}
Hence, \begin{equation}
    \PP\Big(\sup_{t\in[0,T]}\Big|\Big\{i\in W^N_k:\sum_{m\in\mathcal{M}}\sum_{l\in\N_0}|\hat{x}^N_{i,m,l}(t)-x^N_{k,m,l}(t)|>\varepsilon_1\Big\}\Big|\geq \varepsilon_2 M(N)/K\Big)\leq \varepsilon_3.
\end{equation}
Since $\varepsilon_3>0$ are arbitrary, then the desired result holds.
\end{proof}

\section{Bound the Mismatch}\label{app:lipschitz-jsq-d}
\begin{proof}
Define a function $F^N_{m,l}(\cdot):\mathcal{S}\rightarrow [0,1]$ as: for $\mathbf{x}=(x_{m,l},m\in\mathcal{M},l\in\N_0)\in\mathcal{S}$,
\begin{equation}
    F^N_{m,l}(\mathbf{x})=\frac{\sum_{r=1}^d\sum_{r_1=1}^{r}\frac{r_1}{r}{ N x_{m,l}\choose r_1}{ N\sum_{\mathcal{M}\setminus \{m\}}x_{m,l}\choose r-r_1} { N\sum_{\mathcal{M}}\sum_{l'\geq l+1}x_{m,l}\choose d-r}}{{N \choose d}}.
\end{equation}
Also, define a function $f_{m,l}(\cdot)$ as: for $\mathbf{x}\in\mathcal{S}$,
\begin{equation}
    f_{m,l}(\mathbf{x})=\sum_{r=1}^d\sum_{r_1=1}^{r}\frac{r_1}{r}\frac{d!}{r_1!(r-r_1)!(d-r)!}\big(x_{m,l}\big)^{r_1}\big(\sum_{\mathcal{M}\setminus\{m\}}x_{m,l}\big)^{r-r_1}\big(\sum_{\mathcal{M}}\sum_{l'\geq l+1}x_{m,l'}\big)^{d-r}
\end{equation}
Note that for any $0\leq y\leq x\leq 1$ and $1\leq k\leq d$,  $x^k-(x-y)^k\leq kxy\leq ky$. Then, we have
\allowdisplaybreaks
\begin{align}\label{eq:FN-f}
    &\sum_{m\in\mathcal{M}}\sum_{l\in\N_0}|F^N_{m,l}(\mathbf{x})-f_{m,l}(\mathbf{x})|\nonumber\\
    &\leq \sum_{m\in\mathcal{M}}\sum_{l\in\N_0}\sum_{r=1}^d\sum_{r_1=1}^{r}\frac{r_1}{r}\frac{d!}{r_1!(r-r_1)!(d-r)!}\Big(\big(x_{m,l}\big)^{r_1}\big(\sum_{\mathcal{M}\setminus\{m\}}x_{m,l}\big)^{r-r_1}\big(\sum_{\mathcal{M}}\sum_{l'\geq l+1}x_{m,l'}\big)^{d-r}\nonumber\\
    &\hspace{3.5cm}-\big(x_{m,l}-\frac{r_1}{N}\big)^{r_1}\big(\sum_{\mathcal{M}\setminus\{m\}}x_{m,l}-\frac{r-r_1}{N}\big)^{r-r_1}\big(\sum_{\mathcal{M}}\sum_{l'\geq l+1}x_{m,l'}-\frac{d-r}{n}\big)^{d-r}\Big)\nonumber\\
    &\leq \sum_{m\in\mathcal{M}}\sum_{l\in\N_0}\sum_{r=1}^d\sum_{r_1=1}^{r}\frac{r_1}{r}\frac{d! r_1(r-r_1)(d-r)}{r_1!(r-r_1)!(d-r)!}x_{m,l}(\sum_{\mathcal{M}\setminus\{m\}}x_{m,l})(\sum_{\mathcal{M}}\sum_{l'\geq l+1}x_{m,l'})\big(\frac{d}{N}\big)^d\nonumber\\
    &\leq \sum_{m\in\mathcal{M}}\sum_{l\in\N_0}\sum_{r=1}^d\sum_{r_1=1}^{r}\frac{r_1}{r}\frac{d! r_1(r-r_1)(d-r)}{r_1!(r-r_1)!(d-r)!}x_{m,l}\big(\frac{d}{N}\big)^d\nonumber\\
    &= \sum_{r=1}^d\sum_{r_1=1}^{r}\frac{r_1}{r}\frac{d! r_1(r-r_1)(d-r)}{r_1!(r-r_1)!(d-r)!}\big(\frac{d}{N}\big)^d\rightarrow0\text{ as }0.
\end{align}
Let $\hat{\mathbf{x}}^N_i=(\hat{x}^N_{i,m,l},m\in\mathcal{M},l\in\N_0)$ and $\mathbf{x}'^N_k=(x'^N_{k,m,l},m\in\mathcal{M},l\in\N_0)$. By \eqref{eq:p-ml} and \eqref{eq:p-'-ml}, $p^N_{m,l}(i)=F^N_{m,l}(\hat{\mathbf{x}}^N_i)$ and $p'^N_{m,l}(k)=F^N_{m,l}(\mathbf{x}'^N_k)$.
By the Optimal Coupling, we have 
\begin{align}\label{eq:mismatch-1}
    \PP(\textit{Mismatch})&\leq \sum_{m\in\mathcal{M}}\sum_{l\in\N_0}|F^N_{m,l}(\hat{\mathbf{x}}^N_i)-F^N_{m,l}(\mathbf{x}'^N_k)|\nonumber\\
    &\leq \sum_{m\in\mathcal{M}}\sum_{l\in\N_0}|F^N_{m,l}(\hat{\mathbf{x}}^N_i)-f_{m,l}(\hat{\mathbf{x}}^N_i)|+\sum_{m\in\mathcal{M}}\sum_{l\in\N_0}|F^N_{m,l}(\mathbf{x}'^N_k)-f_{m,l}(\mathbf{x}'^N_k)|\nonumber\\
    &\quad +\sum_{m\in\mathcal{M}}\sum_{l\in\N_0}|f_{m,l}(\hat{\mathbf{x}}^N_i)-f_{m,l}(\mathbf{x}'^N_k)|
\end{align}
Next, we are going to show that $f(\cdot)$ is Lipschitz continuous for $\mathbf{x}\in\mathcal{S}$.
\allowdisplaybreaks
\begin{align}\label{eq:mistmatch-2}
    &\sum_{m\in\mathcal{M}}\sum_{l\in\N_0}|f_{m,l}(\hat{\mathbf{x}}^N_i)-f_{m,l}(\mathbf{x}'^N_k)|\nonumber\\
    &\leq \sum_{m\in\mathcal{M}}\sum_{l\in\N_0}\sum_{r=1}^d\sum_{r_1=1}^{r}\frac{r_1}{r}\frac{d!}{r_1!(r-r_1)!(d-r)!}\Big|\big(\hat{x}^N_{i,m,l}\big)^{r_1}\big(\sum_{\mathcal{M}\setminus\{m\}}\hat{x}^N_{i,m,l}\big)^{r-r_1}\big(\sum_{\mathcal{M}}\sum_{l'\geq l+1}\hat{x}^N_{i,m,l'}\big)^{d-r}\nonumber\\
    &\hspace{3.5cm}-\big(x'^N_{k,m,l}\big)^{r_1}\big(\sum_{\mathcal{M}\setminus\{m\}}x'^N_{k,m,l}\big)^{r-r_1}\big(\sum_{\mathcal{M}}\sum_{l'\geq l+1}x'^N_{k,m,l'}\big)^{d-r}\Big|\nonumber\\
    &\leq  \sum_{m\in\mathcal{M}}\sum_{l\in\N_0}\sum_{r=1}^d\sum_{r_1=1}^{r}\frac{r_1}{r}\frac{d!r_1(r-r_1)(d-r)}{r_1!(r-r_1)!(d-r)!}\Big|\big(\hat{x}^N_{i,m,l}-x'^N_{k,m,l}\big)\nonumber\\
    &\hspace{3.5cm}\big(\sum_{\mathcal{M}\setminus\{m\}}\hat{x}^N_{i,m,l}-\sum_{\mathcal{M}\setminus\{m\}}x'^N_{k,m,l}\big)\big(\sum_{\mathcal{M}}\sum_{l'\geq l+1}\hat{x}^N_{i,m,l'}-\sum_{\mathcal{M}}\sum_{l'\geq l+1}x'^N_{k,m,l'}\big)\Big|\nonumber\\
    &\leq  \sum_{m\in\mathcal{M}}\sum_{l\in\N_0}\sum_{r=1}^d\sum_{r_1=1}^{r}\frac{r_1}{r}\frac{d!r_1(r-r_1)(d-r)}{r_1!(r-r_1)!(d-r)!}\Big|\big(\hat{x}^N_{i,m,l}-x'^N_{k,m,l}\big)\Big|\nonumber\\
    =& \sum_{r=1}^d\sum_{r_1=1}^{r}\frac{r_1}{r}\frac{d!r_1(r-r_1)(d-r)}{r_1!(r-r_1)!(d-r)!}\norm{\hat{\mathbf{x}}^N_{i}-\mathbf{x}^N_{k}}_1.
\end{align}
Let $L=2\sum_{r=1}^d\sum_{r_1=1}^{r}\frac{r_1}{r}\frac{d!r_1(r-r_1)(d-r)}{r_1!(r-r_1)!(d-r)!}$. By \eqref{eq:FN-f}, \eqref{eq:mismatch-1} and \eqref{eq:mistmatch-2}, we have that for large enough $N$,
\begin{equation}
    \PP(\textit{Mismatch})\leq L \norm{\hat{\mathbf{x}}^N_{i}-\mathbf{x}^N_{k}}_1.
\end{equation}

\end{proof}

\section{Doubly Exponential Decay}\label{app:proof-doubly-decay}

\begin{proof}[Proof of Proposition~\ref{prop:dbly-decay}]
Since $\mathbf{q}$ is a fixed point of \eqref{eq:ode-sys-2}, then we have 
\begin{equation*}
   u_m(q_{m,l}-q_{m,l+1})=\lambda\xi(q_{m,l-1}-q_{m,l})\sum_{k\in\mathcal{K}}\frac{p_{k,m}w_k}{\delta_k}\frac{(\tilde{q}_{k,l-1})^d-(\tilde{q}_{k,l})^d}{\tilde{q}_{k,l-1}-\tilde{q}_{k,l}}.
\end{equation*}
Multiplying both sides by $v_m$ and summing over $m \in \mathcal{M}$ gives,
\begin{equation}
    \sum_{m\in\mathcal{M}}v_mu_m(q_{m,l}-q_{m,l+1})=\lambda\xi\sum_{k\in\mathcal{K}}w_k\big((\tilde{q}_{k,l-1})^d-(\tilde{q}_{k,l})^d\big).
\end{equation}
Also, since $q_{m,l}\xrightarrow{l\rightarrow\infty}0$, $\forall m\in\mathcal{M}$, then for $\ell\geq 1$, by adding $l\geq \ell$, we have
\begin{equation}\label{eq:sum-qml}
    \sum_{m\in\mathcal{M}}v_mu_mq_{m,\ell}=\lambda\xi \sum_{k\in\mathcal{K}}w_k\big(\tilde{q}_{k,\ell-1}\big)^d.
\end{equation}
From \eqref{eq:sum-qml} and $\sum_{k\in\mathcal{K}}w_k=1$, we have 
$$ \sum_{m\in\mathcal{M}}v_mu_mq_{m,\ell}\leq \lambda\xi (\tilde{q}_{\ell-1}^*)^d$$
where $\tilde{q}_{\ell-1}^*=\max_{k\in\mathcal{K}}\tilde{q}_{k,\ell-1}$. Hence, for all $m\in\mathcal{M}$, 
$$q_{m,\ell}\leq \frac{\lambda\xi}{v_mu_m}(\tilde{q}_{\ell-1}^*)^d\leq c^*(m,\ell-1)\tilde{q}_{\ell-1}^*,$$ 
where $c^*(m,\ell-1)=(\tilde{q}_{\ell-1}^*)^{d-1}\max_{m\in\mathcal{M}}\lambda\xi/(v_mu_m)$. Since we assume that $q_{m,\ell}\xrightarrow{\ell\rightarrow\infty}0$ for all $m\in\mathcal{M}$, then we can choose a large enough $\ell$ such that $c^*(m,\ell-1)<1$. By definition, for each $k\in\mathcal{K}$,
$$\tilde{q}_{k,\ell}=\sum_{m\in\mathcal{M}}\frac{v_mp_{k,m}}{\delta_k}q_{m,\ell}\leq c^*(m,\ell-1)(\tilde{q}^*_{\ell-1})^{d-1}$$ 
which implies that 
$\tilde{q}_{\ell}^*\leq c^*(m,\ell-1)\tilde{q}^*_{\ell-1}$ and 
$$q_{m,\ell+1}\leq \frac{\lambda\xi}{v_mu_m}(\tilde{q}^*_{\ell})^d\leq (c^*(m,\ell-1)\tilde{q}_{\ell-1}^*)^d\max_{m\in\mathcal{M}}\lambda\xi/(v_mu_m)=(c^*(m,\ell-1))^{d+1}\tilde{q}_{\ell-1}^*.$$ By induction, we obtain that for $n\in\N_0$, 
\begin{equation}\label{eq:dbly-expo}
    q_{m,\ell+n}\leq (c^*(m,\ell-1))^{e(n)}\tilde{q}^*_{\ell-1}\leq (c^*(m,\ell-1))^{d^n}\tilde{q}^*_{\ell-1}
\end{equation}
where $e(n)=\sum_{i=0}^nd^{i}$. \eqref{eq:dbly-expo} implies that $\{q_{m,l},l\in\N_0\}$ decreases doubly exponentially.
\end{proof}

\begin{remark}
Recall $\tilde{q}_{k,l}=\sum_{m\in\mathcal{M}}\frac{v_mp_{k,m}}{\delta_k}q_{m,l}$. From Proposition~\ref{prop:dbly-decay}, we know $\{\tilde{q}_{k,l},l\in\N_0\}$ decreases doubly exponentially. 
In fact, they do not decay further faster.
To see this, let $c_0 = \min_{k \in \mathcal{K}} \min_{m \in \mathcal{M}} \frac{p_{k,m}}{\delta_k} \in (0,1].$
Then
$\tilde{q}_{k,l}=\sum_{m=1}^M \frac{v_{m}p_{k,m}}{\delta_k}q_{m,l} \ge c_0 \sum_{m\in\mathcal{M}}v_mq_{m,l}.$
It then follows from \eqref{eq:sum-qml} that
\begin{equation*}
 \min_{k \in \mathcal{K}} \tilde{q}_{k,\ell} \ge c_0 \sum_{m\in\mathcal{M}}v_mq_{m,\ell} = \lambda c_0 \sum_{k\in\mathcal{K}}w_k\big(\tilde{q}_{k,\ell-1}\big)^d \ge \lambda c_0 (\min_{k \in \mathcal{K}} \tilde{q}_{k,\ell-1})^d.
\end{equation*}
So
$$(\lambda c_0)^{\frac{1}{d-1}} \min_{k \in \mathcal{K}} \tilde{q}_{k,\ell} \ge ((\lambda c_0)^{\frac{1}{d-1}} \min_{k \in \mathcal{K}} \tilde{q}_{k,\ell-1})^d \ge \dotsb \ge ((\lambda c_0)^{\frac{1}{d-1}} \min_{k \in \mathcal{K}} \tilde{q}_{k,0})^{d^\ell}$$
and hence
$\min_{k \in \mathcal{K}} \tilde{q}_{k,\ell} \ge (\lambda c_0)^{\frac{d^\ell-1}{d-1}}.$
\end{remark}

\section{Proof of Lemma~\ref{lem:alpha-m}}\label{app:proof-lem-alpha-m}
\begin{proof}[Proof of Lemma~\ref{lem:alpha-m}]
Fix any $(\alpha_1,...,\alpha_M)\in(0,1)^M$ with $\sum_{m\in\mathcal{M}}\alpha_m>0$. Consider any sequence $\{U^N\}_N$ of subsets with $U^N\subseteq V^N$ and $\lim_{N\rightarrow\infty}\frac{|U^N\cap V^N_m|}{|V^N_m|}=\alpha_m$ for all $m\in\mathcal{M}$. By Condition~\ref{cond-1}, we have that for all $k\in\mathcal{K}$ and $m\in\mathcal{M}$,
\begin{equation}\label{eq:edge-density}
    \lim_{N\rightarrow\infty}\frac{|E^N_k(U^N\cap V^N_m)|}{|E^N_k(V^N_m)|}=\alpha_m v_m.
\end{equation}
Fix any $\varepsilon>0$ which will be chosen later. Let $\mathcal{G}^N_{k,\varepsilon}=\Big\{i\in W^N_k:\Big|\frac{|\mathcal{N}^N_w(i)\cap v|}{|\mathcal{N}^N_w(i)|}-\frac{|E^N_k(v)|}{|E^N_k(V^N)|}\Big|\geq \varepsilon\Big\}$ and $\mathcal{B}^N_{k,\varepsilon}=W^N_k\setminus \mathcal{G}^N_{k,\varepsilon}$.
By \eqref{eq:edge-density}, for all large enough $N$ and $i\in \mathcal{G}^N_{k,\varepsilon}$,
\begin{equation}
    N (1-2\varepsilon)\sum_{m\in\mathcal{M}}\alpha_m v_m p_{k,m}\leq |\mathcal{N}^N_w(i)\cap U^N|\leq N (1+2\varepsilon)\sum_{m\in\mathcal{M}}\alpha_m v_m p_{k,m}.
\end{equation}
Also, by Condition~\ref{cond-1}, for all large enough $N$, 
\begin{equation}
    N \delta_k  (1-\varepsilon) \leq \delta^N_i\leq  N \delta_k(1+\varepsilon).
\end{equation}
Since the sequence $\{G^N\}_N$ is in the subcritical, then for large enough $N$,
\begin{equation}
    \begin{split}
        \rho\geq \rho^N\geq &\Big(\sum_{j\in U^N}\sum_{m\in\mathcal{M}'}\mathds{1}_{(j\in V^N_m)}u_m\Big)^{-1}\sum_{i\in W^N}\sum_{\substack{S\subseteq(U^N\cap \mathcal{N}^N_w(i)):\\|S|=d}}\frac{\lambda}{{|\mathcal{N}^N_w(i)|\choose d}}\\
        \geq &c'(N)\Big(\sum_{m\in\mathcal{M}'} Nv_m\alpha_m  u_m \Big)^{-1}\sum_{k\in\mathcal{K}}\sum_{i\in W^N_k}\frac{\lambda {|U^N\cap \mathcal{N}^N_w(i)|\choose d} }{{|\mathcal{N}^N_w(i)|\choose d}}\\
        \geq & \Big(\sum_{m\in\mathcal{M}'} Nv_m\alpha_m u_m \Big)^{-1}\sum_{k\in\mathcal{K}}\frac{\lambda |\mathcal{G}^N_{k,\varepsilon}| {N(1-2\varepsilon)\sum_{m\in\mathcal{M}}\alpha_m v_m p_{k,m}\choose d} }{{N\delta_k(1+\varepsilon)\choose d}},
    \end{split}
\end{equation}
where $c'(N)$ is a constant only depending on $N$ with $c'(N)\xrightarrow{N\rightarrow\infty}1$.
Since the sequence $\{G^N\}$ is proportionally sparse, then $\lim_{N\rightarrow\infty}\frac{|\mathcal{G}^N_{k,\varepsilon}|}{|W^N_k|}=1$. Then, we have 
\begin{equation}\label{eq:alpha-m-varepsilon}
    \rho\geq \Big(\sum_{m\in\mathcal{M}}v_m\alpha_m p_{k,m}\Big)^{-1}\lambda \xi\sum_{m\in\mathcal{M}}w_k\Big(\frac{(1-2\varepsilon)\sum_{m\in\mathcal{M}}\alpha_mv_mp_{k,m}}{\delta_k(1+\varepsilon)}\Big)^d.
\end{equation}
Since \eqref{eq:alpha-m-varepsilon} holds for all $\varepsilon>0$, then 
\begin{equation}
    \rho\geq \Big(\sum_{m\in\mathcal{M}}v_m\alpha_m p_{k,m}\Big)^{-1}\lambda \xi\sum_{m\in\mathcal{M}}w_k\Big(\frac{\sum_{m\in\mathcal{M}}\alpha_mv_mp_{k,m}}{\delta_k}\Big)^d.
\end{equation}

\end{proof}

\section{Proof of Lemma~\ref{lem:tail-bound-expectation-steady}}\label{app:proof-tail-bound-expectation}

\begin{proof}[Proof of Lemma~\ref{lem:tail-bound-expectation-steady}]
Given the system state $X^N$. When a task arrives at the system, by the Poisson thinning property, the probability that the task will be assigned to a server in the set $Q^N_{m,l}(X^N)$ is 
\begin{equation}
    \PP(\mathcal{E}(Q^N_{m,l}))=\frac{1}{W(N)}\sum_{i\in W^N}\sum_{\substack{U\subseteq(Q^N_{m,l}\cap \mathcal{N}^N_w(i))\\ |U|=d}}\frac{1}{{\mathcal{N}^N_w(i)\choose d}}
\end{equation}
where $\mathcal{E}(Q^N_{m,l})\coloneqq\text{ the event that the new task will be assigned to }Q^N_{m,l}(X^N)$. Fix any $\varepsilon>0$. Since the sequence $\{G^N\}$ is subcritical, then for large enough $N$, we have that
\begin{equation}\begin{split}
    \PP(\mathcal{E}(Q^N_{m,l}))&\leq  \frac{N}{W(N)}\frac{\rho^N}{\lambda}\frac{|Q^N_{m,l}(X^N)|u_m}{N}
     \leq  \frac{\rho}{\lambda \xi}q^N_{m,l}u_m(1+\varepsilon).
\end{split}
\end{equation}
We consider the system state at event times $t_0=0<t_1<t_2<...<t_i<...$; for all $i$, $t_i$ can be an arrival or a potential departure epoch.
Define the drift $\Delta L^N_{m,\ell}(X^N)$ as 
\begin{equation}
    \Delta L^N_{m,\ell}(X^N)=\E\Big(L^N_{m,\ell}(X^N(t_1))-L^N_{m,\ell}(X^N)|X^N(t_0)=X^N\Big).
\end{equation}
Again, by the Poisson thinning property, we have that for all large $N$,
\begin{align}\label{eq:drift-L}
    \Delta L^N_{m,\ell}(X^N)&=\sum_{i=\ell}^{\infty}\Big(\frac{\lambda W(N)}{\lambda W(N)+\sum_{m\in\mathcal{M}}|V^N_m|u_m}\PP(\mathcal{E}(Q^N_{m,i-1}))-\frac{\sum_{m\in\mathcal{M}}|V^N_m|u_m}{\lambda W(N)+\sum_{m\in\mathcal{M}}|V^N_m|u_m}\frac{|Q^N_{m,i}|u_m}{\sum_{m\in\mathcal{M}}|V^N_m|u_m}\Big)\nonumber\\
    & \leq  \sum_{i=\ell}^{\infty}\Big(\frac{\rho q^N_{m,i-1}u_m(1+\varepsilon)}{\lambda\xi+\sum_{m\in\mathcal{M}}v_mu_m}-\frac{q^N_{m,i}u_m}{\lambda\xi+\sum_{m\in\mathcal{M}}v_mu_m}\Big)\nonumber\\
    &=\frac{\rho q^N_{m,\ell-1}u_m(1+\varepsilon)}{\lambda\xi+\sum_{m\in\mathcal{M}}v_mu_m}-\frac{1-(1+\varepsilon)\rho}{\lambda\xi+\sum_{m\in\mathcal{M}}v_mu_m}\sum_{i=\ell}^{\infty}q^N_{m,i}u_m
    \end{align}
By the definition of the steady state, $\E(\Delta L^N_{m,\ell}X^N(\infty))=0$. Choosing $\varepsilon$ such that $(1+\varepsilon)\rho\leq (1+\rho)/2<1$, we have 
\begin{equation}
    \sum_{i=\ell}^{\infty}\E(q^N_{m,i}(\infty))\leq \frac{(1+\rho)/2}{1-(1+\rho)/2}  \E(q^N_{m,\ell-1}).
\end{equation}
Finally, summing over $m\in\mathcal{M}$, we get the desired result.
\end{proof}

\section{Proof for the Sequence of Random Graphs}\label{app:random-graph}

\begin{proof}[Proof of Proposition~\ref{prop:random-graph}]
First to show that the sequence $\{G^N\}_N$ satisfies Condition~\ref{cond-1}.
Consider any fixed $k\in\mathcal{K}$ and $m\in\mathcal{M}$.
Let $e_{i,j}$ be a Bernoulli random variable with probability $p_{k,m}$ for each $i\in W^N_k$ and $j\in V^N_m$. Then, $E^N(k,m)=\sum_{(i,j)\in W^N_k\times V^N_m}e_{i,j}$, and by the L.L.N., we have that $$\lim_{N\rightarrow\infty}\frac{E^N(k,m)}{|W^N_k|\times |V^N_m|}=p_{k,m},$$
which implies that Condition~\ref{cond-1} $(a)$ holds.
Next, we prove that Condition~\ref{cond-1} $(b)$ holds. Based on the definition $\deg_w^N(i)$, we have $\deg_w^N(i)=\sum_{j\in V^N_m}e_{i,j}$ which is a binomial random variable $\mathrm{Binomial}(|V^N_m|,p_{k,m})$. By the Chernoff bound (\cite[Theorem~2.4]{CY05}), it follows that for $i\in W^N_k$,
\begin{equation*}
    \PP\Big(\big|\deg^N_w(i)-\E(\deg^N_w(i))\big|\geq x\Big)\leq 2\exp\Big(-\frac{x^2}{2\E(\deg^N_w(i))+2x/3}\Big).
\end{equation*}
Let $X(N)=p_{k,m}N^{3/4}\big(\ln (N)\big)^{1/4}$. Then, for some $c_1\in(0,\infty)$, 
\begin{equation}\label{eq:A-1}
    \begin{split}
        \PP\Big(\big|\deg^N_w(i)-|V^N_m|p_{k,m}\big|\geq X(N)\Big)& \leq  c_1\exp\Big(-c_1 p_{k,m}N^{1/2}(\ln (N))^{1/2}/v_m\Big),
    \end{split}
\end{equation}
for sufficiently large $N$. Also by $\lim_{N\rightarrow\infty}\frac{W^N_k}{W(N)}=w_k$, $\lim_{N\rightarrow\infty}\frac{W(N)}{N}=\xi$, and the union bound, we have that there exists $c_2\in(0,\infty)$ such that for large enough $N$, 
\begin{equation}\label{eq:A-2}
    \PP\Big(\cup_{i\in W^N_k}\big|\deg^N_w(i)-|V^N_m|p_{k,m}\big|\geq X(N)\Big)\leq c_2 w_k \xi N \exp\Big(-c_1 p_{k,m}N^{1/2}(\ln (N))^{1/2}/v_m\Big)
\end{equation}
Then, the RHS of \eqref{eq:A-2} is summable over $N$. From the Borel-Cantelli lemma, we get that a.s., for all large enough $N$, 
$$\big|\deg^N_w(i)-|V^N_m|p_{k,m}\big|\leq  X(N),\quad i\in W^N_k,$$
which implies that the following equation holds
$$1\leq \lim_{N\rightarrow\infty}\frac{\max_{i\in W^N_k}\deg^N_w(i)}{\min_{i\in W^N_k} \deg^N_w(i)}\leq \lim_{N\rightarrow\infty}\frac{|V^N_m|p_{k,m}+X(N)}{|V^N_m|p_{k,m}-X(N)}=1,\quad \text{a.s..}$$
Thus, Condition~\ref{cond-1} $(b)$ holds. 

Now, we show that the sequence $\{G^N\}_N$ is clustered proportionally sparse. Fix any $k\in\mathcal{K}$, $i\in W^N_k$, $\varepsilon>0$, and $U\subseteq V^N$. Let $B_i(U)$ be the event that the dispatcher $i$ is bad w.r.t. the set $U$, i.e., 
\begin{equation}
    B_i(U)\coloneqq\Big\{\Big|\frac{\mathcal{N}^N_w(i)\cap U}{\mathcal{N}^N_w(i)}-\frac{E^N_k(U)}{E^N_k(V^N)} \Big|\geq \varepsilon \Big\} .
\end{equation}
Define $\alpha_m\coloneqq\frac{|U\cap V^N_m|}{|V^N|}$ for each $m\in\mathcal{M}$. By the union bound, we have that 
\begin{equation}\label{eq:A-3}
    \begin{split}
        \PP\big(B_i(U)\big)& \leq  \PP\Big(B_i(U),\ \Big||\mathcal{N}^N_w(i)\cap U|-\sum_{m\in\mathcal{M}}|V^N_m\cap U|p_{k,m}\Big|<\varepsilon_1 \sum_{m\in\mathcal{M}}|V^N_m|p_{k,m},\\
        &\qquad \big|\frac{E^N_k(U)}{E^N_k(V^N)}-\frac{\sum_{m\in\mathcal{M}}\alpha_m p_{k,m}}{\sum_{v_m p_{k,m}}}\big|<\varepsilon_2,
        \text{ and }\Big|\mathcal{N}^N_w(i)-\sum_{m\in\mathcal{M}}|V^N_m|p_{k,m}\Big|<\varepsilon_3\sum_{m\in\mathcal{M}}|V^N_m|p_{k,m}\Big)\\
        & +\PP\Big(\Big||\mathcal{N}^N_w(i)\cap U|-\sum_{m\in\mathcal{M}}|V^N_m\cap U|p_{k,m}\Big|\geq \varepsilon_1 \sum_{m\in\mathcal{M}}|V^N_m|p_{k,m}\Big)\\
        &+\PP\Big(\Big|\mathcal{N}^N_w(i)-\sum_{m\in\mathcal{M}}|V^N_m|p_{k,m}\Big|\geq \varepsilon_2\sum_{m\in\mathcal{M}}|V^N_m|p_{k,m}\Big)
    \end{split}
\end{equation}
We will bound each term of the RHS of \eqref{eq:A-3}. By choosing 
$\varepsilon_1$, $\varepsilon_2$ and $\varepsilon_3$ satisfying 
\begin{equation}\label{eq:A-4}
    \begin{split}
        \frac{\varepsilon_3\sum_{m\in\mathcal{M}}\alpha_m p_{k,m}+\varepsilon_1\sum_{m\in\mathcal{M}}v_mp_{k,m}}{(1-\varepsilon_3)\sum_{m\in\mathcal{M}}v_mp_{k,m}}+\varepsilon_2<\varepsilon,
    \end{split}
\end{equation}
we have that \begin{equation}\label{eq:A-5}
    \begin{split}
        \frac{\mathcal{N}^N_w(i)\cap U}{\mathcal{N}^N_w(i)}-\frac{E^N_k(U)}{E^N_k(V^N)} &= \frac{\mathcal{N}^N_w(i)\cap U}{\mathcal{N}^N_w(i)}-\frac{\sum_{m\in\mathcal{M}}\alpha_m p_{k,m}}{\sum_{v_m p_{k,m}}}+\frac{\sum_{m\in\mathcal{M}}\alpha_m p_{k,m}}{\sum_{v_m p_{k,m}}}-\frac{E^N_k(U)}{E^N_k(V^N)}\\
        <&\frac{\varepsilon_3\sum_{m\in\mathcal{M}}\alpha_m p_{k,m}+\varepsilon_1\sum_{m\in\mathcal{M}}v_mp_{k,m}}{(1-\varepsilon_3)\sum_{m\in\mathcal{M}}v_mp_{k,m}}+\varepsilon_2<\varepsilon,
    \end{split}
\end{equation}
and 
\begin{equation}\label{eq:A-6}
    \begin{split}
        \frac{\mathcal{N}^N_w(i)\cap U}{\mathcal{N}^N_w(i)}-\frac{E^N_k(U)}{E^N_k(V^N)} &= \frac{\mathcal{N}^N_w(i)\cap U}{\mathcal{N}^N_w(i)}-\frac{\sum_{m\in\mathcal{M}}\alpha_m p_{k,m}}{\sum_{v_m p_{k,m}}}+\frac{\sum_{m\in\mathcal{M}}\alpha_m p_{k,m}}{\sum_{v_m p_{k,m}}}-\frac{E^N_k(U)}{E^N_k(V^N)}\\
        >&-\frac{\varepsilon_3\sum_{m\in\mathcal{M}}\alpha_m p_{k,m}+\varepsilon_1\sum_{m\in\mathcal{M}}v_mp_{k,m}}{(1+\varepsilon_3)\sum_{m\in\mathcal{M}}v_mp_{k,m}}-\varepsilon_2>-\varepsilon,
    \end{split}
\end{equation}
which implies that the first term is equal to 0 with $\varepsilon_1$, $\varepsilon_2$ and $\varepsilon_3$.
Using the Chernoff bound again, we can bound the second term and the third term as follows: for some $c_3\in(0,\infty)$ and large enough $N$,
\begin{equation}\label{eq:A-7}
        \PP\Big(\Big||\mathcal{N}^N_w(i)\cap U|-\sum_{m\in\mathcal{M}}|V^N_m\cap U|p_{k,m}\Big|\geq \varepsilon_1 \sum_{m\in\mathcal{M}}|V^N_m|p_{k,m}\Big)\\
        \leq  c_3\exp\Big(-c_3N\sum_{m\in\mathcal{M}}v_mp_{k,m}\Big),
\end{equation}
and 
\begin{equation}\label{eq:A-8}
    \begin{split}
        &\PP\Big(\Big|\mathcal{N}^N_w(i)-\sum_{m\in\mathcal{M}}|V^N_m|p_{k,m}\Big|\geq \varepsilon_2\sum_{m\in\mathcal{M}}|V^N_m|p_{k,m}\Big)
         \leq  c_3\exp\Big(-c_3N\sum_{m\in\mathcal{M}}v_mp_{k,m}\Big).
    \end{split}
\end{equation}
Therefore, for large enough $N$, we have 
\begin{equation}
    \PP(B_i(U))\leq 2c_3\exp\Big(-c_3N\sum_{m\in\mathcal{M}}v_mp_{k,m}\Big),
\end{equation}
and 
\begin{equation}
    \PP(\cup_{i\in W^N_k}B_i(U))\leq 2c_3|W^N_k|\exp\Big(-c_3N\sum_{m\in\mathcal{M}}v_mp_{k,m}\Big).
\end{equation}
Moreover, for some $c_4\in(0,\infty)$ and large enough $N$,
\begin{equation}\label{eq:A-9}
    \PP\big(\sup_{U\subseteq V^N}\cup_{i\in W^N_k}B_i(U)\big)\leq \exp(-c_4N).
\end{equation}
The RHS of \eqref{eq:A-9} is summable over $N$ and the set $\mathcal{K}$ is finite, so by the Borel Cantelli lemma, the sequence is clustered proportionally sparse. 

If $\mathbf{p}$ satisfies \eqref{eq:max-rho-set}, by Lemma~\ref{lem:relationship-v-u-pkm}, there exists an $N_0\in\N_0$ such that for all $N\geq N_0$, the queue length process $\big(X_j^N(t)\big)_{j\in V^N}$ under the local JSQ($d$) policy is ergodic, which implies that all assumptions of Theorem~\ref{thm:interchange of limits} hold.
\end{proof}

\end{document}